\documentclass[oneside,english,reqno]{amsart}
\usepackage[T1]{fontenc}
\usepackage[latin9]{inputenc}
\usepackage{color}
\usepackage{babel}
\usepackage{float}
\usepackage{mathrsfs}
\usepackage{amsthm}
\usepackage{amsbsy}
\usepackage{amstext}
\usepackage{amssymb}
\usepackage{graphicx}
\usepackage{esint}
\usepackage[all]{xy}
\PassOptionsToPackage{normalem}{ulem}
\usepackage{ulem}
\usepackage[unicode=true,pdfusetitle,
 bookmarks=true,bookmarksnumbered=false,bookmarksopen=false,
 breaklinks=false,pdfborder={0 0 1},backref=false,colorlinks=false]
 {hyperref}
\hypersetup{
 colorlinks=true,citecolor=blue,linkcolor=blue,linktocpage=true}
\usepackage{breakurl}

\makeatletter

\providecommand{\tabularnewline}{\\}
\newcommand{\lyxdot}{.}

\numberwithin{equation}{section}
\numberwithin{figure}{section}
\usepackage{enumitem}		
\theoremstyle{plain}
\newtheorem{thm}{\protect\theoremname}[section]
  \theoremstyle{plain}
  \newtheorem{lem}[thm]{\protect\lemmaname}
  \theoremstyle{remark}
  \newtheorem{rem}[thm]{\protect\remarkname}
  \theoremstyle{plain}
  \newtheorem{prop}[thm]{\protect\propositionname}
  \theoremstyle{plain}
  \newtheorem{cor}[thm]{\protect\corollaryname}
  \theoremstyle{definition}
  \newtheorem{example}[thm]{\protect\examplename}

\providecommand{\MR}[1]{}

\makeatother

  \providecommand{\corollaryname}{Corollary}
  \providecommand{\examplename}{Example}
  \providecommand{\lemmaname}{Lemma}
  \providecommand{\propositionname}{Proposition}
  \providecommand{\remarkname}{Remark}
\providecommand{\theoremname}{Theorem}

\begin{document}

\title{Translation representations and scattering by two intervals}

\author{Palle Jorgensen, Steen Pedersen, and Feng Tian}

\address{(Palle E.T. Jorgensen) Department of Mathematics, The University
of Iowa, Iowa City, IA 52242-1419, U.S.A. }

\email{jorgen@math.uiowa.edu }

\urladdr{http://www.math.uiowa.edu/\textasciitilde{}jorgen/}

\address{(Steen Pedersen) Department of Mathematics, Wright State University,
Dayton, OH 45435, U.S.A. }

\email{steen@math.wright.edu }

\urladdr{http://www.wright.edu/\textasciitilde{}steen.pedersen/}

\address{(Feng Tian) Department of Mathematics, Wright State University, Dayton,
OH 45435, U.S.A. }

\email{feng.tian@wright.edu }

\urladdr{http://www.wright.edu/\textasciitilde{}feng.tian/}


\maketitle
\begin{center}
To the memory of William B. Arveson.
\par\end{center}
\begin{abstract}
Studying unitary one-parameter groups in Hilbert space $(U(t),\mathscr{H})$,
we show that a model for obstacle scattering can be built, up to unitary
equivalence, with the use of translation representations for $L^{2}$-functions
in the complement of two finite and disjoint intervals. 

The model encompasses a family of systems $(U(t),\mathscr{H})$. For
each, we obtain a detailed spectral representation, and we compute
the scattering operator, and scattering matrix. We illustrate our
results in the Lax-Phillips model where $(U(t),\mathscr{H})$ represents
an acoustic wave equation in an exterior domain; and in quantum tunneling
for dynamics of quantum states. 
\end{abstract}
\tableofcontents{}

\section{Introduction\label{sec:Intro}}

For a number of problems in analysis, one is faced with a unitary
one-parameter group acting on a Hilbert space. In such a problem,
if an energy form is preserved, this allows one to create a Hilbert
space $\mathscr{H}$, and then to study how states change, as a function
of time, via a one-parameter group of unitary operators $U(t)$ acting
in $\mathscr{H}$. Here $t$ is representing time.

For the study of scattering theory, Lax and Phillips suggested in
\cite{LP68} that one looks for two unitarily equivalent versions
of $U(t)$. For the acoustic wave equation, for example, with scattering
around a finite obstacle, Lax and Phillips proved that, in each of
these two representations, the equivalent unitary one-parameter group
may be taken to be a \textquotedblleft{}copy of\textquotedblright{}
the group of translations of $L^{2}$-functions on the real line $\mathbb{R}$,
but functions taking values in a fixed Hilbert space $\mathscr{M}$.
By \textquotedblleft{}a copy\textquotedblright{} we mean a one-parameter
group which in unitarily equivalent to $U(t)$. As a result, one gets
two isometric transforms from $\mathscr{H}$ into $L^{2}(\mathbb{R},\mathscr{M})$. 

The two representations are called \textquotedblleft{}translation
representations;\textquotedblright{} one incoming, and the other outgoing.
It is known that the same idea is applicable to certain instances
of dynamics of quantum states governed by a Schrödinger equation.
In the Lax-Phillips model, given, as above, a pair $(\mathscr{H},U(t))$,
Hilbert space, and unitary one-parameter group, one looks for two
closed subspaces $D_{in}$ (incoming states) and $D_{out}$ (outgoing
states) in $\mathscr{H}$. Incoming refers to \textquotedblleft{}before
the obstacle;\textquotedblright{} and outgoing, after. On the incoming
states $f_{in}$, $U(t)$ acts by translation to the left, so acting
before the \textquotedblleft{}obstacle,\textquotedblright{} while
the outgoing states $f_{out}$, $U(t)$ acts by translation to the
right. A scattering operator $S$ will act between the two subspaces,
$Sf_{in}=f_{out}$, sending $f_{in}$ into $f_{out}$. 

A second source of motivation for our analysis of exterior problems
derives from recent work on exterior dynamical systems; now extensively
studied under the heading \textquotedblleft{}outer billiard,\textquotedblright{}
or dual billiard, or anti-billiard; see e.g., \cite{Sc11,Sc9}. Unlike
billiard \cite{Mo08}, the \textquotedblleft{}outer\textquotedblright{}
game is played outside of the table (a convex domain). The role of
unitary operators in Hilbert space is supported by a theorem of Moser
which asserts that the outer billiard map is area-preserving.

Now, in realistic models, detailed properties of an obstacle are often
difficult to come by, and it is therefore useful to work through some
idealized models for obstacle. In the simplest such models, for example
the complement of two bounded disjoint intervals in $\mathbb{R}$,
one can rephrase the problem in the language of von Neumann\textquoteright{}s
deficiency indices, and deficiency subspaces, see \cite{vNeu32,DS88b}
and Section \ref{sec:P} below.

This is the focus of our present analysis, and it covers such examples
from quantum mechanics as quantum tunneling. Now, as above, fix two
bounded intervals $I_{1}$ and $I_{2}$, and let $\Omega$ denote
the complement, i.e., $\Omega=\mathbb{R}\backslash(I_{1}\cup I_{2})$.
So $\Omega$ has one bounded component, and two unbounded. Since translation
of $L^{2}$-functions is generated by the derivative operator $D$,
it is natural to study $D$ as a skew-symmetric operator with domain
dense in $L^{2}(\Omega)$ consisting of functions $f$ such that $f,Df\in L^{2}(\Omega)$,
and $f$ vanishes on the four boundary points. This is called the
\emph{minimal operator}. The corresponding adjoint operator is the
maximal one; see Remark \ref{rmk:Pmin} and \cite{FeJoPed1,DS88b}.

A degenerate instance of this is when $\Omega$ is instead the complement
of 2 points. In both cases, the minimal operator $D$ will have deficiency
indices $(2,2)$. Using our analysis from \cite{FeJoPed1}, one sees
that we then get all the skew-selfadjoint extensions of $D=\frac{d}{dx}$
indexed by the group $U(2)$ of all unitary $2\times2$ matrices.
This can be done such that, for every $B$ in $U(2)$, we realize
a corresponding skew-selfadjoint boundary conditions (bc-B), and therefore
a unitary one-parameter group $U_{B}(t)$ acting on $L^{2}(\Omega)$.
In our paper we find the scattering theory, as well as the spectral
theory, of each of these unitary one-parameter groups. 

For each $U_{B}(t)$ we find a system of generalized eigenfunctions.
They are determined by three functions $a_{B}$, $b_{B}$, and $c_{B}$,
one for each of the three connected components of $\Omega$.

\subsection{Overview}

We undertake a systematic study of interconnections between geometry
and spectrum for a family of selfadjoint operator extensions indexed
by two things: by (i) the prescribed configuration of the two intervals,
and by (ii) the von Neumann parameters (see (\ref{eq:vN1})). This
turns out to be subtle, and we show in detail how variations in both
(i) and (ii) translate into explicit spectral properties for the extension
operators. Indeed, for each choice in (i) , i.e., relative length
of the two intervals, we have a Hermitian operator with deficiency
indices $(2,2).$ Our main theme is spectral theory of the corresponding
family of $(2,2)$-selfadjoint extension operators. 

In section \ref{sec:P}, we introduce some tools, reproducing kernels
and von Neumann deficiency indices, for dealing with the main setting:
A direct integral representation of the boundary value problem. The
selfadjoint realizations correspond to a family of unitary one-parameter
groups, each one generated by skew-selfadjoint extension of a minimal
first order differential operator in open and unbounded subset $\Omega$
of $\mathbb{R}$.

A key point here is that the unitary one-parameter groups are parametrized
by one of the unitary matrix groups $U(n)$. Here, the number $n$
is related to $\Omega$ as follows: $\Omega$ has two unbounded components,
and $n-1$ bounded components.

Section \ref{sec:sp} contains detailed computations of spectral data
for unitary one parameter groups $U_{B}(t)$ acting in $L^{2}(\Omega)$,
indexed by $B$ in $U(n)$: An explicit presentation of the generalized
eigenfunction direct-integral presentation. The essential points in
our analysis are revealed in the case $n=2$, and we therefore present
the details for $U(2)$. We show that the measure in the direct integral
decomposition of $U_{B}(t)$ is of the form $\sigma_{B}(d\lambda)=P_{B}(\lambda)d\lambda$
with periodic density, and where, in each period-interval, $P_{B}$
is a Poisson kernel depending on $B$ in $U(2)$. We further find
that the cases of embedded point-spectrum (Dirac combs) arise as a
limit taking place in the group $U(2)$.

Within each section, the results are illustrated with applications
from physics and from harmonic analysis.

Sections \ref{sec:UB} - \ref{sec:points} deal with scattering theory
for the unitary one-parameter groups $U_{B}(t)$. This is presented
in terms of time-delay operators, translation representations, and
Lax-Phillips scattering operators. Closely connected to the scattering
operator is the Lax-Phillips contraction semigroup; it is computed
in section \ref{sec:ss}.

\subsection{Unbounded Operators}

We recall the following fundamental result of von Neumann on extensions
of Hermitian operators.

In order to make precise the boundary form for the cases (\ref{eq:BoundaryForm2})
- (\ref{eq:BoundaryForm3}) we need a:
\begin{lem}
\label{lem:cont}Let $\Omega\subset\mathbb{R}$ be as above. Suppose
$f$ and $f'=\frac{d}{dx}f$ (distribution derivative) are both in
$L^{2}(\Omega)$; then there is a continuous function $\tilde{f}$
on $\overline{\Omega}$ (closure) such that $f=\tilde{f}$ a.e. on
$\Omega$, and $\lim_{\left|x\right|\rightarrow\infty}\tilde{f}(x)=0$. \end{lem}
\begin{proof}
Let $p\in\mathbb{R}$ be a boundary point. Then for all $x\in\Omega$,
we have:
\begin{equation}
f(x)-f(p)=\int_{p}^{x}f'(y)dy.\label{eq:tmp-5}
\end{equation}
Indeed, $f'\in L_{loc}^{1}$ on account of the following Schwarz estimate
\[
\left|f(x)-f(p)\right|\leq\sqrt{\left|x-p\right|}\left\Vert f'\right\Vert _{L^{2}(\Omega)}.
\]
Since the RHS in (\ref{eq:tmp-5}) is well-defined, this serves to
make the LHS also meaningful. Now set
\[
\tilde{f}(x):=f(p)+\int_{p}^{x}f'(y)dy,
\]
and it can readily be checked that $\tilde{f}$ satisfies the conclusions
in the Lemma.\end{proof}
\begin{lem}[see e.g. \cite{DS88b}]
\label{lem:vN def-space}Let $L$ be a \textcolor{black}{closed}
Hermitian operator with dense domain $\mathscr{D}_{0}$ in a Hilbert
space. Set 
\begin{eqnarray}
\mathscr{D}_{\pm} & = & \{\psi_{\pm}\in dom(L^{*})\left.\right|L^{*}\psi_{\pm}=\pm i\psi_{\pm}\}\nonumber \\
\mathscr{C}(L) & = & \{U:\mathscr{D}_{+}\rightarrow\mathscr{D}_{-}\left.\right|U^{*}U=P_{\mathscr{D}_{+}},UU^{*}=P_{\mathscr{D}_{-}}\}\label{eq:vN1}
\end{eqnarray}
where $P_{\mathscr{D}_{\pm}}$ denote the respective projections.
Set 
\[
\mathscr{E}(L)=\{S\left.\right|L\subseteq S,S^{*}=S\}.
\]
Then there is a bijective correspondence between $\mathscr{C}(L)$
and $\mathscr{E}(L)$, given as follows: 

If $U\in\mathscr{C}(L)$, and let $L_{U}$ be the restriction of $L^{*}$
to 
\begin{equation}
\{\varphi_{0}+f_{+}+Uf_{+}\left.\right|\varphi_{0}\in\mathscr{D}_{0},f_{+}\in\mathscr{D}_{+}\}.\label{eq:vN2}
\end{equation}
Then $L_{U}\in\mathscr{E}(L)$, and conversely every $S\in\mathscr{E}(L)$
has the form $L_{U}$ for some $U\in\mathscr{C}(L)$. With $S\in\mathscr{E}(L)$,
take 
\begin{equation}
U:=(S-iI)(S+iI)^{-1}\left.\right|_{\mathscr{D}_{+}}\label{eq:vN3}
\end{equation}
and note that 
\begin{enumerate}
\item $U\in\mathscr{C}(L)$, and
\item $S=L_{U}$.
\end{enumerate}

Vectors $f$ in $dom(L^{*})$ admit a unique decomposition $f=\varphi_{0}+f_{+}+f_{-}$
where $\varphi_{0}\in dom(L)$, and $f_{\pm}\in\mathscr{D}_{\pm}$.
For the boundary-form $\mathbf{B}(\cdot,\cdot)$, we have
\begin{eqnarray*}
i\mathbf{B}(f,f) & = & \left\langle L^{*}f,f\right\rangle -\left\langle f,L^{*}f\right\rangle \\
 & = & \left\Vert f_{+}\right\Vert ^{2}-\left\Vert f_{-}\right\Vert ^{2}.
\end{eqnarray*}

\end{lem}

\subsection{Prior Literature}

There are related investigations in the literature on spectrum and
deficiency indices. For the case of indices $(1,1)$, see for example
\cite{ST10,Ma11}. For a study of odd-order operators, see \cite{BH08}.
Operators of even order in a single interval are studied in \cite{Oro05}.
The paper \cite{BV05} studies matching interface conditions in connection
with deficiency indices $(m,m)$. Dirac operators are studied in \cite{Sak97}.
For the theory of selfadjoint extensions operators, and their spectra,
see \cite{Smu74,Gil72}, for the theory; and \cite{Naz08,VGT08,Vas07,Sad06,Mik04,Min04}
for recent papers with applications. For applications to other problems
in physics, see e.g., \cite{AH11,PoRa76,Ba49,MK08}. For related problems
regarding spectral resolutions, but for fractal measures, see e.g.,
\cite{DJ07,DJ09,DJ11}.

\section{Momentum Operators\label{sec:P}}

By momentum operator $P$ we mean the generator for the group of translations
in $L^{2}(-\infty,\infty)$, see (\ref{eq:MomentumOperator}) below.
There are several reasons for taking a closer look at restrictions
of the operator $P.$ In our analysis, we study spectral theory determined
by the complement of two bounded disjoint intervals, i.e., the union
of one bounded component and two unbounded components (details below.)
Our motivation derives from quantum theory (see section \ref{sec:scatter}),
and from the study of spectral pairs in geometric analysis; see e.g.,
\cite{DJ07}, \cite{Fu74}, \cite{JP99}, \cite{Laba01}, and \cite{PW01}.
In our model, we examine how the spectral theory depends on both variations
in the choice of the two intervals, as well as on variations in the
von Neumann parameters.

Granted that in many applications, one is faced with vastly more complicated
data and operators; nonetheless, it is often the case that the more
subtle situations will be unitarily equivalent to a suitable model
involving $P$. This is reflected for example in the conclusion of
the Stone-von Neumann uniqueness theorem: The Weyl relations for quantum
systems with a finite number of degree of freedom are unitarily equivalent
to the standard model with momentum and position operators $P$ and
$Q$. For details, see e.g., \cite{Jo81}.

\subsection{The boundary form, spectrum, and the group $U(2)$\label{sub:bform}}

Since the problem is essentially invariant under affine transformations
we may assume the two intervals are $I_{1}=[0,1]$ and $I_{2}=[\alpha,\beta]$,
$\alpha>1$; and the exterior domain 
\begin{equation}
\Omega:=I_{-}\cup I_{0}\cup I_{+}\label{eq:Omega}
\end{equation}
consists of three components 
\begin{equation}
I_{-}:=(-\infty,0),\: I_{0}:=(1,\alpha),\: I_{+}:=(\beta,\infty).\label{eq:Omega-1}
\end{equation}

Let $L^{2}(\Omega)$ be the Hilbert space with respect to the inner
product
\begin{equation}
\langle f\mid g\rangle:=\int_{I_{-}}f\overline{g}+\int_{I_{0}}f\overline{g}+\int_{I_{+}}f\overline{g}\label{eq:InnerProduct}
\end{equation}

The \emph{maximal momentum operator} is 
\begin{equation}
P:=\frac{1}{i2\pi}\frac{d}{dt}\label{eq:MomentumOperator}
\end{equation}
with domain $\mathscr{D}(P)$ equal to the set of absolutely continuous
functions on $\Omega$ where both $f$ and $Pf$ are square-integrable. 

The \emph{boundary form} associated with $P$ is defined as the form
\begin{equation}
\mathbf{B}(f,g):=\langle Pf\mid g\rangle-\langle f\mid Pg\rangle\label{eq:BoundaryForm1}
\end{equation}
on $\mathscr{D}(P).$ Clearly, 
\begin{equation}
\mathbf{B}(f,g)=f(1)\overline{g(1)}-f(0)\overline{g(0)}+f(\beta)\overline{g(\beta)}-f(\alpha)\overline{g(\alpha)}.\label{eq:BoundaryForm2}
\end{equation}
For $f\in\mathscr{D}\left(P\right),$ let $\rho_{1}(f):=\left(f(1),f(\beta)\right)$
and $\rho_{2}(f):=\left(f(0),f(\alpha)\right).$ Then 
\begin{equation}
\mathbf{B}(f,g)=\left\langle \rho_{1}(f)\mid\rho_{1}(g)\right\rangle -\left\langle \rho_{2}(f)\mid\rho_{2}(g)\right\rangle .\label{eq:BoundaryForm3}
\end{equation}
Hence $\left(\mathbb{C}^{2},\rho_{1},\rho_{2}\right)$ is a boundary
triple for $P.$ The set of selfadjoint restrictions of $P$ is parametrized
by the group $U(2)$ of all unitary $2\times2$ matrices, see e.g.,
\cite{dO09}. Explicitly, any unitary $2\times2$ matrix $B$ determines
a selfadjoint restriction $P_{B}$ of $P$ by setting 
\begin{equation}
\mathscr{D}\left(P_{B}\right):=\left\{ f\in\mathscr{D}\left(P\right)\mid B\rho_{1}(f)=\rho_{2}(f)\right\} .\label{eq:ExtensionDomain1}
\end{equation}
Conversely, every selfadjoint restriction of $P$ is obtained in this
manner. 

When $B\in U(2)$ is fixed, we will denote the corresponding selfadjoint
extension operator $P_{B}$. (For our parametrization of $U(2)$ see
(\ref{eq:2-by-2 Unitary}).)

In sections \ref{sec:P} and \ref{sec:sp} below we prove the following
theorem. 
\begin{thm}
\label{thm:UB}If $B\in U(2)$ has its parameter $w$ satisfying $0<w\leq1$,
then there is a system of bounded generalized eigenfunctions $\{\psi_{\lambda}^{(B)};\lambda\in\mathbb{R}\}$,
and a positive Borel function $F_{B}(\cdot)$ on $\mathbb{R}$ such
that the unitary one-parameter group $U_{B}(t)$ in $L^{2}(\Omega)$
generated by $P_{B}$ has the form
\begin{equation}
\left(U_{B}(t)f\right)(x)=\int_{\mathbb{R}}e_{\lambda}(-t)\left\langle \psi_{\lambda}^{(B)},f\right\rangle _{\Omega}\psi_{\lambda}^{(B)}(x)F_{B}(\lambda)d\lambda\label{eq:UB-3}
\end{equation}
for all $f\in L^{2}(\Omega)$, $x\in\Omega$, and $t\in\mathbb{R}$;
where
\[
\left\langle \psi_{\lambda}^{(B)},f\right\rangle _{\Omega}:=\int_{\Omega}\overline{\psi_{\lambda}^{(B)}(y)}f(y)dy.
\]

We further show that (when $w(B)>0$) the density function $F_{B}(\cdot)$
in (\ref{eq:UB-3}) is periodic in $\lambda$, and that, in each period,
$F_{B}(\cdot)$ is a Poisson kernel, determined from a specific action
of the group $U(2)$.
\end{thm}
In section \ref{sec:P}, we prepare with some technical lemmas; and
in section \ref{sec:sp} we compute explicit formulas for the expansion
(\ref{eq:UB-3}) above, and we discuss their physical significance.

In particular, we note that when $w>0$, there are no bound-state
contributions to the expansion (\ref{eq:UB-3}). By contrast if $w=0$,
there are bound-states. This entails embedded point-spectrum. In all
cases the point-spectrum has the form $\frac{1}{l}\mathbb{Z}$ where
$l=\alpha-1$ is the length of the interval $I_{0}$.

\subsection{Reproducing Kernel Hilbert Space}

In this section we introduce a certain reproducing kernel Hilbert
space $\mathscr{H}_{1}(\Omega)$; a first order Sobolev space, hence
the subscript 1. Its reproducing kernel is found (Lemma \ref{lem:k-1}),
and it serves two purposes: First, we show that each of the unbounded
selfadjoint extension operators $P_{B}$, defined from (\ref{eq:ExtensionDomain1})
in sect \ref{sub:bform}, have their graphs naturally embedded in
$\mathscr{H}_{1}(\Omega)$. Secondly, for each $P_{B}$, the reproducing
kernel for $\mathscr{H}_{1}(\Omega)$ helps us pin down the generalized
eigenfunctions for $P_{B}$. The arguments for this are based in turn
on Lemma \ref{lem:vN def-space} and the boundary form $\boldsymbol{B}$
from (\ref{eq:BoundaryForm1}) and (\ref{eq:BoundaryForm2}).
\begin{lem}
\label{lem:k-1}Let 
\begin{equation}
\Omega=I_{-}\cup I_{0}\cup I_{+}\label{eq:Omega-2}
\end{equation}
be as above, and $L^{2}(\Omega)$ be the Hilbert space of all $L^{2}$-functions
on $\Omega$ with inner product $\left\langle \cdot,\cdot\right\rangle _{\Omega}$
and norm $\left\Vert \cdot\right\Vert _{\Omega}$. Set 
\[
\mathscr{H}_{1}(\Omega)=\{f\in L^{2}(\Omega)\left|\right.Df=f'\in L^{2}(\Omega)\};
\]
then $\mathscr{H}_{1}(\Omega)$ is a reproducing kernel Hilbert space
of functions on $\overline{\Omega}$ (closure). \end{lem}
\begin{proof}
For the special case where $\Omega=\mathbb{R}$, the details are in
\cite{Jo81}. For the case where $\Omega$ is the exterior domain
from (\ref{eq:Omega-2}), we already noted (Lemma \ref{lem:cont})
that each $f\in\mathscr{H}_{1}(\Omega)$ has a continuous representation
$\tilde{f}$, and that $\tilde{f}$ vanishes at $\pm\infty$. The
inner product in $\mathscr{H}_{1}(\Omega)$ is
\begin{equation}
\left\langle f,g\right\rangle _{\mathscr{H}_{1}(\Omega)}=\left\langle f,g\right\rangle _{\Omega}+\left\langle f',g'\right\rangle _{\Omega}.\label{eq:H-inner}
\end{equation}
Let $x\in\overline{\Omega}=\overline{I}_{-}\cup\overline{I}_{0}\cup\overline{I}_{+}$,
and denote by $J$ the interval containing $x;$ and let $p$ be a
boundary point in $J$. Then an application of Cauchy-Schwarz yields
\begin{alignat*}{1}
\left|\tilde{f}(x)\right|^{2}-\left|\tilde{f}(p)\right|^{2} & =2\Re\int_{p}^{x}\overline{f(y)}f'(y)dy\\
 & \leq\left\Vert f\right\Vert _{J}^{2}+\left\Vert f'\right\Vert _{J}^{2}\leq\left\Vert f\right\Vert _{\mathscr{H}_{1}(\Omega)}^{2}.
\end{alignat*}

We conclude that the linear functional 
\[
\mathscr{H}_{1}(\Omega)\ni f\rightsquigarrow\tilde{f}(x)\in\mathbb{C}
\]
is continuous on $\mathscr{H}_{1}(\Omega)$ with respect to the norm
from (\ref{eq:H-inner}). By Riesz, applied to $\mathscr{H}_{1}(\Omega)$,
we conclude that there is a unique $k_{x}\in\mathscr{H}_{1}(\Omega)$
such that
\begin{equation}
\tilde{f}(x)=\left\langle k_{x},f\right\rangle _{\mathscr{H}_{1}(\Omega)}\label{eq:rep}
\end{equation}
for all $f\in\mathscr{H}_{1}(\Omega)$. 

If $x$ in (\ref{eq:rep}) is a boundary point, then the formula must
be modified using instead $\tilde{f}(x_{+})=$ limit from the right
if $x$ is a left-hand side end-point in $J$. If $x$ is instead
a right-hand side end-point in $J$, then use $\tilde{f}(x_{-})$
in formula (\ref{eq:rep}). This concludes the proof of the Lemma.
\end{proof}
We are using here standard tools on reproducing kernel Hilbert spaces
(RKHS). For the essential properties of RKHSs, and their use in scattering
theory, see \cite{ASV06,ADR02}.
\begin{lem}
\label{lem:k}Let $\Omega\subset\mathbb{R}$ be an open subset, and
let $J=(a,b)$ be a bounded connected component in $\Omega$. Then
the reproducing kernels for evaluation in $\mathscr{H}_{1}(\Omega)$
at the two endpoints $a$ and $b$ depend only on $\mathscr{H}_{1}(J)$.
The two kernels $k_{a}$ and $k_{b}$ can be taken to be zero in $\Omega\backslash J$.
Let
\begin{equation}
k_{a}(x)=\frac{\mbox{coh}(b-x)}{\mbox{sih}(b-a)}\label{eq:ka}
\end{equation}
and
\begin{equation}
k_{b}(x)=\frac{\mbox{coh}(x-a)}{\mbox{sih}(b-a)}\label{eq:kb}
\end{equation}
defined for all $x\in J$, and $0$ in $\Omega\backslash J$. Here,
$\mbox{coh}$, and $\mbox{sih}$ denote the usual hyperbolic trigonometric
functions. Then
\begin{equation}
\tilde{f}(a_{+})=\left\langle k_{a},f\right\rangle _{\mathscr{H}_{1}(\Omega)},\:\mbox{and }\tilde{f}(b_{-})=\left\langle k_{b},f\right\rangle _{\mathscr{H}_{1}(\Omega)}\label{eq:ab}
\end{equation}
hold for all $f\in\mathscr{H}_{1}(\Omega)$.

For $a<x<b$, the reproducing kernel function $k_{x}(\cdot)$ of 
\begin{equation}
f(x)=\left\langle k_{x},f\right\rangle _{\mathscr{H}_{1}(\Omega)},\: f\in\mathscr{H}_{1}(\Omega)\label{eq:kx}
\end{equation}
 is
\begin{equation}
k_{x}(y)=\frac{\mbox{sih}(b-x)\mbox{coh}(b-y)+\mbox{sih}(x-a)\mbox{coh}(y-a)}{\left(\mbox{sih}(b-a)\right)^{2}}.\label{eq:kx-1}
\end{equation}
\end{lem}
\begin{proof}
Since the two kernels are zero in the complement $\Omega\backslash J$,
we only need to determine them in the interval $J=(a,b)$. A direct
analysis shows that they must have the form
\begin{equation}
Ae^{a-x}+Be^{x-b}\label{eq:tmp-6}
\end{equation}
where $A$ and $B$ are constants to be determined from the two conditions
(\ref{eq:ab}). When this is done we find the values of $A$ and $B$
in (\ref{eq:tmp-6}), and a computation yields the desired formulas
(\ref{eq:ka}) and (\ref{eq:kb}). The formula (\ref{eq:kx-1}) for
the kernel function $k_{x}$, when $x$ is an interior point, may
be obtained from the endpoint formulas (\ref{eq:ka}) and (\ref{eq:kb}),
and an interpolation argument.\end{proof}
\begin{rem}
\label{rmk:Pmin}Consider the operator $P_{min}$ in $L^{2}(\Omega)$
with domain
\begin{equation}
\mathscr{D}(P_{min})=\{f\in\mathscr{H}_{1}(\Omega);\tilde{f}\equiv0\mbox{ on }\partial\Omega\},\label{eq:Pmin}
\end{equation}
and $P_{min}=\frac{1}{i2\pi}\frac{d}{dx}$. Then $P_{min}$ is Hermitian
(symmetric) on its domain in $L^{2}(\Omega)$, and for its adjoint
operator $P_{min}^{*}$ we have $\mathscr{D}(P_{min}^{*})=\mathscr{H}_{1}(\Omega)$.
Moreover, for every $B\in U(2)$, we have two strict inclusions of
graphs:
\begin{equation}
P_{min}\subsetneqq P_{B}\subsetneqq P_{min}^{*}(=P_{max}).\label{eq:Pmin-1}
\end{equation}

\end{rem}

\begin{rem}
\label{rem:UB}The connection between the boundary form formulation
and the von Neumann deficiency space approach is further explored
in \cite{FeJoPed1}.

The family of unitary $2\times2$ matrices is parameterized by 
\begin{equation}
B=\left(\begin{array}{cc}
w\: e(\phi) & -\sqrt{1-w^{2}}\: e(\theta-\psi)\\
\sqrt{1-w^{2}}\: e(\psi) & w\: e(\theta-\phi)
\end{array}\right)\label{eq:2-by-2 Unitary}
\end{equation}
where $0\leq w\leq1$, $\theta,\phi,\psi\in\mathbb{R}$, and 
\begin{equation}
e(x):=e^{i2\pi x}.\label{eq:NormalizedExp}
\end{equation}
From (\ref{eq:2-by-2 Unitary}), note $\det B=e(\theta)$; and (\ref{eq:2-by-2 Unitary})
is consistent with the parametrization of $SU_{2}$ as follows:
\begin{equation}
\left(\begin{array}{cc}
\overline{a} & -b\\
\overline{b} & a
\end{array}\right)\label{eq:B}
\end{equation}
where $a,b\in\mathbb{C}$ satisfy $\left|a\right|^{2}+\left|b\right|^{2}=1$.
We have $a=w\: e(-\phi)$, and so $w=\left|a\right|\in[0,1]$. \end{rem}
\begin{prop}
\label{prop:kD}Let $\alpha,\beta\in\mathbb{R}$, $1<\alpha<\beta<\infty$,
and set 
\begin{equation}
\Omega=(-\infty,0)\cup(1,\alpha)\cup(\beta,\infty)\label{eq:Omega-3}
\end{equation}
be as in (\ref{eq:Omega-4})-(\ref{eq:Omega-1}). Let $B\in U(2)$,
and let $P_{B}$ be the corresponding selfadjoint operator (see Lemma
\ref{lem:vN def-space}). Let $k_{0}$, $k_{1}$, $k_{\alpha}$, and
$k_{\beta}$ be the reproducing kernels of the four boundary points
in (\ref{eq:Omega-3}), see Lemma \ref{lem:k-1}. Set 
\begin{equation}
k_{R}=\left(\begin{array}{c}
k_{0}\\
k_{\alpha}
\end{array}\right)\mbox{ and \,\ }k_{L}=\left(\begin{array}{c}
k_{1}\\
k_{\beta}
\end{array}\right)\label{eq:tmp-7}
\end{equation}
as elements in $\mathscr{H}_{1}(\Omega)\oplus\mathscr{H}_{1}(\Omega)$,
$L$ for points on the left, and $R$ for right-hand side boundary
points. Then $P_{B}$ is characterized by its dense domain in $L^{2}(\Omega)$
as follows:
\begin{equation}
\mathscr{D}(P_{B})=\left\{ f\in\mathscr{H}_{1}(\Omega);f\oplus f\perp\left(k_{R}-Bk_{L}\right)\mbox{ in }\mathscr{H}_{1}(\Omega)\oplus\mathscr{H}_{1}(\Omega)\right\} .\label{eq:tmp-10}
\end{equation}
\end{prop}
\begin{proof}
The graph of $P_{B}$ is 
\[
G(P_{B})=\left\{ \left(\begin{array}{c}
f\\
P_{B}f
\end{array}\right);f\in\mathscr{D}(P_{B})\right\} ,
\]
see (\ref{eq:ExtensionDomain1}), and
\[
\left\Vert f\right\Vert _{L^{2}(\Omega)}^{2}+\left\Vert P_{B}f\right\Vert _{L^{2}(\Omega)}^{2}=\left\Vert f\right\Vert _{\mathscr{H}_{1}(\Omega)}^{2},\: f\in\mathscr{D}(P_{B}).
\]
Hence the characterization of $\mathscr{D}(P_{B})$ in (\ref{eq:ExtensionDomain1})
reads:
\begin{equation}
B\left(\begin{array}{c}
\left\langle k_{1},f\right\rangle _{\mathscr{H}_{1}(\Omega)}\\
\left\langle k_{\beta},f\right\rangle _{\mathscr{H}_{1}(\Omega)}
\end{array}\right)=\left(\begin{array}{c}
\left\langle k_{0},f\right\rangle _{\mathscr{H}_{1}(\Omega)}\\
\left\langle k_{\alpha},f\right\rangle _{\mathscr{H}_{1}(\Omega)}
\end{array}\right),\: f\in\mathscr{D}(P_{B}),\label{eq:tmp-8}
\end{equation}
where we have used Lemma \ref{lem:k-1}. Introducing $k_{L}$ and
$k_{R}$ as in (\ref{eq:tmp-7}), we see that (\ref{eq:tmp-8}) is
indeed equivalent to the characterization in (\ref{eq:tmp-10}).\end{proof}
\begin{rem}
\label{rem:n}The characterization (\ref{eq:tmp-10}) in Proposition
\ref{prop:kD} extends to more general open subsets $\Omega$ in $\mathbb{R}$:
It holds \emph{mutatis mutandis, }that if $\Omega$ is the union of
a finite number of bounded components, and two unbounded, i.e., 
\begin{equation}
\Omega=(-\infty,\beta_{1})\cup\bigcup_{i=1}^{n-1}(\alpha_{i},\beta_{i+1})\cup(\alpha_{n},\infty)\label{eq:Omega-4}
\end{equation}
where
\[
-\infty<\beta_{1}<\alpha_{1}<\beta_{2}<\alpha_{2}<\cdots<\alpha_{n-1}<\beta_{n}<\alpha_{n}<\infty.
\]

Set 
\[
k_{R}=\left(\begin{array}{c}
k_{\beta_{1}}\\
k_{\beta_{2}}\\
\vdots\\
k_{\beta_{n}}
\end{array}\right)\;\mbox{and }\: k_{L}=\left(\begin{array}{c}
k_{\alpha_{1}}\\
k_{\alpha_{2}}\\
\vdots\\
k_{\alpha_{n}}
\end{array}\right)
\]
in $\bigoplus_{i=1}^{n}\mathscr{H}_{1}(\Omega)$. Let $B$ be a unitary
complex $n\times n$ matrix, i.e., $B\in U(n)$; then there is a unique
selfadjoint operator $P_{B}$ with dense domain $\mathscr{D}(P_{B})$
in $L^{2}(\Omega)$ such that
\begin{equation}
\mathscr{D}(P_{B})=\left\{ f\in\mathscr{H}_{1}(\Omega);\underset{n\mbox{ times}}{\underbrace{f\oplus\cdots\oplus f}}\perp(k_{R}-Bk_{L})\mbox{ in }\bigoplus_{i=1}^{n}\mathscr{H}_{1}(\Omega)\right\} ;\label{eq:tmp-11}
\end{equation}
and all the selfadjoint extensions of the minimal operator $D_{min}$
in $L^{2}(\Omega)$ arise this way. In particular, the deficiency
indices are $(n,n)$.
\end{rem}
\begin{figure}[H]
\setlength{\unitlength}{1cm}
\begin{picture}(10,0.1)

\thicklines
\put(0,0){\line(1,0){1}}
\put(2,0){\line(3,0){1}}
\put(4,0){\line(5,0){1}}
\put(7,0){\line(8,0){1}}
\put(9,0){\line(10,0){1}}

\put(-0.4,-0.5){$-\infty$}
\put(0.8,-0.5){$\beta_{1}$}

\put(1.8,-0.5){$\alpha_{1}$}
\put(2.8,-0.5){$\beta_{2}$}

\put(3.8,-0.5){$\alpha_{2}$}
\put(4.8,-0.5){$\beta_{3}$}

\put(6,0){$\ldots$}

\put(6.8,-0.5){$\alpha_{n-1}$}
\put(7.8,-0.5){$\beta_{n}$}

\put(8.8,-0.5){$\alpha_{n}$}
\put(9.8,-0.5){$+\infty$}

\end{picture}

\vspace{1cm}

\caption{$\Omega=$ the complement in $\mathbb{R}$ of $n$ finite and disjoint
intervals.}

\end{figure}
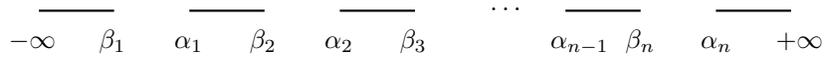

\begin{prop}
\label{prop:n}Let $n>2$; and set $J_{i}=(\alpha_{i},\beta_{i+1})$,
$J_{-}=(-\infty,\beta_{1})$, $J_{+}=(\alpha_{n},\infty)$ as in (\ref{eq:Omega-4}).
Set $\tilde{\Omega}=\cup_{i=1}^{n-1}J_{i}$, so
\begin{equation}
L^{2}(\Omega)\cong L^{2}(\tilde{\Omega})\oplus L^{2}(J_{-}\cup J_{+}).\label{eq:tmp-38}
\end{equation}
Of the selfadjoint extension operators $P_{B}$, indexed by $B\in U(n)$,
we get the $\oplus$ direct decomposition 
\begin{equation}
P_{B}\cong P_{\tilde{\Omega}}\oplus P_{ext}\label{eq:tmp-39}
\end{equation}
where $P_{\tilde{\Omega}}$ is densely defined and s.a. in $L^{2}(\tilde{\Omega})$
and $P_{ext}$ is densely defined and s.a. in $L^{2}(J_{-}\cup J_{+})$,
if and only if $B$ (in $U(n)$) has the form

\begin{equation}
\left(\begin{array}{ccccc}
0 & \cdots & 0 & \vline & e(\theta)\\
\hline  &  &  & \vline & 0\\
 & \huge\mbox{\ensuremath{\tilde{B}}} &  & \vline & \vdots\\
 &  &  & \vline & 0
\end{array}\right)\label{tmp}
\end{equation}
for some $\theta\in\mathbb{R}/\mathbb{Z}$, and $\tilde{B}\in U(n-1)$.\end{prop}
\begin{proof}
Note that presentation (\ref{tmp}) for some $B\in U(n)$ implies
the boundary condition $f(\beta_{n})=e(\theta)f(\alpha_{1})$ for
$f\in\mathscr{D}(P_{B})$ when $P_{B}$ is the selfadjoint operator
in $L^{2}(\Omega)$ determined in Remark \ref{rem:n}. And, moreover,
the $\oplus$ sum decomposition (\ref{eq:tmp-39}) will be satisfied.

One checks that the converse holds as well; see also Theorem \ref{thm:sp-w0}
below; which is a special case.\end{proof}
\begin{rem}[Internal domains vs external]
 It is of interest to compare spectral theory for the selfadjoint
restrictions $P_{B}$ of the momentum operator in $L^{2}(\Omega)$
in the two cases when $\Omega$ is internal, as opposed to external.
By $\Omega$ internal we mean that $\Omega$ is a finite union of
disjoint and finite intervals. The case when $\Omega$ is the union
of two finite disjoint intervals was considered in \cite{FeJoPed1},
and we found that the possibilities for the spectral representation
of $P_{B}$ includes both continuous and discrete; but more importantly,
we found in \cite{FeJoPed1} that the embedded point-spectrum of some
of the selfadjoint operators $P_{B}$ arising this way may be non-periodic.

Contrast this with the external case studied here, i.e., when $\Omega$
is instead the complement of two finite disjoint closed intervals;
so the case when $\Omega$ is the union of three components, one bounded
$I_{0},$ and two unbounded. There are some aspects of this external
problem that are simpler: In the present external problem, the only
possibility for point-spectrum is periodic (see Corollary \ref{cor:su2}).
The reason for this is that point-spectrum corresponds to bound-states
for wave functions trapped in a single bounded interval. In other
words, there are only those bound-states that are trapped in the single
finite component (see Figure \ref{fig:w0-1}.) Note in Fig \ref{fig:w0-1}
the two thick walls (barriers) on either side of $I_{0}$, and the
corresponding periodic motion inside $I_{0}$.

Had we instead taken $\Omega$ to be the complement of three finite
disjoint closed intervals, then the von Neumann deficiency indices
would be $(3,3)$ and there would be examples of $B$ in $U(3)$ such
that $P_{B}$ could have non-periodic embedded point-spectrum; so
cases analogous to the non-periodic case in \cite{FeJoPed1}. And
as a result, the spectral density measure $\sigma_{B}(\cdot)$ might
be non-periodic. 
\end{rem}

\section{Spectral Theory\label{sec:sp}}

In this section, we fix an exterior domain $\Omega$, the complement
of two finite disjoint intervals. For every $B$ in $U(2)$, we introduce
the corresponding selfadjoint operator $P_{B}$ with dense domain
in $L^{2}(\Omega)$, see (\ref{eq:ExtensionDomain1}). We are concerned
about spectral theory for $P_{B}$, and scattering theory for the
unitary one-parameter group $U_{B}(t)$ generated by $P_{B}$. In
our study of spectral theory for $\{U_{B}(t)\}_{t\in\mathbb{R}}$,
we rely on tools from \cite{Sto90}. In Theorem \ref{thm:sp-w-1}
below we show that, for the general case of $B$, $U_{B}(t)$ has
simple spectrum (i.e., multiplicity one). Simple spectrum was introduced
in \cite{Sto90}. For fixed $B$, we further write down the spectral
representation for $U_{B}(t)$.

Our spectral representation formula for $P_{B}$ is presented in Theorem
\ref{thm:sp-w-1} below; and the scattering operator (and scattering
matrix) for $U_{B}(t)$ is given in Theorem \ref{thm:scatter}.

Fix two intervals $I_{1}=[0,1]$ and $I_{2}=[\alpha,\beta]$, $\alpha>1$,
and $\Omega=I_{-}\cup I_{0}\cup I_{+}$ being the exterior domain
in (\ref{eq:Omega-1-1}), where $I_{-}=(-\infty,0)$, $I_{0}=(1,\alpha)$,
and $I_{+}=(\beta,\infty)$. Let $\chi_{-}$, $\chi_{0}$, $\chi_{+}$
be the corresponding characteristic functions. 

There is a one-to-one correspondence between selfadjoint restrictions
$P_{B}$ of the maximal momentum operator $P$ in $L^{2}(\Omega)$,
and the $2\times2$ unitary matrices $B$ parameterized via (\ref{eq:2-by-2 Unitary}). 

The two extreme cases $w=0$, and $w=1$ will be considered separately,
i.e., 
\begin{alignat}{1}
w=0\::\; & \left(\begin{array}{cc}
0 & -e(\theta-\psi)\\
e(\psi) & 0
\end{array}\right)\label{eq:w0}\\
w=1\::\; & \left(\begin{array}{cc}
e(\phi) & 0\\
0 & e(\theta-\phi)
\end{array}\right)\label{eq:w1}
\end{alignat}

We use (\ref{eq:2-by-2 Unitary}) in the computation of the spectrum
of the family of selfadjoint operators $P_{B}$ from Section \ref{sec:Intro}.

We show that $w=0$ is a singularity, and gives rise to embedded point-spectrum
\begin{equation}
pt_{spectrum}(P_{B(w=0,\phi,\psi,\theta)})=\frac{\psi}{\alpha-1}+\frac{1}{\alpha-1}\mathbb{Z}\label{eq:psp}
\end{equation}
embedded in the continuum. (The subscript in (\ref{eq:psp}) refers
to the degenerate matrix (\ref{eq:w0}).) For details, we refer to
Theorem \ref{thm:sp-w0}, Figure \ref{fig:w0-1}, and Remark \ref{rem:w0}
below.

\subsection{Spectrum and Eigenfunctions}

Fix a unitary $2\times2$ matrix $B.$ 

Since the selfadjoint operator has continuous spectrum, possibly with
embedded point-atoms, its spectral representation must entail generalized
eigenfunctions $\psi_{\lambda}^{(B)}$ with $\lambda\in\mathbb{R}$
denoting the spectral-variable. The reason for ``generalized'' is
that, when $\lambda$ is fixed, $\psi_{\lambda}^{(B)}$ is ``trying''
to be an eigenfunction, but it is not in $L^{2}(\Omega)$. Hence to
make precise the spectral resolution of $P_{B}$ we will need some
Gelfand-Schwartz distribution theory.

Let $\mathscr{D}_{B}$ be the $\mbox{\ensuremath{\mathscr{D}}}=C_{c}^{\infty}$
functions on the real line that together with all their derivatives
satisfies the boundary condition $B\rho_{1}(\cdot)=\rho_{2}(\cdot).$
Let $\mathscr{D}_{B}\left(\Omega\right)$ be the restrictions of the
functions in $\mathscr{D}_{B}$ to $\Omega.$ Since $\mathscr{D}$
and $C_{c}^{\infty}\left(\mathbb{R}\setminus\Omega\right)$ are nuclear
and subspaces and quotients of nuclear spaces are nuclear, it follows
that $\mathscr{D}_{B}\left(\Omega\right)$ is nuclear.

Let $\dot{P}_{B}$ denote the restriction of $P_{B}$ to $\mathscr{D}_{B}\left(\Omega\right),$
then $\dot{P}_{B}$ is continuous $\mathscr{D}_{B}\left(\Omega\right)\to\mathscr{D}_{B}\left(\Omega\right)$.
Let $\mathscr{D}_{B}'\left(\Omega\right)$ denote the set of anti-linear
continuous functionals on $\mathscr{D}_{B}\left(\Omega\right).$ Then
$\dot{P}_{B}$ extends by duality to an operator $\dot{P}_{B}'$ on
$\mathscr{D}_{B}'\left(\Omega\right).$ The duality formula for extending
$\dot{P}_{B}$ to $\mathscr{D}_{B}'\left(\Omega\right)$ is $\left(\dot{P}_{B}'\psi\right)\left(\phi\right)=\psi\left(\dot{P}_{B}\phi\right),$
sometimes we will write this as $\left\langle \phi\left|\dot{P}_{B}'\psi\right.\right\rangle =\left\langle \left.\dot{P}_{B}\phi\right|\psi\right\rangle ,$
for all $\phi$ in $\mathscr{D}_{B}\left(\Omega\right)$ and all $\psi$
in $\mathscr{D}_{B}'\left(\Omega\right).$

A \emph{generalized eigenvalue} of $P_{B}$ is a real scalar $\lambda$
for which there is a corresponding \emph{generalized eigenvector},
i.e., a $\psi_{\lambda}$ in $\mathscr{D}_{B}'\left(\Omega\right)$
such that
\begin{equation}
\left\langle \phi\left|\dot{P}_{B}'\psi_{\lambda}\right.\right\rangle =\lambda\left\langle \phi\mid\psi\right\rangle \label{eq:1-1GenEigen}
\end{equation}
for all $\phi$ in $\mathscr{D}_{B}\left(\Omega\right).$ Hence the
generalized eigenvalues/eigenvectors of $P_{B}$ are ordinary eigenvalues/eigenvectors
of $\dot{P}_{B}'.$ 

The following lemmas establish that the spectrum of $P_{B}$ is the
real line and that for fixed $\lambda$ the corresponding generalized
eigenspace is spanned by the functions

\begin{equation}
\psi_{\lambda}:=\left(a_{\lambda}\chi_{-}+b_{\lambda}\chi_{0}+c_{\lambda}\chi_{+}\right)e_{\lambda}\label{eq:2-1GenEigen}
\end{equation}
where $a_{\lambda},b_{\lambda},c_{\lambda}$ are scalars such that
(\ref{eq:ExtensionDomain1}) holds, i.e., $B\rho_{1}(\psi_{\lambda})=\rho_{2}(\psi_{\lambda}).$
Recall, $e_{\lambda}(x)=e(\lambda x)=e^{i2\pi\lambda x},$ so we write
(\ref{eq:2-1GenEigen}) as 
\[
\psi_{\lambda}(x)=\left(a_{\lambda}\chi_{-}(x)+b_{\lambda}\chi_{0}(x)+c_{\lambda}\chi_{+}(x)\right)e_{\lambda}(x)
\]
for $\lambda,x\in\mathbb{R}.$ 
\begin{lem}
Each real number $\lambda$ is a generalized eigenvalue of $P_{B}$
and the corresponding generalized eigenfunctions are the functions
(\ref{eq:2-1GenEigen}). \end{lem}
\begin{proof}
We can write (\ref{eq:1-1GenEigen}) as $\psi_{\lambda}'=i2\pi\lambda\psi_{\lambda}.$
Solving this differential equation using weak solution are also strong
solutions we see that (\ref{eq:2-1GenEigen}) holds. It follows from
(\ref{eq:2-1GenEigen}) that both sides of (\ref{eq:1-1GenEigen})
are given by integrals, hence we can rewrite (\ref{eq:1-1GenEigen})
as 
\[
\int_{\Omega}\phi(t)\,\left(\dot{P}_{B}'\psi_{\lambda}\right)(x)\, dx=\int_{\Omega}\left(\dot{P}_{B}\phi\right)(x)\,\psi_{\lambda}(x)\, dx,
\]
where $\lambda\in\mathbb{R}$ and $\phi\in\mathscr{D}_{B}\left(\Omega\right).$
Integration by parts, then shows that the boundary form 
\[
B(\phi,\psi_{\lambda})=\phi(1)\overline{\psi_{\lambda}(1)}-\phi(0)\overline{\psi_{\lambda}(0)}+\phi(\beta)\overline{\psi_{\lambda}(\beta)}-\phi(\alpha)\overline{\psi_{\lambda}(\alpha)}=0,
\]
for all $\lambda\in\mathbb{R}$ and all $\phi\in\mathscr{D}_{B}\left(\Omega\right).$
Fixing $\lambda$ and using $\phi$ in $\mathscr{D}_{B}\left(\Omega\right)$
is arbitrary, it follows that $\psi_{\lambda}$ satisfies the boundary
condition $B\rho_{1}(\cdot)=\rho_{2}(\cdot).$ 
\end{proof}
The boundary condition (\ref{eq:ExtensionDomain1}) gives 
\begin{align}
wb\: e(\phi+\lambda)-\sqrt{1-w^{2}}c\: e(\theta-\psi+\beta\lambda) & =a\label{eq:3}\\
\sqrt{1-w^{2}}b\: e(\psi+\lambda)+wc\: e(\theta-\phi+\beta\lambda) & =b\: e(\alpha\lambda)\label{eq:4}
\end{align}

\begin{lem}
\label{lem:multiplicity-0<w}If $0<w\leq1,$ then each generalized
eigenvalue has multiplicity one, and the two functions $\lambda\mapsto a_{B}(\lambda)$
and $\lambda\mapsto c_{B}(\lambda)$ are given by the following formulas:

\begin{equation}
a_{\lambda}=b_{\lambda}w^{-1}e(\phi)e(\lambda)\left(1-\sqrt{1-w^{2}}e(-\psi+(\alpha-1)\lambda)\right)\label{eq:a(lambda)}
\end{equation}
and
\begin{equation}
c_{\lambda}=b_{\lambda}w^{-1}e(\phi-\theta)e(-(\beta-\alpha)\lambda)\left(1-\sqrt{1-w^{2}}e(\psi-(\alpha-1)\lambda)\right).\label{eq:c(lambda)}
\end{equation}
\end{lem}
\begin{proof}
If $b=0,$ then (\ref{eq:3}) and (\ref{eq:4}) shows that $a=c=0.$
If $b=1$ we can solve (\ref{eq:3}) and (\ref{eq:4}) for $a$ and
$c$. If we assume $b=1$, we can re-write the boundary conditions
as
\begin{alignat}{1}
w\: e(\phi+\lambda)-\sqrt{1-w^{2}}c_{\lambda}\: e(\theta-\psi+\beta\lambda) & =a_{\lambda}\label{eq:tmp}\\
\sqrt{1-w^{2}}e(\psi+\lambda)+w\: c_{\lambda}\: e(\theta-\phi+\beta\lambda) & =e(\alpha\lambda).\label{eq:tmp-1}
\end{alignat}
From (\ref{eq:tmp-1}), 
\begin{equation}
c_{\lambda}=\frac{e(\alpha\lambda)-\sqrt{1-w^{2}}e(\psi+\lambda)}{we(\theta-\phi+\beta\lambda)}\label{eq:tmp-2}
\end{equation}
which can be written as \ref{eq:a(lambda)}). Substituting (\ref{eq:tmp-2})
into (\ref{eq:tmp-1}), we get (\ref{eq:c(lambda)}). \end{proof}
\begin{lem}
\label{lem:multiplicity-w=00003D0}If $w=0$ then each point $\lambda\in-\frac{\psi}{1-\alpha}+\frac{1}{1-\alpha}\mathbb{Z}$
is a generalized eigenvalue of multiplicity two and all other generalized
eigenvalues have multiplicity one. In fact, for any $\lambda$ in
$\mathbb{R}$ 
\[
\psi_{\lambda}=c_{\lambda}\left(-e(\theta-\psi+\beta\lambda)\chi_{-}+\chi_{+}\right)e_{\lambda}
\]
is a generalized eigenfunctions and for $\lambda\in-\frac{\psi}{1-\alpha}+\frac{1}{1-\alpha}\mathbb{Z}$
\[
\psi_{\lambda}=b_{\lambda}\chi_{0}e_{\lambda}
\]
is also a (generalized) eigenfunction. \end{lem}
\begin{proof}
If $w=0$ then (\ref{eq:3}) and (\ref{eq:4}) reduce to 
\begin{align*}
-c_{\lambda}e(\theta-\psi+\beta\lambda) & =a_{\lambda}\\
b_{\lambda}e(\psi+\lambda) & =b_{\lambda}e(\alpha\lambda)
\end{align*}
Hence the stated formulas for the generalized eigenfunctions follow
from (\ref{eq:2-1GenEigen}). \end{proof}
\begin{lem}
The spectrum of $P_{B}$ is the real line. In particular, the set
of generalized eigenvalues equal the spectrum of $P_{B}.$ \end{lem}
\begin{proof}
Let $\lambda$ be a real number and suppose $\psi_{\lambda}$ is determined
by (\ref{eq:2-1GenEigen}). Let $h$ be a smooth functions on the
real line such that $0\leq h,-h'\leq1,$ $h(x)=1$ when $x<0,$ and
$h(x)=0$ when $x>2.$ Let 
\[
g_{k}(x):=h(x-k)h(k-x),\: k\in\mathbb{N},x\in\mathbb{R}.
\]
Then $g_{k}$ is a sequence of smooth functions on the real line such
that $0\leq g_{k},g_{k}'\leq1,$ $g_{k}(x)=0$ when $|x|>k+2,$ and
$g_{k}(x)=1$ when $|x|<k.$ Let $c_{k}$ be a positive real number
such that $\int_{\Omega}\left|c_{k}g_{k}\psi_{\lambda}\right|^{2}=1.$
For $k>\beta$ the functions $f_{\lambda,k}:=c_{k}g_{k}\psi_{\lambda}$
are unit vectors in the domain of $P_{B}$ and 
\[
\left\Vert P_{B}f_{\lambda,k}-\lambda f_{\lambda,k}\right\Vert _{2}^{2}\to0\text{ as }k\to\infty.
\]
Consequently, $\lambda$ is in the spectrum of $P_{B}.$ 
\end{proof}

\subsection{Direct Integral Representation}

von Neumann \cite{vN49} showed there exists a probability measure
$\nu$ on $\mathbb{R}$ and a $\nu$-measurable field $H\left(\xi\right)$
of separable Hilbert spaces such that, if 
\[
K:=\int_{\mathbb{R}}H\left(\xi\right)d\nu\left(\xi\right),
\]
then there is a unitary $F:L^{2}(\Omega)\to K,$ such that 
\begin{equation}
\left(F\left(P_{B}f\right)\right)\left(\xi\right)=\xi\left(Ff\right)\left(\xi\right)\label{eq:1GenEigen}
\end{equation}
for all $\xi\in\Xi=\mathrm{supp}\left(\nu\right)$ and all $f$ in
the domain of $P_{B}.$ Furthermore, if $n\left(\xi\right)$ denotes
the dimension of $H(\xi)$ there exists a sequence $\left(g_{k}\right)$
of $\nu$-measurable vector fields such that 
\[
\left\{ g_{k}\left(\xi\right)\mid k<n\left(\xi\right)+1\right\} 
\]
is an orthonormal basis for $H\left(\xi\right)$ and $g_{k}\left(\xi\right)=0$
when $n\left(\xi\right)<k.$ Note $n\left(\xi\right)=\infty$ is possible. 

Let 
\begin{equation}
\left(Ff\right)_{k}\left(\xi\right):=\left\langle \left(Ff\right)\left(\xi\right)\mid g_{k}\left(\xi\right)\right\rangle _{H\left(\xi\right)}\label{eq:2GenEigen}
\end{equation}
for $f$ in $L^{2}\left(\Omega\right)$ and $k=1,2,\ldots.$ By \cite[p.83]{Mau68}
the mapping $\phi\to\left(F\phi\right)\left(\xi\right)$ is continuous
as a function $\mathscr{D}_{B}\left(\Omega\right)\to H\left(\xi\right).$
Combining this continuity with (\ref{eq:2GenEigen}) we conclude 
\begin{equation}
\delta_{\xi,k}\left(\phi\right):=\left\langle \left(F\phi\right)\left(\xi\right)\mid g_{k}\left(\xi\right)\right\rangle _{H\left(\xi\right)}\label{eq:3GenEigen}
\end{equation}
 is a continuous linear functional on $\mathscr{D}_{B}\left(\Omega\right),$
i.e., a distribution on $\Omega.$ 

Combining (\ref{eq:1GenEigen}) and (\ref{eq:3GenEigen}) we see that
\begin{align*}
\delta_{\xi,k}\left(P_{B}\phi\right) & =\left\langle \left(F\left(P_{B}\phi\right)\right)\left(\xi\right)\mid g_{k}\left(\xi\right)\right\rangle _{H\left(\xi\right)}\\
 & =\left\langle \xi\left(F\left(\phi\right)\right)\left(\xi\right)\mid g_{k}\left(\xi\right)\right\rangle _{H\left(\xi\right)}\\
 & =\xi\delta_{\xi,k}\left(\phi\right)
\end{align*}
for all $\phi$ in $\mathscr{D}_{B}\left(\Omega\right).$ Hence, $\delta_{\xi,k}'=i2\pi\xi\delta_{\xi,k}$
and consequently,
\begin{equation}
\delta_{\xi,k}=\psi_{\xi}=\left(a_{\xi,k}\chi_{-}+b_{\xi,k}\chi_{0}+c_{\xi,k}\chi_{+}\right)e_{\xi}\label{eq:4GenEigen}
\end{equation}
for some choice of constants $a_{\xi,k},b_{\xi,k},c_{\xi,k}$ such
that $\delta_{\xi,k}$ satisfies the boundary condition (\ref{eq:ExtensionDomain1}).
By (\ref{eq:3GenEigen}) these constants all vanish when $k>n\left(\xi\right).$ 

Let $f\in L^{2}(\Omega)$, write $f=f_{-}+f_{0}+f_{+}$, where $f_{-}:=\chi_{-}f$,
$f_{0}:=\chi_{0}f$, and $f_{+}:=\chi_{+}f$. By (\ref{eq:2GenEigen}),
(\ref{eq:3GenEigen}), and (\ref{eq:4GenEigen}) 
\begin{equation}
\left(F\phi\right)_{k}\left(\xi\right)=\overline{a_{\xi,k}}\widehat{\phi_{-}}\left(\xi\right)+\overline{b_{\xi,k}}\widehat{\phi_{0}}\left(\xi\right)+\overline{c_{\xi,k}}\widehat{\phi_{+}}\left(\xi\right)\label{eq:5GenEigen}
\end{equation}
for any test function $\phi$ in $\mathscr{D}{}_{B}\left(\Omega\right)$.
Here $\widehat{\psi}$ denotes the Fourier transform of $\psi.$ 
\begin{thm}
\label{thm:DirectIntegral-0<w}If $0<w\leq1$ then 
\[
\left\langle \phi\mid\psi\right\rangle _{\Omega}=\int_{\mathbb{R}}\left|a_{\xi}\right|^{2}\widehat{\phi_{-}}\left(\xi\right)\overline{\widehat{\psi_{-}}}\left(\xi\right)+\left|b_{\xi}\right|^{2}\widehat{\phi_{0}}\left(\xi\right)\overline{\widehat{\psi_{0}}}\left(\xi\right)+\left|c_{\xi}\right|^{2}\widehat{\phi_{+}}\left(\xi\right)\overline{\widehat{\psi_{+}}}\left(\xi\right)d\nu(\xi)
\]
and 
\[
\left\langle \phi\mid P_{B}\psi\right\rangle _{\Omega}=\int_{\mathbb{R}}\left|a_{\xi}\right|^{2}\widehat{\phi_{-}}\left(\xi\right)\overline{\widehat{\psi_{-}}}\left(\xi\right)+\left|b_{\xi}\right|^{2}\widehat{\phi_{0}}\left(\xi\right)\overline{\widehat{\psi_{0}}}\left(\xi\right)+\left|c_{\xi}\right|^{2}\widehat{\phi_{+}}\left(\xi\right)\overline{\widehat{\psi_{+}}}\left(\xi\right)\xi d\nu(\xi)
\]
for all $\phi,\psi$ in $\mathscr{D}_{B}(\Omega).$\end{thm}
\begin{proof}
Since $0<w\leq1$ it follows from Lemma \ref{lem:multiplicity-0<w}
that the multiplicity of each generalized eigenvalue is one, consequently
$n(\xi)=1$ for all $\xi$ in $\mathbb{R}$ and each $H(\xi)$ has
dimension one. Since $F$ is an isometry and $\phi=\phi_{-}+\phi_{0}+\phi_{+}$
is orthogonal we have 
\begin{align*}
\left\langle \phi\mid\psi\right\rangle _{\Omega} & =\left\langle \phi_{-}\mid\psi_{-}\right\rangle _{\Omega}+\left\langle \phi_{0}\mid\psi_{0}\right\rangle _{\Omega}+\left\langle \phi_{+}\mid\psi_{+}\right\rangle _{\Omega}\\
 & =\left\langle F\phi_{-}\mid F\psi_{-}\right\rangle _{\nu}+\left\langle F\phi_{0}\mid F\psi_{0}\right\rangle _{\nu}+\left\langle F\phi_{+}\mid F\psi_{+}\right\rangle _{\nu}
\end{align*}
so the result follows from (\ref{eq:5GenEigen}).
\end{proof}
Below we set $d\sigma_{B}(\xi)=\left|b_{\xi}\right|^{2}d\nu(\xi)$,
and we show that this measure $d\sigma_{B}(\xi)$ is absolutely continuous
with respect to Lebesgue measure on $\mathbb{R}$. Moreover, we calculate
the Radon--Nikodym derivative.

\begin{rem}
Setting $\psi=\phi$ we can write the first equation in Theorem \ref{thm:DirectIntegral-0<w}
as 
\begin{align*}
 & \int\left|\overline{a_{\xi}}\widehat{\phi_{-}}\left(\xi\right)+\overline{b_{\xi}}\widehat{\phi_{0}}\left(\xi\right)+\overline{c_{\xi}}\widehat{\phi_{+}}\left(\xi\right)\right|^{2}d\nu\left(\xi\right)\\
 & =\int\left|a_{\xi}\right|^{2}\left|\widehat{\phi_{-}}\left(\xi\right)\right|^{2}+\left|b_{\xi}\right|^{2}\left|\widehat{\phi_{0}}\left(\xi\right)\right|^{2}+\left|c_{\xi}\right|^{2}\left|\widehat{\phi_{+}}\left(\xi\right)\right|^{2}d\nu\left(\xi\right).
\end{align*}
The ``cross terms'' in the expansion of the square on the left hand
side vanish. 
\end{rem}
Similarly, using Lemma \ref{lem:multiplicity-w=00003D0} and separating
out the discrete part of the meausure we have: 
\begin{thm}
\label{thm:DirectIntegral-w=00003D0}If $w=0$ then 
\begin{eqnarray*}
\left\langle \phi\mid\psi\right\rangle _{\Omega} & = & \int_{\mathbb{R}}\left|a_{\xi}\right|^{2}\widehat{\phi_{-}}\left(\xi\right)\overline{\widehat{\psi_{-}}}\left(\xi\right)+\left|c_{\xi}\right|^{2}\widehat{\phi_{+}}\left(\xi\right)\overline{\widehat{\psi_{+}}}\left(\xi\right)d\nu(\xi)\\
 &  & +\sum_{\xi\in-\frac{\psi}{1-\alpha}+\frac{1}{1-\alpha}\mathbb{Z}}\left|b_{\xi}\right|^{2}\widehat{\phi_{0}}\left(\xi\right)\overline{\widehat{\psi_{0}}}\left(\xi\right)
\end{eqnarray*}
and
\begin{eqnarray*}
\left\langle \phi\mid P_{b}\psi\right\rangle _{\Omega} & = & \int_{\mathbb{R}}\xi\left(\left|a_{\xi}\right|^{2}\widehat{\phi_{-}}\left(\xi\right)\overline{\widehat{\psi_{-}}}\left(\xi\right)+\left|c_{\xi}\right|^{2}\widehat{\phi_{+}}\left(\xi\right)\overline{\widehat{\psi_{+}}}\left(\xi\right)\right)d\nu(\xi)\\
 &  & +\sum_{\xi\in-\frac{\psi}{1-\alpha}+\frac{1}{1-\alpha}\mathbb{Z}}\xi\left|b_{\xi}\right|^{2}\widehat{\phi_{0}}\left(\xi\right)\overline{\widehat{\psi_{0}}}\left(\xi\right)
\end{eqnarray*}
 for all $\phi,\psi$ in $\mathscr{D}_{B}(\Omega).$
\end{thm}

\subsection{Extreme Cases}

Fix $B$ with parameters $w,\theta,\phi,\psi$. Our analysis depends
on the parameter $w.$ We begin by considering the extreme cases $w=0$
and $w=1$. 
\begin{thm}[$w=0$]
\label{thm:sp-w0}Choose a boundary matrix $B\in U(2)$ with parameters
$w,\theta,\phi,\psi$, let $P_{B}$ be the corresponding selfadjoint
restriction of $P$. For $w=0$, there is a mixture of continuous
and discrete spectrum. More precisely, setting $a=c=0$ in (\ref{eq:2-1GenEigen})
gives eigenfunctions that are multiples of
\begin{equation}
\chi_{0}\: e_{\lambda}\label{eq:eigen-w0-disc}
\end{equation}
when $\psi+\lambda-\alpha\lambda$ is an integer, i.e., $\lambda\in-\frac{\psi}{1-\alpha}+\frac{1}{1-\alpha}\mathbb{Z}.$
On the other hand setting $b=0$ and $c=1$ gives generalized eigenfunctions
that are multiples of 
\begin{equation}
\psi_{\lambda}:=\left(-e(\theta-\psi+\beta\lambda)\chi_{-}+\chi_{+}\right)e_{\lambda}\label{eq:eigen-w0-cont}
\end{equation}
for all $\lambda\in\mathbb{R}$. Hence the spectrum equals the real
line with uniform multiplicity one and the points in $-\frac{\psi}{1-\alpha}+\frac{1}{1-\alpha}\mathbb{Z}$
are embedded eigenvalues each with multiplicity one. \end{thm}
\begin{proof}
The statement follows from Lemma \ref{lem:multiplicity-w=00003D0}
and Theorem \ref{thm:DirectIntegral-w=00003D0}.
\end{proof}
\begin{figure}[H]
\includegraphics[scale=0.8]{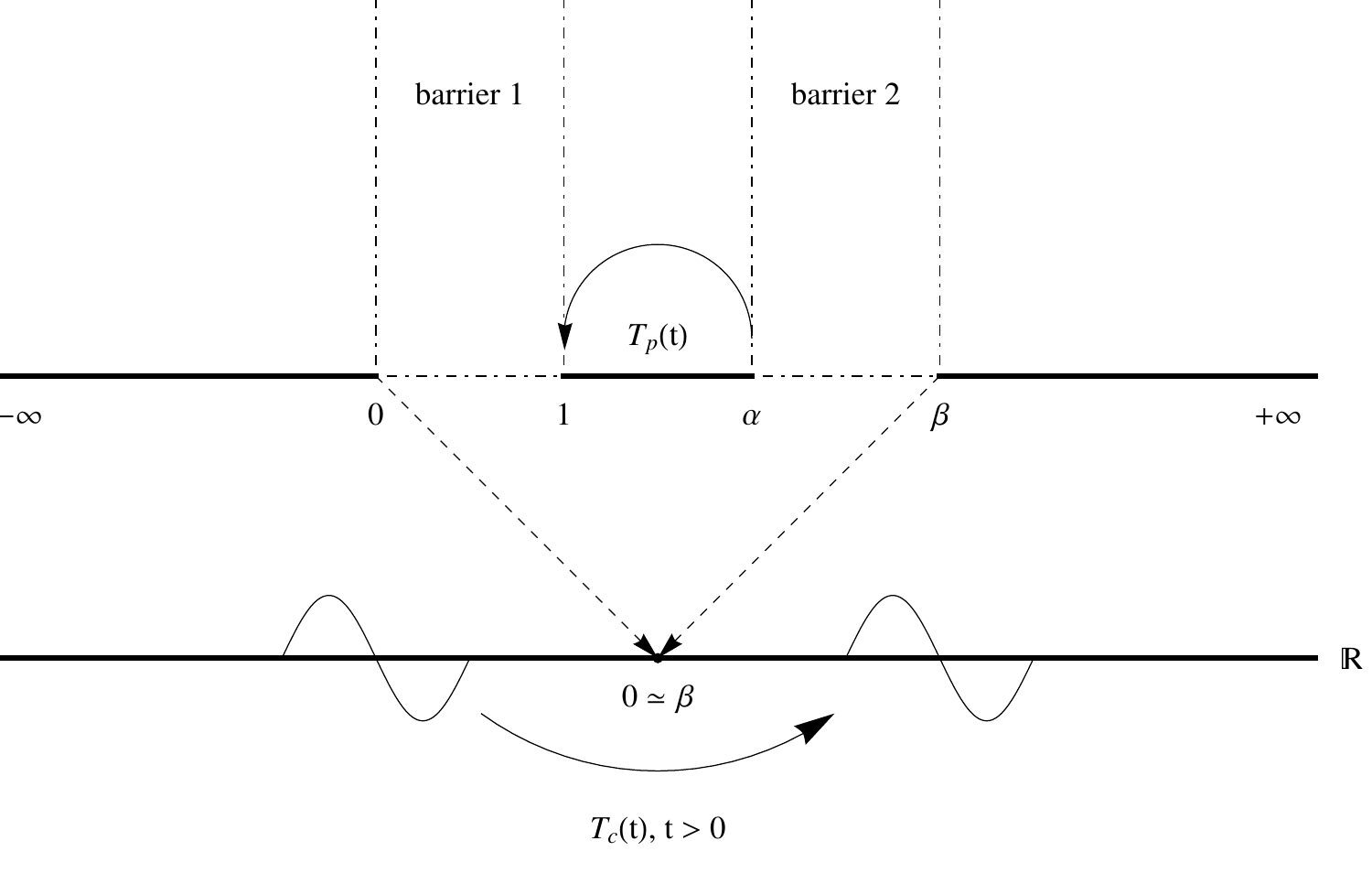}

\caption{\label{fig:w0-1}Infinite barriers.}
\end{figure}

\begin{rem}
\label{rem:w0}For $w=0$, there is no mixing/interaction between
the bounded component $I_{0}$ and the union of the two unbounded
components $I_{-}$ and $I_{+}$, i.e., the two half-lines, $I_{-}$
including $-\infty$ ; and $I_{+}$ including $+\infty$. The unitary
one-parameter group $U_{B}(t)$, acting on $L^{2}(\Omega)$, is unitarily
equivalent to a direct sum of two one-parameter groups, $T_{p}(t)$
and $T_{c}(t)$.

These two one-parameter groups are obtained as follows: Start with
$T(t)$, the usual one-parameter group of right-translation by $t$.
The subscript p indicates periodic translation, i.e., translation
by $t$ modulo $1$, and with a phase factor. 

Hence, $T_{p}(t)$ accounts for the bound-states. By contrast, the
one-parameter group $T_{c}(t)$ is as follows: Glue the rightmost
endpoint of the interval $I_{-}$ starting at $-\infty$ to the leftmost
endpoint in the interval $I_{+}$ out to $+\infty$. These two finite
end-points are merged onto a single point, say $0$, on $\mathbb{R}$
(the whole real line.) This way, the one-parameter group $T_{c}(t)$
becomes a summand of $U_{B}(t)$. $T_{c}(t)$ is just translation
in $L^{2}(\mathbb{R})$ modulo a phase factor at $x=0$.

There is subtlety: Indeed, $d/dx$ as a skew Hermitian operator in
$L^{2}$ of the separate infinite half-lines has deficiency indices
$(1,0)$ or $(0,1)$. Hence no selfadjoint extensions (when a half-line
is taken by itself.) It is only via the splicing of the two infinite
half-lines that one creates a unitary one-parameter group. In summary,
the orthogonal sum of $T_{p}(t)$ and $T_{c}(t)$ is $U_{B}(t)$.
\end{rem}

\begin{rem}
The conclusion illustrated in Figure \ref{fig:w0-1} holds \emph{mutatis
mutandis} with more than three intervals. 

Indeed, the case $n>2$ is covered in Proposition \ref{prop:n}. The
modification of Figure \ref{fig:w0-1} for this case, i.e., $n>2$
is as follows:

\begin{figure}[H]
\includegraphics[scale=0.85]{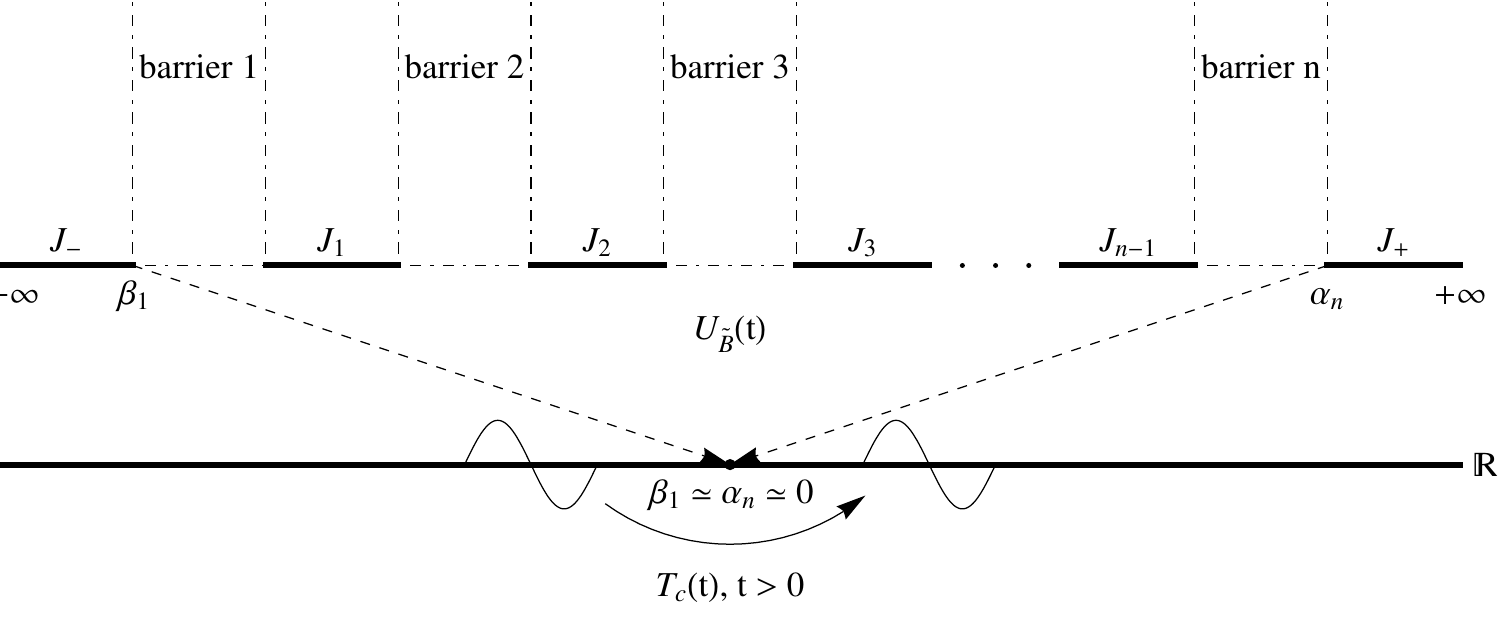}

\caption{The case $n>2$.}

\end{figure}

\end{rem}
Below, we consider the subset in $U(2)$ given by $0<w(B)\leq1$,
but it is of interest to isolate the subfamily specified by $w(B)=1$. 

But by contrast with the case $n=2$ in Fig \ref{fig:w0-1}, note
that now the $\tilde{B}$-part ($\tilde{B}\in U(n-1)$) in the orthogonal
splitting
\[
U_{B}(t)\cong U_{\tilde{B}}(t)\oplus T_{c}(t),\: t\in\mathbb{R}
\]
in 
\[
L^{2}(\Omega)\cong L^{2}(\bigcup{}_{i=1}^{n-1}J_{i})\oplus L^{2}(\mathbb{R})
\]
allows for a rich variety of inequivalent unitary one-parameter groups
$U_{\tilde{B}}(t)$. The case $L^{2}(J_{1}\cup J_{2})$ is covered
in \cite{FeJoPed1}.
\begin{thm}[$w=1$]
\label{thm:sp-w1}Choose a boundary matrix $B\in U(2)$ with parameters
$w,\theta,\phi,\psi$ as in (\ref{eq:2-by-2 Unitary}), and let $P_{B}$
be the corresponding selfadjoint restriction of $P$. For $w=1$,
the generalized eigenfunction is a multiple of 
\begin{equation}
\psi_{\lambda}=\left(e(\phi+\lambda)\chi_{-}+\chi_{0}+e(\phi-\theta-(\beta-\alpha)\lambda)\chi_{+}\right)e_{\lambda}\label{eq:eigen-w1}
\end{equation}
for any $\lambda\in\mathbb{R}$. In particular, the spectrum of $P_{B}$
is $\mathbb{R}$ with uniform multiplicity equal to one. \end{thm}
\begin{proof}
The statement follows from Lemma \ref{lem:multiplicity-0<w} and Theorem
\ref{thm:DirectIntegral-0<w} by setting $w=1.$ \end{proof}
\begin{rem}
\label{rem:w1}For $w=1$, the unitary one-parameter group $U_{B}(t)$
generated by $P_{B}$ is characterized by the phase transitions from
$0$ to $1$, and from $\alpha$ to $\beta$; see Figure \ref{fig:w1}.
Specifically, glue the rightmost endpoint of the interval $I_{-}$
starting at $-\infty$ to the left endpoint in the interval $I_{0}$;
meanwhile, glue the right endpoint in $I_{0}$ to the left endpoint
of the interval $I_{+}$ out to $+\infty$. This way, $U_{B}(t)$
is just translation in $L^{2}(\mathbb{R})$ modulo two phase factors
(see (\ref{eq:w1})) at $x=0$ and $x=\alpha$, respectively.
\end{rem}

\subsection{Generic Case }

Fix $I_{1}$, $I_{2}$, and let $\Omega=I_{-}\cup I_{0}\cup I_{+}$
be the exterior domain as before. Meanwhile, it is convenient to consider
$\mathbb{R}\backslash\{1,\alpha\}$, i.e., the union of three components
\begin{equation}
J_{-}:=(-\infty,1),\: J_{0}=(1,\alpha),\: J_{+}:=(\alpha,\infty)\label{eq:Omega-1-1}
\end{equation}

\begin{thm}[$0<w<1$]
\label{thm:sp-w} Choose $B\in U(2)$ with parameters $w,\theta,\phi,\psi$.
Let $P_{B}$ be the corresponding selfadjoint extension. For $0<w<1$,
the generalized eigenfunction is a multiple of 
\begin{equation}
\psi_{\lambda}:=\left(a(\lambda)\chi_{-}+\chi_{0}+c(\lambda)\chi_{+}\right)e_{\lambda}\label{eq:eigen-w}
\end{equation}
for any $\lambda\in\mathbb{R}$, where

In particular, the spectrum is $\mathbb{R}$ with uniform multiplicity
equal to one. 
\end{thm}
(We stress that all three systems (\ref{eq:eigen-w})-(\ref{eq:c(lambda)})
depend on the chosen $B\in U(2)$, so $\psi_{\lambda}^{(B)}$, $a_{B}(\lambda)$,
and $c_{B}(\lambda)$, but the variable $B$ will be suppressed on
occasion.)
\begin{proof}
The statement follows from Lemma \ref{lem:multiplicity-0<w} and Theorem
\ref{thm:DirectIntegral-0<w}.
\end{proof}

Rewrite (\ref{eq:a(lambda)}) and (\ref{eq:c(lambda)}) as 
\begin{alignat}{1}
a(\lambda) & =w^{-1}e(\phi)e(\lambda)H(\lambda)^{-1}\label{eq:a(lambda)-1}\\
c(\lambda) & =w^{-1}e(\phi-\theta)e(-(\beta-\alpha)\lambda)\overline{H(\lambda)}^{-1}\label{eq:c(lambda)-1}
\end{alignat}
where
\begin{equation}
H(\lambda):=\frac{1}{1-\sqrt{1-w^{2}}e(-\psi+(\alpha-1)\lambda)}.\label{eq:H(lambda)}
\end{equation}
By assumption, $0<w<1$, so that

\begin{alignat}{1}
a^{-1}(\lambda) & =w\: e(-\phi)\sum_{n=0}^{\infty}\left(1-w^{2}\right)^{\frac{n}{2}}e(-\lambda-n\psi+n(\alpha-1)\lambda)\label{eq:a-inv(lambda)}\\
c^{-1}(\lambda) & =w\: e(\theta-\phi)\sum_{n=0}^{\infty}\left(1-w^{2}\right)^{\frac{n}{2}}e((\beta-\alpha)\lambda+n\psi-n(\alpha-1)\lambda).\label{eq:c-inv(lambda)}
\end{alignat}

\begin{rem}
$e(-(\alpha-1)\lambda)H(\lambda)$ is the the transfer function for
the feedback component in Figure \ref{fig:forward}. Note the RHS
of (\ref{eq:a-inv(lambda)}), (\ref{eq:c-inv(lambda)}) are the corresponding
Fourier series expansions.
\end{rem}
\begin{figure}[H]
\begin{tabular}{cc}
\setlength{\unitlength}{0.9cm}
\begin{picture}(6,0.1)

\thicklines
\put(0,0){\line(1,0){1}}
\put(2,0){\line(3,0){1}}
\put(4,0){\line(5,0){1}}


\put(-0.4,-0.5){$-\infty$}
\put(0.8,-0.5){$0$}

\put(1.8,-0.5){$1$}
\put(2.8,-0.5){$\alpha$}

\put(3.8,-0.5){$\beta$}
\put(4.8,-0.5){$\infty$}




\end{picture}
\vspace{0.5cm} & \setlength{\unitlength}{0.9cm}
\begin{picture}(5,0.1)

\thicklines
\put(0,0){\line(1,0){1.96}}
\put(2.04,0){\line(1,0){0.91}}
\put(3.04,0){\line(1,0){1.96}}

\put(2,0){\circle{0.1}}
\put(3,0){\circle{0.1}}

\put(-0.4,-0.5){$-\infty$}
\put(1.9,-0.5){$1$}
\put(2.9,-0.5){$\alpha$}
\put(4.8,-0.5){$\infty$}

\end{picture}
\vspace{0.5cm}\tabularnewline
(a) $\Omega=I_{-}\cup I_{0}\cup I_{+}$ & (b) $\mathbb{R}\backslash\{1,\alpha\}=J_{-}\cup J_{0}\cup J_{+}$\tabularnewline
\end{tabular}

\caption{\label{fig:splitting}The exterior domain $\Omega$ and the splitting
of $\mathbb{R}\backslash\{1,\alpha\}$}

\end{figure}

\begin{rem}
Let $0<w<1$. Note that for $\lambda$ fixed, the function $x\mapsto\psi_{\lambda}^{(B)}(x)$
is not in $L^{2}(\Omega)$, see (\ref{eq:eigen-w}); hence generalized
eigenfunctions. Nonetheless for every finite interval, $l_{1}<\lambda<l_{2}$,
the ``wave packet'': $x\mapsto\int_{l_{1}}^{l_{2}}\psi_{\lambda}(x)d\lambda$
is in $L^{2}(\Omega)$. The role of generalized eigenfunctions here
is consistent with Heisenberg's uncertainty principle. 
\end{rem}
Below, we write $\wedge$ for Fourier transform, and $\vee$ for inverse
Fourier transform.
\begin{cor}
\label{cor:I-J}Let $\Omega=I_{-}\cup I_{0}\cup I_{+}$ be the exterior
domain, and let $J_{-}$, $J_{0}$, and $J_{+}$ be as in (\ref{eq:Omega-1-1}).
See Figure \ref{fig:splitting}. Then
\begin{enumerate}
\item $f\mapsto(a^{-1}\hat{f})^{\vee}$ is an isometric isomorphism from
$L^{2}(I_{-})$ onto $L^{2}(J_{-})$; 
\item $f\mapsto(c^{-1}\hat{f})^{\vee}$ is an isometric isomorphism from
$L^{2}(I_{+})$ onto $L^{2}(J_{+})$; 
\item $f\mapsto((a^{-1}c)\hat{f})^{\vee}$is an isometric isomorphism from
$L^{2}(I_{-})$ onto $L^{2}(-\infty,\beta)$.
\end{enumerate}
\end{cor}
\begin{proof}
Let $f\in L^{2}(I_{-})$. By (\ref{eq:a-inv(lambda)}), 
\[
(a^{-1}\hat{f})^{\vee}(x)=w\: e(-\phi)\sum_{n=0}^{\infty}\left(1-w^{2}\right)^{\frac{n}{2}}e(-n\psi)f(x-1+n(\alpha-1)\lambda);
\]
hence $(a^{-1}\hat{f})^{\vee}\in L^{2}(J_{-})$, where $J_{-}=(-\infty,\alpha)$.
This proves part (1). Part (2) is similar. 

Now, set $g:=(a^{-1}\hat{f})^{\vee}\in L^{2}(J_{-})$, where $f\in L^{2}(I_{-})$
as before. By (\ref{eq:c(lambda)}), we have 
\begin{alignat*}{1}
((a^{-1}c)\hat{f})^{\vee}(x) & =(c\hat{g})^{\vee}(x)\\
 & =w^{-1}e(\phi-\theta)\left(g(x-(\beta-\alpha))-\sqrt{1-w^{2}}e(\psi)g(x-(\beta-1)\lambda)\right);
\end{alignat*}
it follows that $((a^{-1}c)\hat{f})^{\vee}\in L^{2}(-\infty,\beta)$.
Thus, part (3) is true.
\end{proof}
Note the coefficients $a,c$ have equal modulus, and we define
\begin{equation}
m_{B}(\lambda):=\left|a_{B}(\lambda)\right|=\left|c_{B}(\lambda)\right|\label{eq:m(lambda)}
\end{equation}
for all $\lambda\in\mathbb{R}$.
\begin{lem}
\label{lem:m}Let $m(\lambda)$ be as in (\ref{eq:m(lambda)}). 
\begin{enumerate}
\item The following estimate holds:
\begin{equation}
\frac{w}{2}\leq m_{B}(\lambda)\leq\frac{2}{w}\label{eq:m-bound}
\end{equation}
In particular, the Fourier multiplier $m(\lambda)$ is strictly positive,
bounded, and invertible. 
\item Setting $z:=e(-\psi+(\alpha-1)\lambda)$, then $m_{B}^{-2}(\cdot)$
has Fourier series expansion 
\begin{equation}
m_{B}^{-2}(z)=\sum_{k=-\infty}^{\infty}\left(1-w^{2}\right)^{\frac{\vert k\vert}{2}}z^{k}.\label{eq:m-2}
\end{equation}

\end{enumerate}
\end{lem}
\begin{proof}
(1) By (\ref{eq:a(lambda)-1}) and (\ref{eq:m(lambda)}) 
\[
m_{B}(\lambda)=\left|a(\lambda)\right|=\frac{1}{w}\left|1-\sqrt{1-w^{2}}\: e(-\psi+(\alpha-1)\lambda)\right|;
\]
hence 
\[
\frac{w}{2}\leq\frac{1}{w}(1-(1-\frac{1}{2}w^{2}))\leq m_{B}(\lambda)\leq\frac{1}{w}(1+\sqrt{1-w^{2}})\leq\frac{2}{w}.
\]

(2) From (\ref{eq:a(lambda)-1}), we have 
\begin{alignat*}{1}
m_{B}^{-2} & (z)=w^{2}H(z)\overline{H(z)}\\
 & =w^{2}\left(\sum_{l=0}^{\infty}\left(1-w^{2}\right)^{\frac{l}{2}}z^{l}\right)\left(\sum_{n=0}^{\infty}\left(1-w^{2}\right)^{\frac{n}{2}}z^{-n}\right)\\
 & =w^{2}\sum_{k=-\infty}^{\infty}\left(\sum_{n=0}^{\infty}\left(1-w^{2}\right)^{\frac{2n+\vert k\vert}{2}}\right)z^{k}\\
 & =\sum_{k=-\infty}^{\infty}\left(1-w^{2}\right)^{\frac{\vert k\vert}{2}}z^{k}.
\end{alignat*}

\end{proof}

\begin{cor}
Fix $B=\left(\begin{array}{cc}
\overline{a} & -b\\
\overline{b} & a
\end{array}\right)\in SU(2)$, $a\neq0$, then $\mathbb{R}\ni\lambda\mapsto m_{B}^{-2}(\lambda)$
is periodic with period $(\alpha-1)^{-1}$, and the integral over
a period is
\begin{equation}
\int_{0}^{(\alpha-1)^{-1}}m_{B}^{-2}(\lambda)d\lambda=\frac{1}{\alpha-1}.\label{eq:sigB-1}
\end{equation}
In particular, for every subset $J\subset\mathbb{R}$ of length $(\alpha-1)^{-1}$,
we have
\begin{equation}
\sigma_{B}(J)=\frac{1}{\alpha-1}.\label{eq:sigB}
\end{equation}
\end{cor}
\begin{proof}
This follows directly from (\ref{eq:a-inv(lambda)}), $\sigma_{B}(\lambda)=m_{B}^{-2}(\lambda)d\lambda$,
and
\[
\left|\boldsymbol{a}(B,\lambda)\right|^{-2}=m_{B}^{-2}(\lambda),\:\lambda\in\mathbb{R}.
\]
As a result, we may apply Parseval's identity to $\lambda\mapsto\boldsymbol{a}(B,\lambda)^{-1}$
over a period-interval in $\lambda$.\end{proof}
\begin{cor}
\label{cor:poisson}Fix $B$ in $SU(2)$. On a period interval (in
$\lambda$), the function $m_{B}^{-2}(\lambda)$ is a Poisson kernel.
In the complex coordinates $a$ and $b$, see Remark \ref{rem:UB},
i.e., for $B(a,b)$ in $SU(2)$, the radial variable in the $B$-Poisson
kernel is $\left|b\right|$. \end{cor}
\begin{proof}
The proof is immediate from Lemma \ref{lem:m} (2). See also Corollary
\ref{cor:su2} (2) below.\end{proof}
\begin{rem}[The Poisson-kernel]
 For $b\in\mathbb{C}$, $\left|b\right|<1$, $b=\left|b\right|e(-\psi)$,
and $\left|b\right|=\sqrt{1-w^{2}}$, and recall $B=\left(\begin{array}{cc}
\overline{a} & -b\\
\overline{b} & a
\end{array}\right)\in SU(2)$. Set
\begin{equation}
P_{b}(\lambda)=\frac{1-\left|b\right|^{2}}{1-2\left|b\right|\cos\left(2\pi\left(\left(\alpha-1\right)\lambda-\psi\right)\right)+\left|b\right|^{2}}.\label{eq:possion}
\end{equation}
Hence, $P_{b}(\lambda)=m_{B}^{-2}(\lambda)$, $\lambda\in\mathbb{R}$.
Let $J$ be a period interval (see (\ref{eq:sigB-1})-(\ref{eq:sigB}))
and let $f\in L^{2}(J)$, then the Poisson-kernel in (\ref{eq:possion})
defined a harmonic extension $F$ as follows:

Wrap the period-interval $J$ around the unit circle in $\mathbb{C}$,
and make the identification 
\begin{equation}
f(\lambda)\simeq f(e(\lambda)),\:\lambda\in J.\label{eq:tmp-37}
\end{equation}
Then
\begin{equation}
F(b)=P_{b}[f]=\int_{J}f(\lambda)P_{b}(\lambda)d\lambda\label{eq:hext}
\end{equation}
is a representation of the harmonic extension; see \cite{DyMc72}.
\end{rem}

\subsection{Isometries}

Let $L^{2}(\sigma_{B})$ be the Hilbert space of $L^{2}$-functions
on $\mathbb{R}$ with respect to the Borel measure 
\begin{equation}
\sigma_{B}(d\lambda):=m_{B}^{-2}(\lambda)d\lambda.\label{eq:sigma(B)}
\end{equation}
Here, $d\lambda$ on the right in (\ref{eq:sigma(B)}) is the Lebesgue
measure on $\mathbb{R}^{1}$. 

Define $V_{B}:L^{2}(\Omega)\to L^{2}(\sigma_{B})$ by 
\begin{equation}
\left(V_{B}f\right)(\lambda):=\left\langle \psi_{\lambda}^{(B)},f\right\rangle =\int_{\Omega}\overline{\psi_{\lambda}^{(B)}(x)}f(x)dx\label{eq:V}
\end{equation}
for all $f\in L^{2}(\Omega)$. The adjoint operator $V_{B}^{*}:L^{2}(\sigma_{B})\rightarrow L^{2}(\Omega)$
is given by
\begin{equation}
\left(V_{B}^{*}g\right)(x)=\int_{\mathbb{R}}g(\lambda)\psi_{\lambda}^{(B)}(x)\:\sigma_{B}(d\lambda)\label{eq:V-adj}
\end{equation}
for all $g\in L^{2}(\sigma_{B})$.

Note that, by (\ref{eq:eigen-w})-(\ref{eq:c(lambda)}), the generalized
eigenfunctions $\psi_{\lambda}^{(B)}$ depends on $B$ from $U(2)$;
and as a result the transforms $V_{B}$ and $V_{B}^{*}$ depend on
$B$ as well.

We now spell out for every $U_{B}(t)$, $w>0$, an explicit spectral
representation:
\begin{cor}
Let $d\sigma_{B}(\cdot)$ be the measure in (\ref{eq:sigma(B)}) and
let $V_{B}:L^{2}(\Omega)\rightarrow L^{2}(\mathbb{R},\sigma_{B})$
be the spectral transform in (\ref{eq:V}) with adjoint operator $V_{B}^{*}:L^{2}(\mathbb{R},\sigma_{B})\rightarrow L^{2}(\Omega)$.
Then 
\begin{align*}
V_{B}V_{B}^{*} & =I_{L^{2}(\sigma_{B})}\;\mbox{ and }\\
V_{B}^{*}V_{B} & =I_{L^{2}(\Omega)}.
\end{align*}
Moreover,
\begin{equation}
V_{B}U_{B}(t)V_{B}^{*}=M_{t}\label{eq:Mt}
\end{equation}
where $M_{t}$ is the unitary one-parameter group acting on $L^{2}(\mathbb{R},\sigma_{B})$
as follows 
\[
\left(M_{t}g\right)(\lambda)=e_{\lambda}(-t)g(\lambda)
\]
for all $t,\lambda\in\mathbb{R}$, and all $g\in L^{2}(\mathbb{R},\sigma_{B})$.
\end{cor}
Let $e_{\xi}(x):=e^{i2\pi\xi x}.$ Following \cite{Ped87,JP99} we
say that a measurable set $\Omega$ is a \emph{spectral set} if there
is a positive Borel measure $\mu$ such that the map 
\begin{equation}
\mathscr{F}_{\Omega}:f\to\widehat{f}(\xi):=\int_{\Omega}f(x)e_{\xi}(x)dx\label{eq:F-3}
\end{equation}
 is an surjective isometry $L^{2}(\Omega)\to L^{2}(\mu).$ In the
affirmative case we say $\left(\Omega,\mu\right)$ is a \emph{spectral
pair}. 

Below we consider the case where the measure $\mu$ has atoms, i.e.,
points $\xi\in\mathbb{R}$ such that $\mu(\{\xi\})>0$.
\begin{lem}
\label{lem:points}If there is a point $\xi_{0}$ such that $\mu(\{\xi_{0}\})>0,$
then $\mu$ is discrete and $\xi\to\mu(\{\xi\})$ is constant on the
support of $\mu.$ \end{lem}
\begin{proof}
By Lemma 1.6 of \cite{Ped87}, if $K$ is compact, then $\mu(K)<\infty.$
By Corrollary 5 of \cite{JP99}, we then get $\mu(\{\xi\})=\mu(\{\xi_{0}\})$
for all points $\xi$ in the support of $\mu.$ \end{proof}
\begin{prop}
There is no unitary $2\times2$ matrix $B$ such that $\left(\Omega,\sigma_{B}\right)$
is a spectral pair. \end{prop}
\begin{proof}
As a consequence of Lemma \ref{lem:points}, if a measure $\mu,$
contains a mixture of atoms and Lebesgue spectrum then $(\Omega,\mu)$
is not a spectral pair. That is, no $B$ with $w=0$ gives a spectral
pair. 

Suppose $0<w\leq1$, and the other entries in $B$ are chosen such
that $\left(\Omega,\sigma_{B}\right)$ is a spectral pair. By \cite{Ped87}
the generalized eigenfunctions are $e_{\lambda},\lambda\in\mathbb{R}.$
Hence, it follows from (\ref{eq:a(lambda)}) that $\alpha=0$, contradicting
$1<\alpha.$ \end{proof}
\begin{rem}
Our results below shows that when a revised spectral transform $V$
in $L^{2}(\Omega)$ is used, taking scattering into consideration,
then via this transform $V$ in (\ref{eq:V}), we do have a spectral
pair, a $V$-spectral pair. And, moreover, the spectral density measure
$\sigma_{B}$ computed from $V$ is purely non-atomic. Moreover, $\sigma_{B}$
in (\ref{eq:sigma(B)}) is absolutely continuous with respect to Lebesgue
measure; see (\ref{eq:Decomposition-Pf}) and the details in Theorem
\ref{thm:sp-w-1}. In other words, in the theorem below, we use $V$
in place of $\mathscr{F}_{\Omega}$ from eq (\ref{eq:F-3}).

For comparison, in \cite{FeJoPed1} we studied the complementary case
when $\Omega$ is instead taken as the union of two finite and disjoint
intervals. In this case, there are some configurations which yield
spectral pairs in the sense of \cite{JP99}, and moreover the measures
$\mu$ that arise there are purely discrete.\end{rem}
\begin{thm}
\label{thm:sp-w-1}Fix $B=B(w,\theta,\phi,\psi)\in U(2)$, with $0<w<1$.
Then $V$ in (\ref{eq:V}) is a unitary operator from $L^{2}(\Omega)$
onto $L^{2}(\sigma_{B})$. In particular, 
\begin{equation}
f(x)=\int\left\langle \psi_{\lambda},f\right\rangle _{\Omega}\psi_{\lambda}(x)\:\sigma_{B}(d\lambda),\;\mbox{and}\label{eq:decomp}
\end{equation}
\[
\int_{\Omega}\left|f(x)\right|^{2}dx=\int_{\mathbb{R}}\left|\left\langle \psi_{\lambda},f\right\rangle _{\Omega}\right|^{2}\sigma_{B}(d\lambda)
\]
for all $f\in L^{2}(\Omega)$. 

Here $\{\psi_{\lambda}(\cdot)\}$ is the family of functions in (\ref{eq:eigen-w})
and (\ref{eq:V}). Moreover, the extension operator $P_{B}$ satisfies
\begin{alignat}{1}
P_{B}f(x) & =\int\left\langle \psi_{\lambda},Pf\right\rangle \psi_{\lambda}(x)\sigma_{B}(d\lambda)\nonumber \\
 & =\int\lambda\left\langle \psi_{\lambda},f\right\rangle \psi_{\lambda}(x)\sigma_{B}(d\lambda)\label{eq:Decomposition-Pf}
\end{alignat}
for all $f\in\mathscr{D}(P_{B})$.\end{thm}
\begin{proof}
For convergence of the integral on the RHS in (\ref{eq:Decomposition-Pf}),
we refer to the theory of direct integral decompositions; see e.g.,
\cite{MM63}, and \cite{Sto90}. 

For all $f\in L^{2}(\Omega)$, write $f=f_{-}+f_{0}+f_{+}$, where
$f_{-}:=\chi_{-}f$, $f_{0}:=\chi_{0}f$, and $f_{+}:=\chi_{+}f$.
Then 
\begin{alignat}{1}
\left(V_{B}f\right)(\lambda)=\left\langle \psi_{\lambda},f\right\rangle  & =\int(\overline{a(\lambda)}\chi_{-}+\chi_{0}+\overline{c(\lambda)}\chi_{+})e_{-\lambda}f\nonumber \\
 & =\int(\overline{a(\lambda)}f_{-}+f_{0}+\overline{c(\lambda)}f_{+})e_{-\lambda}\nonumber \\
 & =\overline{a(\lambda)}\hat{f}_{-}(\lambda)+\hat{f}_{0}(\lambda)+\overline{c(\lambda)}\hat{f}_{+}(\lambda).\label{eq:tmp-9}
\end{alignat}
Now,
\begin{alignat*}{1}
\left\Vert Vf_{-}\right\Vert _{L^{2}(\sigma)}^{2} & =\int\left|\overline{a(\lambda)}\hat{f}_{-}(\lambda)\right|^{2}m^{-2}(\lambda)d\lambda\\
 & =\int\left|\hat{f}_{-}(\lambda)\right|^{2}d\lambda=\left\Vert f_{-}\right\Vert _{L^{2}(\Omega)}^{2}
\end{alignat*}
i.e., $V$ is isometric on $L^{2}(I_{-})$. Similarly, we can readily
check that $V$ is isometric on $L^{2}(I_{+})$. On the other hand,
by Lemma \ref{lem:m}, 
\begin{alignat}{1}
\left\Vert Vf_{0}\right\Vert _{L^{2}(\sigma)}^{2} & =\int\left|\hat{f}_{0}(\lambda)\right|^{2}m^{-2}(\lambda)d\lambda\nonumber \\
 & =\sum_{k=-\infty}^{\infty}\left(1-w^{2}\right)^{\frac{\left|k\right|}{2}}e(-k\psi)\int\left|\hat{f}_{0}(\lambda)\right|^{2}e(k(\alpha-1)\lambda)d\lambda\nonumber \\
 & =\sum_{k=-\infty}^{\infty}\left(1-w^{2}\right)^{\frac{\left|k\right|}{2}}e(-k\psi)\:\varphi(k(\alpha-1))\label{eq:tmp-12}
\end{alignat}
where $\varphi(x):=f_{0}(x)*\overline{f_{0}(-x)}$. Note that $supp(\varphi)\subset\left[-\left|I_{0}\right|,\left|I_{0}\right|\right]$,
and $\varphi$ vanishes on the boundary points $\pm(\alpha-1)$. Thus,
the only non-zero term in (\ref{eq:tmp-12}) is when $k=0$; it follows
that 
\[
\left\Vert Vf_{0}\right\Vert _{L^{2}(\sigma_{B})}^{2}=\varphi(0)=\int\left|\hat{f}_{0}(\lambda)\right|^{2}d\lambda=\Vert f_{0}\Vert_{L^{2}(\Omega)}^{2}.
\]
That is, $V$ is isometric on $L^{2}(I_{0})$. 

For all $f,g\in L^{2}(\Omega)$, 
\begin{alignat}{1}
\left\langle Vf,Vg\right\rangle _{L^{2}(\sigma)} & =\left\langle V\left(f_{-}+f_{0}+f_{+}\right),V\left(g_{-}+g_{0}+g_{+}\right)\right\rangle _{L^{2}(\sigma)}\nonumber \\
 & =\left\langle f_{-},g_{-}\right\rangle _{L^{2}(\Omega)}+\left\langle f_{0},g_{0}\right\rangle _{L^{2}(\Omega)}+\left\langle f_{+},g_{+}\right\rangle _{L^{2}(\Omega)}+\mbox{cross terms};\label{eq:tmp-20}
\end{alignat}
where the cross terms are given by
\begin{flalign}
\mbox{cross terms}= & \left\langle \overline{a}\hat{f}_{-},\hat{g}_{0}\right\rangle _{L^{2}(\sigma)}+\left\langle \overline{a}\hat{f}_{-},\overline{c}\hat{g}_{+}\right\rangle _{L^{2}(\sigma)}\nonumber \\
 & +\left\langle \hat{f}_{0},\overline{a}\hat{g}_{-}\right\rangle _{L^{2}(\sigma)}+\left\langle \hat{f}_{0},\overline{c}\hat{g}_{+}\right\rangle _{L^{2}(\sigma)}\nonumber \\
 & +\left\langle \overline{c}\hat{f}_{+},\overline{a}\hat{g}_{-}\right\rangle _{L^{2}(\sigma)}+\left\langle \overline{c}\hat{f}_{+},\hat{g}_{0}\right\rangle _{L^{2}(\sigma)}.\label{eq:tmp-18}
\end{flalign}
Since $\sigma_{B}(d\lambda)=m^{-2}(\lambda)d\lambda$, we see that
(\ref{eq:tmp-18}) can be written as, after dividing out $m^{-2}(\lambda)$
inside the inner product $\left\langle \cdot,\cdot\right\rangle _{L^{2}(\sigma)}$,
\begin{flalign}
\mbox{cross terms}= & \left\langle a^{-1}\hat{f}_{-},\hat{g}_{0}\right\rangle _{L^{2}(\hat{\mathbb{R}})}+\left\langle a^{-1}c\hat{f}_{-},\hat{g}_{+}\right\rangle _{L^{2}(\hat{\mathbb{R}})}\nonumber \\
 & +\left\langle \hat{f}_{0},a^{-1}\hat{g}_{-}\right\rangle _{L^{2}(\hat{\mathbb{R}})}+\left\langle \hat{f}_{0},c^{-1}\hat{g}_{+}\right\rangle _{L^{2}(\hat{\mathbb{R}})}\nonumber \\
 & +\left\langle \hat{f}_{+},a^{-1}c\hat{g}_{-}\right\rangle _{L^{2}(\hat{\mathbb{R}})}+\left\langle c^{-1}\hat{f}_{+},\hat{g}_{0}\right\rangle _{L^{2}(\hat{\mathbb{R}})}.\label{eq:tmp-19}
\end{flalign}
By Corollary \ref{cor:I-J}, each term in (\ref{eq:tmp-19}) vanishes.
Hence, by (\ref{eq:tmp-20}), 
\[
\left\langle Vf,Vg\right\rangle _{L^{2}(\sigma)}=\left\langle f,g\right\rangle _{L^{2}(\Omega)}
\]
for all $f,g\in L^{2}(\Omega)$. We conclude that $V$ is an isometry,
i.e., $V^{*}V=I$; and (\ref{eq:decomp}) holds.

Next, we show that $V$ is surjective. It suffices to show the range
of $V$ is $L^{2}(\hat{\mathbb{R}})$, as the Fourier multiplier $m^{-2}(\lambda)$
is positive, invertible and bounded away from $0$; see Lemma \ref{lem:m}.
Suppose $\hat{g}\in L^{2}(\hat{\mathbb{R}})$, such that
\[
\int\overline{\hat{g}(\lambda)}\left(Vf\right)(\lambda)d\lambda=0
\]
for all $f\in L^{2}(\Omega)$. That is, by (\ref{eq:tmp-9}), 
\[
\int\overline{a(\lambda)\hat{g}(\lambda)}\hat{f}_{-}(\lambda)d\lambda+\int\overline{\hat{g}(\lambda)}\hat{f}_{0}(\lambda)d\lambda+\int\overline{c(\lambda)\hat{g}(\lambda)}\hat{f}_{+}(\lambda)d\lambda=0
\]
for all $f=f_{-}+f_{0}+f_{+}$ in $L^{2}(\Omega)$. This is true if
and only if 
\[
\chi_{-}(a(\lambda)\hat{g}(\lambda))^{\vee}=\chi_{0}g=\chi_{+}(c(\lambda)\hat{g}(\lambda))^{\vee}=0.
\]
In particular, $g$ vanishes on $I_{0}$. 

If $supp(g)\subset J_{-}$, then by (\ref{eq:c(lambda)-1}), $(c\hat{g})^{\vee}$
is supported in $(-\infty,\beta]$, and so $\chi_{+}(c\hat{g})^{\vee}=0$.
By Corollary \ref{cor:I-J}, the mapping $g\mapsto(a\hat{g})^{\vee}$
is a bijection from $L^{2}(J_{-})$ onto $L^{2}(I_{-})$. Thus, $\chi_{-}(a\hat{g})^{\vee}=0$
implies $g=0$. Similarly, $supp(g)\subset J_{+}$ implies $g=0$.
Hence, $g$ ($\hat{g}$) is identically zero. Consequently, $V$ is
onto.

It remains to establish (\ref{eq:Decomposition-Pf}). Let $f$ be
in the domain of $P_{B}.$ By (\ref{eq:decomp}) 
\[
P_{B}f(x)=\int\left\langle \psi_{\lambda},Pf\right\rangle \psi_{\lambda}(x)\sigma_{B}(d\lambda)
\]
 hence we just need to establish that 
\[
\left\langle \psi_{\lambda},Pf\right\rangle =\lambda\left\langle \psi_{\lambda},f\right\rangle .
\]
But this follows by integration by parts since both $\psi_{\lambda}$
and $f$ satisfy the boundary conditions $B\rho_{1}(\cdot)=\rho_{2}(\cdot).$
This proves (\ref{eq:Decomposition-Pf}).
\end{proof}

\begin{rem}
By (\ref{eq:sigma(B)}) the measures $\sigma_{B}$ all are mutally
absolutely continuous. Hence, it follows from Theorem \ref{thm:sp-w-1}
that the operators $P_{B},$ $B$ the unitary $2\times2$ matrices
parametrized as in (\ref{eq:2-by-2 Unitary}) with $w\neq0,$ are
all pairwise unitarily equivalent equivalent.
\begin{rem}
For $w=0$, as we see in Remark \ref{rem:w1} that the unitary one-parameter
group $U_{B}(t)$, acting on $L^{2}(\Omega)$, is the usual translation
by $t$ in $L^{2}(\mathbb{R})$ modulo two phase factors at $x=0,\alpha$. 
\end{rem}
If, in addition, $\theta=\phi=\psi=0$, i.e., $B$ is the identity
matrix in $U(2)$, then the generalized eigenfunction is specified
by (see (\ref{eq:eigen-w1})) 
\[
\psi_{\lambda}(x)=\left(e(\lambda)\chi_{-}(x)+\chi_{0}(x)+e(-(\beta-\alpha)\lambda)\chi_{+}(x)\right)e_{\lambda}(x),\:\lambda\in\mathbb{R};
\]
and for the measure $\sigma_{B}$ ($B=I$), we get $\sigma_{I}(d\lambda)=d\lambda$.
Moreover, $V_{I}:L^{2}(\Omega)\rightarrow L^{2}(\mathbb{R})$ is given
by 
\begin{alignat*}{1}
\left(V_{I}f\right)(\lambda) & =e(-\lambda)\hat{f}_{-}(\lambda)+\hat{f}_{0}(\lambda)+e((\beta-\alpha)\lambda)\hat{f}_{+}(\lambda)\\
 & =\left(f_{-}(\cdot-1)\right)^{\wedge}(\lambda)+\hat{f}_{0}(\lambda)+\left(f_{+}(\cdot+(\beta-\alpha))\right)^{\wedge}(\lambda)\\
 & =\left(f_{-}(\cdot-1)+f_{0}(\cdot)+f_{+}(\cdot+(\beta-\alpha))\right)^{\wedge}(\lambda).
\end{alignat*}
In this case, $U_{B}(t)$, acting on $L^{2}(\Omega)$, is unitarily
equivalent to the unitary group $T(t)$ of translation by $t$ in
$L^{2}(\mathbb{R})$. 
\end{rem}
For more information about the geometric significance of the vanishing
cross-terms, we refer to section \ref{sec:vc} below.
\begin{cor}
Let $B\in U(2)$, $P_{B}$, $a=a_{B}$, $c=c_{B}$, and $\sigma_{B}(\cdot)$
be as above. Let $P_{\pm}$ and $P_{0}$ be the projections in $L^{2}(\Omega)$
corresponding to the intervals $I_{\pm}$, and $I_{0}$. Then
\begin{alignat*}{1}
\int_{\mathbb{R}}\left|a_{B}(\lambda)\left(P_{-}f\right)^{\wedge}(\lambda)\right|^{2}\sigma_{B}(d\lambda) & =\int_{I_{-}}\left|f(x)\right|^{2}dx,\\
\int_{\mathbb{R}}\left|c_{B}(\lambda)\left(P_{+}f\right)^{\wedge}(\lambda)\right|^{2}\sigma_{B}(d\lambda) & =\int_{I_{+}}\left|f(x)\right|^{2}dx,\mbox{ and}\\
\int_{\mathbb{R}}\left|\left(P_{0}f\right)^{\wedge}(\lambda)\right|^{2}\sigma_{B}(d\lambda) & =\int_{I_{0}}\left|f(x)\right|^{2}dx
\end{alignat*}
for all $f\in L^{2}(\Omega)$.\end{cor}
\begin{rem}[The generalized eigenfunctions from an ODE, and from boundary values
indexed by $U(2)$]
 Fix an element $B\in U(2)$ as above. In the course of the proof,
we saw that the field of functions $\{\psi_{\lambda}\}_{\lambda\in\mathbb{R}}$
from (\ref{eq:eigen-w}) - (\ref{eq:c(lambda)}) is a system of generalized
eigenfunctions for the selfadjoint operator $P_{B}$ in $L^{2}(\Omega)$,
where $\Omega$ is the union of the three open intervals $I_{-}$,
$I_{0}$, and $I_{+}$ in (\ref{eq:Omega-1}). 

Using Lemma \ref{lem:k-1} and Remark \ref{rmk:Pmin}, we conclude
that, for each $\lambda\in\mathbb{R}$, and in each of the three open
intervals, we get the function $\psi_{\lambda}$ as a differentiable
solution to the following ODE,
\begin{equation}
\frac{d}{dx}\psi_{\lambda}(x)=i2\pi\lambda\psi_{\lambda}(x)\label{eq:ode}
\end{equation}
with boundary conditions:
\begin{equation}
\left(\begin{array}{c}
\psi_{\lambda}(1_{+})\\
\psi_{\lambda}(\beta_{+})
\end{array}\right)=B\left(\begin{array}{c}
\psi_{\lambda}(0_{-})\\
\psi_{\lambda}(\alpha_{-})
\end{array}\right)\label{eq:tmp-13}
\end{equation}
where we used (\ref{eq:ExtensionDomain1}). Now a generalized eigenfunction
is determined only up to a constant multiple, and to fix this, we
imposed the condition
\begin{equation}
\psi_{\lambda}(1_{+})=e_{1}(\lambda),\label{eq:tmp-16}
\end{equation}
see (\ref{eq:eigen-w}). Combining (\ref{eq:ode}) - (\ref{eq:tmp-16}),
and using uniqueness of a first order ODE boundary value-problem (in
each of the three intervals), we get uniquely determined constants
$a(\lambda)$ and $c(\lambda)$ such that
\begin{alignat*}{1}
\psi_{\lambda}(x) & =a(\lambda)e^{i2\pi\lambda x},\: x\in I_{-};\\
\psi_{\lambda}(x) & =e^{i2\pi\lambda x},\: x\in I_{0};
\end{alignat*}
and 
\[
\psi_{\lambda}(x)=c(\lambda)e^{i2\pi\lambda x},\: x\in I_{+}.
\]
In other words, $\psi_{\lambda}$ has the form (\ref{eq:eigen-w})
with the two functions $a(\lambda)$ and $c(\lambda)$ determined
uniquely. As a result, (\ref{eq:a(lambda)}) and (\ref{eq:c(lambda)})
are the \uline{only} solution; hence multiplicity-one. It is well
known, see e.g., \cite{Sim82}, that that solving the generalized
eigenfunction equations may lead to to many generalized eigenfunctions.
We saw above that this is not the case in our situation. 
\end{rem}

\subsection{Limit of measures}

In this section we discuss two limit theorems for the measures $\sigma_{B}$,
indexed by $B$ in $U(2)$, arising in the spectral resolution for
the corresponding selfadjoint operators, and the unitary one-parameter
groups $U_{B}(t)$.

Modding out by the determinant of $B$, we reduce to the case of the
subgroup $SU(2)$. If $B$ in $SU(2)$ is represented in the usual
way (Remark \ref{rem:UB}) by a pair of complex numbers $a$ and $b$,
with $\left|a\right|^{2}+\left|b\right|^{2}=1$, we show that in the
limit as $a$ tends to $0$, the corresponding measure $\sigma_{B}$
bifurcate resulting in two measures, the Lebesgue measure on $\mathbb{R}$,
and the sum of the Dirac delta measures picking out the point spectrum
of the unitary one-parameter groups $U_{B}(t)$ arising in the limit;
hence accounting in a direct way for the jump in multiplicity.

Our second result is a Cesaro limit formed from a fixed unitary one-parameter
groups $U_{B}(t)$ .
\begin{cor}
\label{cor:su2}Working with $B\in SU(2)$ in the form 
\begin{equation}
B=\left(\begin{array}{cc}
\overline{a} & -b\\
\overline{b} & a
\end{array}\right),\;\left|a\right|^{2}+\left|b\right|^{2}=1,\label{eq:tmp-31}
\end{equation}
we get
\begin{equation}
\psi_{\lambda}^{(B)}=\left(\boldsymbol{a}(B,\lambda)\chi_{-}(x)+\chi_{0}(x)+\boldsymbol{c}(B,\lambda)\chi_{+}(x)\right)e_{\lambda}(x),\label{eq:tmp-35}
\end{equation}
and the following presentations:
\begin{enumerate}
\item ~
\begin{equation}
\boldsymbol{a}(B,\lambda)=\frac{e(\lambda)}{a}\left(1-b\: e((\alpha-1)\lambda)\right).\label{eq:tmp-32}
\end{equation}
Note $a=w\: e(-\phi)$, $b=\sqrt{1-w^{2}}\: e(-\psi)$, and $a\rightarrow0$
$\Longleftrightarrow$ $b\rightarrow e(-\psi)$.\end{enumerate}
\begin{enumerate}[resume]
\item Poisson-kernel representation:
\begin{equation}
m_{B}^{-2}(z)=\sum_{k\in\mathbb{Z}}\left|b\right|^{k}z^{k}=P_{b}((\alpha-1)\lambda-\psi)\label{eq:tmp-33}
\end{equation}
where $z=e(-\psi+(\alpha-1)\lambda)$; see Corollary \ref{cor:poisson}.
\item In the sense of Schwartz-distributions we get the following two limits
(\ref{enu:tmp}) $\&$ (\ref{enu:tmp-1}):
\item \label{enu:tmp}
\begin{equation}
\lim_{a\rightarrow0}m_{B}^{-2}(z)=\sum_{k\in\mathbb{Z}}e(k(-\psi+(\alpha-1)\lambda))\;\mbox{ (as a distribution),}\label{eq:tmp-36}
\end{equation}
and for the family of measures $\sigma_{B}(d\lambda)$, we get the
following limit-measure
\item Dirac-comb representation:\label{enu:tmp-1}
\begin{equation}
\lim_{a\rightarrow0}\sigma_{\left(\begin{array}{cc}
\overline{a} & -b\\
\overline{b} & a
\end{array}\right)}(d\lambda)=\sum_{n\in\mathbb{Z}}\delta_{\frac{\psi}{\alpha-1}+\frac{n}{\alpha-1}}\label{eq:tmp-34}
\end{equation}
accounting for the embedded point-spectrum inside the continuum spectrum,
and Lebesgue measure $d\lambda$. 
\end{enumerate}
\end{cor}
\begin{proof}
See Lemma \ref{lem:m}, (\ref{eq:m-2}), and Theorems \ref{thm:sp-w0}
and \ref{thm:sp-w-1}. For the theory of limit of measures, see for
example \cite{Va71}. For the use of ``Dirac combs'' in analysis,
see e.g., \cite{BDM05,BM04}.
\end{proof}
Below we show that the family of unitary one-parameter groups $U_{B}(t)$
acting on $L^{2}(\Omega)$ reduces under unitary equivalence. Nonetheless,
as we note in sections \ref{sec:UB} - \ref{sec:ss} below, unitarily
equivalent one-parameter groups $U_{B}(t)$ can have quite different
scattering properties.
\begin{cor}
The subfamily of unitary one-parameter groups $U_{B}(t)$ acting on
$L^{2}(\Omega)$ corresponding to $B$ in $U(2)$ such that $0<w(B)<1$
represent a single equivalence class under unitary equivalence .\end{cor}
\begin{proof}
It is known (see \cite{Ar09}) that two strongly continuous unitary
one-parameter groups are unitarily equivalent if and only if they
have the same spectrum, including counting multiplicity, and measure
in the corresponding spectral representation. We saw that when $0<w(B)<1$,
the spectrum is continuous in the Lebesgue class. As a result of our
computation of the measures $\sigma_{B}$ in this subfamily, we note
that any two of the measures must be mutually absolutely continuous.
As a result, all of our unitary one-parameter groups $U_{B}(t)$,
for $0<w(B)<1$, are pairwise unitarily equivalent.\end{proof}
\begin{cor}
Let $1<\alpha<\beta<\infty$ be fixed, set $\Omega=(-\infty,0)\cup(1,\alpha)\cup(\beta,\infty)$,
and let $B\in U(2)$ be chosen as in (\ref{eq:2-by-2 Unitary}), $0<w<1$.
Let $P_{B}$ be the corresponding selfadjoint operator in $L^{2}(\Omega)$. 

(i) Then the three terms in the spectral transform, $V_{B}:L^{2}(\Omega)\rightarrow L^{2}(\mathbb{R},\sigma_{B})$
are as follows:
\begin{equation}
(V_{B}f)(\lambda)=\overline{a}(\lambda)\left(P_{-}f\right)^{\wedge}(\lambda)+\left(P_{0}f\right)^{\wedge}(\lambda)+\overline{c}(\lambda)\left(P_{+}f\right)^{\wedge}(\lambda),\:\lambda\in\mathbb{R}\label{eq:VB}
\end{equation}
where $\lambda\rightarrow a(\lambda)$, and $\lambda\rightarrow c(\lambda)$
are given by (\ref{eq:a(lambda)}) and (\ref{eq:c(lambda)}); $\hat{.}$
denotes the usual $L^{2}$-Fourier transform, and $P_{-}f=\chi_{(-\infty,0)}f$,
$P_{0}f:=\chi_{(1,\alpha)}f$, and $P_{+}f:=\chi_{(\beta,\infty)}f$.\\

(ii) The first term on the RHS in (\ref{eq:VB}) is in the Hardy-space
$H_{\widehat{up}}^{2}$ of analytic functions in the upper half-plane
in $\mathbb{C}$ with $L^{2}$-boundary values on the real line; i.e.,
referring to analytic continuation in the $\lambda$-variable from
(\ref{eq:VB}).\\

(iii) The third term on the RHS in (\ref{eq:VB}) is in the Hardy-space
$H_{\widehat{down}}^{2}$ of analytic functions in the lower half-plane
in $\mathbb{C}$ with $L^{2}$-boundary values.\\

(iv) The middle term on the RHS in (\ref{eq:VB}) is in the Hilbert
space of band-limited functions with frequency band equal to the interval
$[1,\alpha]$.\end{cor}
\begin{proof}
Parts (i) - (iii) follow directly from the formulas (\ref{eq:m-2}),
(\ref{eq:a(lambda)}) and (\ref{eq:c(lambda)}) which we already derived. 

Indeed, the stated analytic continuation properties of 
\[
\lambda\mapsto\left(\chi_{-}f\right)^{\wedge}(\lambda),\mbox{ and }\:\lambda\mapsto\left(\chi_{+}f\right)^{\wedge}(\lambda)
\]
are clear. And it follows from (\ref{eq:a(lambda)}) \& (\ref{eq:c(lambda)})
that the two functions $\overline{a}$ and $\overline{c}$ in (\ref{eq:VB})
have the stated analytic continuation properties. 

Part (iv) follows from (\ref{eq:VB}) and the definition of Hilbert
spaces of band-limited functions; see e.g. \cite{DyMc72}.

The latter conclusion is important because Shannon's interpolation
formula holds for the Hilbert spaces of band-limited functions.\end{proof}
\begin{cor}
Let $\Omega=I_{-}\cup I_{0}\cup I_{+}$ be as above, and let $P_{\pm}$
and $P_{0}$ be the respective projections in $L^{2}(\Omega)$ onto
the subspaces $L^{2}(I_{\pm})$ and $L^{2}(I_{0})$. Let $B=\left(\begin{array}{cc}
\overline{a} & -b\\
\overline{b} & a
\end{array}\right)\in SU(2)$ satisfy $a\neq0$, and let $P_{b}(\lambda)$ be the Poisson-kernel.
Then the unitary one-parameter group $U_{B}(t)$ in $L^{2}(\Omega)$
has the following block-operator matrix-representation: 

\begin{center}
\begin{tabular}{|c|c|c|c|}
\hline 
$U_{B}(t)$ in $L^{2}(\Omega)$ & $L^{2}(I_{-})$ & $L^{2}(I_{0})$ & $L^{2}(I_{+})$\tabularnewline
\hline 
$L^{2}(I_{-})$ & $P_{-}U_{B}(t)P_{-}$ & $P_{-}U_{B}(t)P_{0}$ & $P_{-}U_{B}(t)P_{+}$\tabularnewline
\hline 
$L^{2}(I_{0})$ & $P_{0}U_{B}(t)P_{-}$ & $P_{0}U_{B}(t)P_{0}$ & $P_{0}U_{B}(t)P_{+}$\tabularnewline
\hline 
$L^{2}(I_{+})$ & $P_{+}U_{B}(t)P_{-}$ & $P_{+}U_{B}(t)P_{0}$ & $P_{+}U_{B}(t)P_{+}$\tabularnewline
\hline 
\end{tabular}
\par\end{center}

\medskip{}

The inside of the block-operator matrix may be indexed as follow:
$i,j\in\{\pm,0\}$, $a_{0}^{(B)}(\lambda)\equiv1$. Then, for all
$f\in L^{2}(\Omega)$, 
\begin{equation}
\left(P_{i}U_{B}(t)P_{j}f\right)(x)=\chi_{I_{i}}(x)\left(P_{b}(\lambda)a_{i}(\lambda)\overline{a_{j}(\lambda)}\widehat{\left(P_{j}f\right)}(\lambda)\right)^{\vee}(x-t),\label{eq:tmp-40}
\end{equation}
for all $x\in\Omega$, and $t\in\mathbb{R}$.
\end{cor}

\begin{cor}
Let $f,g\in L^{2}(\Omega)$, and let $B\in U(2)$ as in (\ref{eq:2-by-2 Unitary}).
Then
\begin{enumerate}
\item ~
\begin{equation}
\lim_{t\rightarrow\infty}\left\langle f,U_{B}(t)g\right\rangle _{L^{2}(\Omega)}=0,\mbox{ and}\label{eq:UB-1}
\end{equation}

\item ~
\begin{equation}
\lim_{T\rightarrow\infty}\frac{1}{2T}\int_{-T}^{T}\left|\left\langle f,U_{B}(t)g\right\rangle _{L^{2}(\Omega)}\right|^{2}dt=0.\label{eq:UB-2}
\end{equation}

\end{enumerate}
\end{cor}
\begin{proof}
By Theorem \ref{thm:sp-w-1}, we get with the use of the transform
$V_{B}$ (\ref{eq:V}) and the direct integral decomposition (\ref{eq:decomp}):
\begin{equation}
\left\langle f,U_{B}(t)g\right\rangle _{L^{2}(\Omega)}=\int_{\mathbb{R}}\overline{\left(V_{B}f\right)(\lambda)}\: e_{\lambda}(t)\left(V_{B}g\right)(\lambda)\: d\sigma_{B}(\lambda)\label{eq:tmp-17}
\end{equation}
where $e_{\lambda}(t)=e^{i2\pi\lambda t}$, and $d\sigma_{B}(\lambda)=m^{-2}(\lambda)d\lambda$,
see (\ref{eq:m-bound}) and (\ref{eq:sigma(B)}). But 
\[
\lambda\mapsto\overline{\left(V_{B}f\right)(\lambda)}\:\left(V_{B}g\right)(\lambda)\: m^{-2}(\lambda)\in L^{1}(\mathbb{R},d\lambda)
\]
and so (\ref{eq:UB-1}) follows from the Riemann-Lebesgue theorem.

Part (2) of the Corollary follows from the absence of bounded-states,
and Wiener's lemma. Indeed, $t\mapsto\left\langle f,U_{B}(t)f\right\rangle _{L^{2}(\Omega)}$
is the Fourier transform of the spectral measure
\[
\left|\left\langle \psi_{\lambda}^{(B)},f\right\rangle _{\Omega}\right|^{2}\sigma_{B}(d\lambda),
\]
and the assertion in Theorem \ref{thm:sp-w-1} is that this measure
is non-atomic.
\end{proof}

\section{Unitary One-Parameter Groups: Time Delay Operators\label{sec:UB}}

Consider the Hilbert space $L^{2}(\Omega)$ with $\Omega=I_{-}\cup I_{0}\cup I_{+}$
as before. Choose a boundary matrix $B$ with parameters $w,\theta,\phi,\psi$.
Let $P_{B}$ be the corresponding selfadjoint extension, and form
the one-parameter unitary group
\begin{equation}
U_{B}(t):=e^{-itP_{B}},t\in\mathbb{R}.\label{eq:U_B}
\end{equation}

\textbf{Barriers and bound states.} The reference here is to quantum
states. Since $\Omega$ here is the complement of two finite and disjoint
intervals, we think of these two intervals as barriers. The height
of the barriers is a function of the parameter $w$ from $B$ in (\ref{eq:2-by-2 Unitary}),
see Fig \ref{fig:forward}. The extreme cases are $w=0$, infinite
height, and $w=1$, zero height. Our unitary one-parameter group $U_{B}(t)$
is acting in $L^{2}(\Omega)$, so in the exterior of the two barriers.
For the parameters of $B$ in $U(2)$, see (\ref{eq:2-by-2 Unitary}):
The case $w=0$, is two infinite barriers, and this produces bound
states (Figure \ref{fig:w0-1}), i.e., states trapped between the
two barriers. The other extreme $w=1$ means no barrier. The conclusion
in sect \ref{sec:sp} is that there are bound states only in the case
of infinite barriers ($w=0$). If the barriers have finite height
($w>0$), we prove that there are no bound states; in other words,
the translation representations for the unitary one-parameter group
$U_{B}(t)$ are isometries on all of $L^{2}(\Omega)$; and $U_{B}(t)$
has pure Lebesgue spectrum, i.e., only generalized eigenfunctions
$\psi_{\lambda}$ indexed by $\lambda$ in $\mathbb{R}$. For fixed
$\lambda$, the function $\psi_{\lambda}$ is not in $L^{2}(\Omega)$.

We are using the term \textbf{bound state} as follows. We use $L^{2}(\Omega)$
for modeling quantum mechanical particles (wave functions), not potential
scattering, rather barriers. We identify when an idealized particle
has a tendency to remain localized in the region between the two barriers.
Referring to a Hilbert space of states, this corresponds to interaction
of states where the localized energy is smaller than the total energy.
Therefore these particles cannot be separated unless energy is spent.
The energy spectrum of a bound state (eigenstate) is discrete, unlike
the continuous spectrum of free particles. In the present model, a
finite ``energy barrier'' will be tunneled through. 
\begin{cor}
\label{cor:UB}Let $f=f_{-}+f_{0}+f_{+}$ in $L^{2}(\Omega)$, then
\begin{alignat}{1}
U_{B}(t)f_{-} & =\chi_{-}f_{-}(\cdot-t)+\chi_{0}(a^{-1}\hat{f}_{-})^{\vee}(\cdot-t)+\chi_{+}(a^{-1}c\hat{f}_{-})^{\vee}(\cdot-t)\label{eq:UB(-)}\\
U_{B}(t)f_{0} & =\chi_{-}(\overline{a}^{-1}\hat{f}_{0})^{\vee}(\cdot-t)+\chi_{0}(m^{-2}\hat{f}_{0})^{\vee}(\cdot-t)+\chi_{+}(\overline{c}^{-1}\hat{f}_{0})^{\vee}(\cdot-t)\label{eq:UB(0)}\\
U_{B}(t)f_{+} & =\chi_{-}(c^{-1}a\hat{f}_{+})^{\vee}(\cdot-t)+\chi_{0}(c^{-1}\hat{f}_{0})^{\vee}(\cdot-t)+\chi_{+}f_{+}(\cdot-t).\label{eq:UB(+)}
\end{alignat}
\end{cor}
\begin{proof}
By (\ref{eq:decomp}) and (\ref{eq:tmp-9}), $f_{-}=\int(a^{-1}\hat{f}_{-})\psi_{\lambda}d\lambda$.
Hence, 
\begin{alignat*}{1}
U_{B}(t)f_{-} & =\int(a^{-1}\hat{f}_{-})(a\chi_{-}+\chi_{0}+c\chi_{+})e_{\lambda}(\cdot-t)d\lambda\\
 & =\chi_{-}f_{-}(\cdot-t)+\chi_{0}(a^{-1}\hat{f}_{-})^{\vee}(\cdot-t)+\chi_{+}(a^{-1}c\hat{f}_{-})^{\vee}(\cdot-t).
\end{alignat*}
This is (\ref{eq:UB(-)}). Similarly, we get the other two equations.

\end{proof}
\begin{cor}
\label{cor:UB(t)}~
\begin{enumerate}
\item Let $I$ be any of the three components $I_{-},I_{0},I_{+}$. Let
$f$ be some wave-function localized in $I$. If both $x$ and $x-t$
are in $I$, then 
\[
\left(U_{B}(t)f\right)(x)=f(x-t).
\]

\item Suppose $f$ is supported in $I_{-}$. As the support of $U_{B}(t)f$
hits $x=0$, then it transfers to $1$ with probability $w^{2}$ and
a phase-shift $e(-\phi)$; and to $\beta$ with probability $1-w^{2}$
and a phase-shift $-e(\psi-\theta)$.
\item Suppose $f$ is supported in $I_{0}$. As the support of $U_{B}(t)f$
hits $x=\alpha$, then it transfers to $\beta$ with probability $w^{2}$
and a phase-shift $e(\phi-\theta)$; and to $1$ with probability
$1-w^{2}$ and a phase-shift $e(-\psi)$. 
\item The boundary conditions are preserved by $U_{B}(t)$ for all $t\in\mathbb{R}$;
i.e., we have
\[
\left(\begin{array}{c}
\left(U_{B}(t)f\right)(1)\\
\left(U_{B}(t)f\right)(\beta)
\end{array}\right)=B\left(\begin{array}{c}
\left(U_{B}(t)f\right)(0)\\
\left(U_{B}(t)f\right)(\alpha)
\end{array}\right)
\]
for all $f\in\mathrm{dom}(P_{B})$.
\end{enumerate}
\end{cor}
The dynamics generated by $P_{B}$ corresponds to the following diagrams:

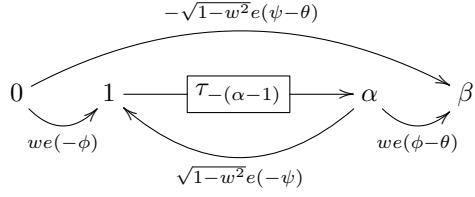
\begin{figure}[H]
\[
\xymatrix{0\ar@/_{1pc}/[r]_{we(-\phi)}\ar@/^{2pc}/[0,4]^{-\sqrt{1-w^{2}}e(\psi-\theta)} & 1\ar@{-}[r] & *+[F]{\tau_{-(\alpha-1)}}\ar[r] & \alpha\ar@/^{2pc}/[0,-2]^{\sqrt{1-w^{2}}e(-\psi)}\ar@/_{1pc}/[r]_{we(\phi-\theta)} & \beta}
\]

\caption{\label{fig:forward}Forward system diagram}

\end{figure}

\begin{rem}
Here $\tau_{-(\alpha-1)}$ denotes the \emph{time delay operator}.
For $w=0$, the transitions from $0$ to $1$ and $\alpha$ to $\beta$
are disconnected, and the the diagram reduces to the union of a compact
(discrete spectrum) and a non-compact (continuous spectrum) component,
see Theorem \ref{thm:sp-w0}. For $w=1$, the transition from $0$
to $\beta$, and the feedback from $\alpha$ to $1$ are reduced,
see Theorem \ref{thm:sp-w1}. For an application, see also section
\ref{sec:scatter}, especially Figure \ref{fig:w}.
\end{rem}

\begin{figure}[H]
\[
\xymatrix{0\ar@/^{2pc}/[0,4]^{-e(\psi-\theta)} & 1\ar@{-}[r] & *+[F]{\tau_{-(\alpha-1)}}\ar[r] & \alpha\ar@/^{2pc}/[0,-2]^{e(-\psi)} & \beta}
\]

\caption{\label{fig:w0}$w=0$}

\end{figure}

\begin{figure}[H]
\[
\xymatrix{0\ar@/_{1pc}/[r]_{e(-\phi)} & 1\ar@{-}[r] & *+[F]{\tau_{-(\alpha-1)}}\ar[r] & \alpha\ar@/_{1pc}/[r]_{e(\phi-\theta)} & \beta}
\]

\caption{\label{fig:w1}$w=1$}
\end{figure}

\begin{rem}
The results above record the cross-overs, and mixing, for the three
components $I_{\pm}$ and $I_{0}$ in $\Omega=I_{-}\cup I_{0}\cup I_{+}$,
$I_{-}=(-\infty,0)$, $I_{0}=(1,\alpha)$, and $I_{+}=(\beta,\infty)$.
Hence the selfadjoint extensions $P_{B}$ of $P_{min}$, with $\mathscr{D}(P_{min})=\{f\in\mathscr{H}_{1}(\Omega)\left|\right.\tilde{f}=0\mbox{ on }\partial\Omega\}$
yield scattering as $U_{B}(t)=e^{itP_{B}}$ is acting on $L^{2}(\Omega)$. 

The individual boundary value problem for the three separate intervals
$I_{-}$, $I_{0}$, and $I_{+}$ do not compare with that for the
union $\Omega$ of the intervals: For example, the operator $P_{min}^{(-)}$
in $L^{2}(I_{-})$ with boundary condition $\tilde{f}(0_{-})=0$ has
deficiency indices $(1,0)$; and so it has no selfadjoint extensions.
Similarly, $P_{min}^{(+)}$ in $L^{2}(I_{+})$ with boundary condition
$\tilde{f}(\beta_{-})=0$ has deficiency indices $(0,1)$, and so
it too does not have any selfadjoint extension. The operator $P_{min}^{(0)}$
with boundary conditions $\tilde{f}(1_{-})=\tilde{f}(\alpha_{+})=0$
has deficiency indices $(1,1)$ and selfadjoint extensions $P_{z}$
corresponds to $\tilde{f}(1_{+})=z\tilde{f}(\alpha_{-})$ as $z$
varies in $\{z\in\mathbb{C},\left|z\right|=1\}$. 

The individual boundary value problems for the three intervals are
not subproblems for the one studied here for $P=\frac{1}{i2\pi}\frac{d}{dx}$
in $L^{2}(\Omega)$.
\end{rem}

\section{Scattering Theory\label{sec:scatter}}

In this section we find the Lax-Phillips scattering operators, one
for each of the selfadjoint operators $P_{B}$ (see Theorems \ref{thm:sp-w1}
and \ref{thm:sp-w}). Recall, from $P_{B}$, we get the corresponding
unitary one parameter groups $U_{B}(t)$; it is computed in Corollary
\ref{cor:UB}. The one-parameter group is needed as Lax-Phillips data
always refer to $U_{B}(t)$. 

When $B$ in $U(2)$ is fixed, we are able in section \ref{sub:scatter}
to explicitly compute both the incoming and outgoing translation representations
for the unitary one parameter group $U_{B}(t)$. From this, in Theorem
\ref{thm:scatter} below, we then compute the Lax-Phillips scattering
operator $S_{B}$, and scattering matrix. Recall the scattering operator
$S_{B}$ commutes with the translation group, and the scattering matrix
with multiplication operators. As a result, $S_{B}$ is a (unitary)
convolution operator, and its transform (the scattering matrix) is
a multiplication operators in the Fourier dual variable $\lambda$;
i.e., the scattering matrix is a unitary valued function of $\lambda$.
It is presented in Theorem \ref{thm:scatter}: Eq (\ref{eq:scatter-2})
gives an expression for this function, with an explicit dependence
on $B$.

\subsection{\label{sub:scatter}Translation Representations and Scattering Operators}

Fix $I_{1}$, $I_{2}$ as before, let $\Omega=I_{-}\cup I_{0}\cup I_{+}$
be the exterior domain. Choose a boundary matrix $B(w,\theta,\phi,\psi)\in U(2)$,
let $P_{B}$ be the selfadjoint extension, and $U_{B}(t)$ the corresponding
unitary one-parameter group.

For $0<w<1$, there is mixing/interaction between the bounded and
unbounded components of $\Omega$, as shown in Corollary \ref{cor:UB(t)},
and Figure \ref{fig:forward}. This fits nicely into the Lax-Phillips
scattering theory \cite{LP68}. 

To begin with, the interacting group $U_{B}(t)$ acts in the perturbed
space $L^{2}(\Omega)$, with $I_{1}$, $I_{2}$ being the obstacles;
meanwhile, there is a free group $U_{0}(t)$ acting in the unperturbed
space $L^{2}(\mathbb{R})$, containing $L^{2}(\Omega)$ as a closed
subspace. Here, $U_{0}(t)$ is the right-translation by $t$ in $L^{2}(\mathbb{R})$.
That is, 
\[
U_{0}(t)f:=f(\cdot-t)
\]
for all $f\in L^{2}(\mathbb{R})$. 

Let $D_{\pm}:=L^{2}(I_{\pm})$ be the outgoing/incoming subspace.
By Corollary \ref{cor:UB(t)}, we have
\begin{enumerate}
\item $U_{B}(t)D_{+}\subset D_{+}$, for all $t>0$; $U_{B}(t)D_{-}\subset D_{-}$,
for all $t<0$.
\item $\bigcap_{t}\left(U_{B}(t)D_{\pm}\right)=\{0\}$.
\item For all $t>0$, $U_{B}(t)=U_{0}(t)$ on $D_{+}$.
\item For all $t<0$, $U_{B}(t)=U_{0}(t)$ on $D_{-}$.
\item Suppose $supp(\varphi)\subset I_{0}$. If $x$, $x-t$ in $I_{0}$,
then $U_{B}(t)\varphi=U_{0}(t)\varphi$.\\

\end{enumerate}

Recall that $P_{B}$ has generalized eigenfunction 
\begin{equation}
\psi_{\lambda}=(a_{\lambda}\chi_{-}+\chi_{0}+c_{\lambda}\chi_{+})e_{\lambda}\label{eq:eigen}
\end{equation}
for all $\lambda\in\mathbb{R}$. See Theorem \ref{thm:sp-w}, and
eq. (\ref{eq:a(lambda)-1}), (\ref{eq:c(lambda)-1}). 

Setting $\psi_{\lambda,+}:=c_{\lambda}^{-1}\psi_{\lambda}$, and $\psi_{\lambda,-}:=a_{\lambda}^{-1}\psi_{\lambda}$,
and define $V_{\pm}:L^{2}(\Omega)\rightarrow L^{2}(\hat{\mathbb{R}})$
by 
\begin{equation}
\left(V_{\pm}f\right)(\lambda):=\left\langle \psi_{\lambda,\pm},f\right\rangle =\int_{\Omega}\overline{\psi_{\lambda,\pm}(x)}f(x)dx\label{eq:S1}
\end{equation}
for all $f\in L^{2}(\Omega)$. 

The adjoint operator $V_{\pm}^{*}:L^{2}(\hat{\mathbb{R}})\rightarrow L^{2}(\Omega)$
is given by
\begin{equation}
(V_{\pm}^{*}\hat{f})(x)=\int_{\mathbb{R}}\hat{f}(\lambda)\psi_{\lambda,\pm}(x)d\lambda\label{eq:S2}
\end{equation}
for all $\hat{f}\in L^{2}(\hat{\mathbb{R}})$.
\begin{rem}
\label{rem:V}In fact, $V_{+}=\overline{c}^{-1}V$ and $V_{-}=\overline{a}^{-1}V$,
where $V$ is given in (\ref{eq:V}). \end{rem}
\begin{thm}
$V_{\pm}$ are unitary operators from $L^{2}(\Omega)$ onto $L^{2}(\hat{\mathbb{R}})$.
In particular, 
\begin{equation}
f(x)=\int_{\mathbb{R}}\left\langle \psi_{\lambda,\pm},f\right\rangle \psi_{\lambda,\pm}(x)\: d\lambda\label{eq:decomp-1}
\end{equation}
 for all $f$ in $L^{2}(\Omega)$. Convergence is in the $L^{2}$-norm
w.r.t. $\sigma_{B}(d\lambda)$.\end{thm}
\begin{proof}
It follows from Remark \ref{rem:V} that 
\begin{alignat*}{1}
V_{+}^{*}V_{+} & =\left(V^{*}\overline{c}\right)\left(\overline{c}^{-1}V\right)=V^{*}V=I.\\
V_{+}V_{+}^{*} & =\left(\overline{c}^{-1}V\right)\left(V^{*}\overline{c}\right)=\overline{c}^{-1}\overline{c}=I.
\end{alignat*}
Hence $V_{+}$ is unitary. Similarly, $V_{-}$ is unitary. Eq. (\ref{eq:decomp-1})
follows from this.
\end{proof}
Pulling the operators $V_{\pm}$ back to $L^{2}(\mathbb{R})$ via
the Fourier transform, we get the outgoing/incoming translation representations
\begin{equation}
R_{\pm}:=\mathscr{F}^{*}V_{\pm}.\label{eq:scatter-1}
\end{equation}

\begin{thm}
$R_{\pm}$ are unitary operators from $L^{2}(\Omega)$ onto $L^{2}(\mathbb{R})$.
Moreover,
\begin{enumerate}
\item $R_{\pm}\big|_{D_{\pm}}=identity$;
\item For all $t\in\mathbb{R}$, we have the following two representations:
\begin{equation}
U_{B}(t)=R_{\pm}^{*}U_{0}(t)R_{\pm}\label{eq:UB}
\end{equation}
i.e., the following diagram commute: 
\[
\xymatrix{L^{2}(\Omega)\ar[r]^{U_{B}(t)}\ar[d]_{V_{\pm}}\ar@/_{3pc}/[dd]_{R_{\pm}} & L^{2}(\Omega)\ar[d]^{V_{\pm}}\ar@/^{3pc}/[dd]^{R_{\pm}}\\
L^{2}(\hat{\mathbb{R}})\ar[r]^{e(-\lambda t)}\ar[d]_{\mathscr{F}^{*}} & L^{2}(\hat{\mathbb{R}})\ar[d]^{\mathscr{F}^{*}}\\
L^{2}(\mathbb{R})\ar[r]^{U_{0}(t)} & L^{2}(\mathbb{R})
}
\]

\end{enumerate}
\end{thm}
\begin{proof}
Clearly, $R_{\pm}$ are unitary. Let $f_{-}\in D_{-}=L^{2}(I_{-})$.
By Remark \ref{rem:V}, 
\[
V_{-}f_{-}=\overline{a}^{-1}Vf_{-}=\overline{a}^{-1}\overline{a}\hat{f}_{-}=\hat{f}_{-};
\]
also see eq. (\ref{eq:tmp-9}). Hence, $R_{-}f_{-}=\mathscr{F}^{*}V_{-}f_{-}=f_{-}$.
Similarly, $R_{+}f_{+}=f_{+}$, for all $f_{+}\in D_{+}=L^{2}(I_{+})$.
Thus, $R_{\pm}$ restricted to $D_{\pm}$ as the identity operator.

From (\ref{eq:decomp-1}), we have 
\begin{alignat*}{1}
\left(U_{B}(t)f\right)(x) & =\int\left\langle \psi_{\lambda,\pm},f\right\rangle U_{B}(t)\psi_{\lambda,\pm}(x)\: d\lambda\\
 & =\int\left\langle \psi_{\lambda,\pm},f\right\rangle e(-\lambda t)\psi_{\lambda,\pm}(x)\: d\lambda.
\end{alignat*}
Hence, $V_{\pm}U_{B}(t)f=e(-\lambda t)V_{\pm}f$, i.e., 
\[
U_{B}(t)=V_{\pm}^{*}e(-\lambda t)V_{\pm}
\]
for all $t\in\mathbb{R}$. Eq (\ref{eq:UB-1}) follows from pulling
the above identity to $L^{2}(\mathbb{R})$ via the Fourier transform. \end{proof}
\begin{rem}
Aside from a possible shift by $\beta$, $R_{\pm}$ are the outgoing/incoming
translation representations in the Lax-Phillips theory.
\end{rem}
Define the scattering operators by
\begin{alignat}{1}
S & :=R_{-}^{*}R_{+}\label{eq:scatter}\\
\tilde{S} & :=R_{+}R_{-}^{*}\label{eq:scatter-3}\\
\hat{S} & :=V_{+}V_{-}^{*}\label{eq:scatter-4}
\end{alignat}
The three operators in (\ref{eq:scatter-8})-(\ref{eq:scatter-4})
are all unitarily equivalent. Specifically, 
\begin{flalign}
S & =R_{-}^{*}\tilde{S}R_{-}\label{eq:scatter-5}\\
\tilde{S} & =\mathscr{F}^{*}\hat{S}\mathscr{F}.\label{eq:scatter-6}
\end{flalign}

In our settings, the usual wave operators $W_{\pm}:L^{2}(\mathbb{R})\rightarrow L^{2}(\Omega)$,
i.e., from the unperturbed space to the perturbed space, are 
\begin{equation}
W_{\pm}:=R_{\pm}^{*};\label{eq:wo}
\end{equation}
and 
\begin{equation}
\tilde{S}=W_{+}^{-1}W_{-}.\label{eq:scatter-7}
\end{equation}
For all $\varphi\in L^{2}(\mathbb{R})$, we have 
\begin{alignat*}{1}
W_{-}\varphi & =s\mbox{ -}\lim_{t\rightarrow-\infty}U_{B}(-t)U_{0}(t)\varphi\\
 & =s\mbox{ -}\lim_{t\rightarrow+\infty}U_{B}(-t)U_{0}(t)\tilde{S}\varphi.
\end{alignat*}
That is, $Range(W_{-})=L^{2}(\Omega)$ consists of scattering states.
Note that $\tilde{S}$ commutes with the free group $\{U_{0}(t)\}$.

The next two results give formulas for the scattering operator and
the scattering matrix.
\begin{thm}
\label{thm:scatter}Let $\hat{S}$ be as in (\ref{eq:scatter-4}),
then $\hat{S}$ is unitary on $L^{2}(\hat{\mathbb{R}})$, and
\begin{equation}
\hat{S}(\lambda)=a(\lambda)^{-1}c(\lambda)\label{eq:scatter-2}
\end{equation}
where $a(\lambda)$, $c(\lambda)$ are the coefficients in the generalized
eigenfunction (\ref{eq:eigen}). More precisely,
\begin{equation}
\hat{S}(\lambda)=e(-\theta-(\beta-\alpha+1)\lambda)\frac{1-\sqrt{1-w^{2}}e(\psi-(\alpha-1)\lambda)}{1-\sqrt{1-w^{2}}e(-\psi+(\alpha-1)\lambda)}.\label{eq:scatter-9}
\end{equation}
\end{thm}
\begin{proof}
By Remark \ref{rem:V}, 
\[
\hat{S}=\left(\overline{c}^{-1}V\right)\left(\overline{a}^{-1}V\right)^{*}=\overline{c}^{-1}VV^{*}\overline{a}=\overline{c}^{-1}\overline{a}=a^{-1}c.
\]
Note the last step follows from $\left|a\right|^{2}=\left|c\right|^{2}$.
By (\ref{eq:a(lambda)-1}) - (\ref{eq:H(lambda)}), we have 
\[
a(\lambda)^{-1}c(\lambda)=e(-\theta-(\beta-\alpha+1)\lambda)H(\lambda)\overline{H(\lambda)}^{-1}
\]
where
\begin{equation}
H(\lambda)=\frac{1}{1-\sqrt{1-w^{2}}e(-\psi+(\alpha-1)\lambda)}=\frac{1}{1-b\: e_{\alpha-1}(\lambda)}\label{eq:H}
\end{equation}
in the $B=\left(\begin{array}{cc}
\overline{a} & -b\\
\overline{b} & a
\end{array}\right)$ presentation (\ref{eq:B}). This yields (\ref{eq:scatter-9}).
\end{proof}
The following alternative characterization of the scattering operator
$\hat{S}(\lambda)$ reveals its effect on incoming wave-packets.
\begin{cor}
\label{cor:scatter}Given $B\in U(2)$ with parameters as in (\ref{eq:2-by-2 Unitary}),
let $\hat{S}(\lambda)$ be as in (\ref{eq:scatter-4}). Then 
\begin{equation}
\hat{S}(\lambda)=e(-\theta)e(-(\beta-\alpha+1)\lambda)w^{2}H(\lambda)-e(\psi-\theta)\sqrt{1-w^{2}}e(-\beta\lambda)\label{eq:scatter-8}
\end{equation}
\end{cor}
\begin{proof}
Set $z:=\sqrt{1-w^{2}}e(-\psi+(\alpha-1)\lambda)$. Then (\ref{eq:scatter-9})
reads 
\begin{alignat*}{1}
\hat{S}(\lambda) & =e(-\theta-(\beta-\alpha+1)\lambda)\frac{1-\overline{z}}{1-z}\\
 & =e(-\theta-(\beta-\alpha+1)\lambda)\left(\frac{1-\left|z\right|^{2}}{1-z}-\overline{z}\right)\\
 & =e(-\theta-(\beta-\alpha+1)\lambda)\left(\frac{w^{2}}{1-z}-\overline{z}\right)\\
 & =e(-\theta-(\beta-\alpha+1)\lambda)\left(w^{2}H(\lambda)-\sqrt{1-w^{2}}e(\psi-(\alpha-1)\lambda)\right)
\end{alignat*}
and (\ref{eq:scatter-8}) follows.\end{proof}
\begin{rem}
The pole of $H(z)$ on the right-side of (\ref{eq:scatter-8}) accounts
for the resonance caused by the two obstacles $I_{1},I_{2}$; the
second term on the right-side corresponds to a direct propagation
from $D_{-}$ into $D_{+}$. See the examples below. \end{rem}
\begin{example}
Consider $I_{1}=[0,1]$, $I_{2}=[2,3]$, and the exterior domain $\Omega$
is the union of three components 
\[
I_{-}=(-\infty,0),\: I_{0}=(1,2),\: I_{+}=(3,\infty).
\]
See Figure \ref{fig:w} below. 

Let $f$ be a unit-step function supported on $[-\frac{1}{2},0]$,
i.e., $f(x)=1$, for all $x\in[-\frac{1}{2},0]$, and vanishes elsewhere;
then $f\in D_{-}$. 

The action of $U_{B}(t)$ is given in section \ref{sec:UB}. For details,
see Corollary \ref{cor:UB(t)} and Figure \ref{fig:forward}.\end{example}
\begin{enumerate}
\item On $D_{-}=L^{2}(I_{-})$, the interacting group acts the same as the
free group, i.e., right-translation by $t$. Hence the wave-packet
vanishes at $t=\frac{1}{2}$. 
\item On $D_{0}:=L^{2}(I_{0})$,
\[
\left(U_{B}(t)f\right)(x)\big|_{D_{0}}=(a^{-1}\hat{f})^{\vee}(x-t)\big|_{D_{0}}.
\]
Recall that 
\begin{equation}
a(\lambda)^{-1}=w\: e(-\phi)e(-\lambda)H(\lambda)\label{eq:tmp-15}
\end{equation}
see eq. (\ref{eq:H}), and (\ref{eq:a-inv(lambda)}).\\
\\
At $t=0$, $f$ moves into $D_{0}$ with a magnitude $w\: e(-\phi)$;
and it propagates within $D_{0}$ until hitting the right-end point
of $I_{0}$ ($x=2$) at $t=1$.\\
\\
For $t>1$, $U_{B}(t)$ generates resonance, as seen in the pole of
the transfer function $H(z)$ in (\ref{eq:tmp-15}). Specifically,
$f$ propagates out of $I_{0}$ at the right-end point ($x=2$), and
moves back into $D_{0}$ from the left-end point ($x=1$), modulated
by $\sqrt{1-w^{2}}e(-\psi)$.\\

\item On $D_{+}=L^{2}(I_{+})$, the scattered wave propagates as
\begin{equation}
\left(U_{B}(t)f\right)(x)\big|_{D_{+}}=(a^{-1}c\hat{f})^{\vee}(x-t)\big|_{D_{+}}.\label{eq:tmp-14}
\end{equation}
The right-side of (\ref{eq:tmp-14}) is the restriction of $(\hat{S}\hat{f})^{\vee}$,
i.e., $\tilde{S}f$, to $D_{+}$. See (\ref{eq:scatter-2}) and (\ref{eq:scatter-6}).
From (\ref{eq:scatter-8}), we see that $\hat{S}\hat{f}$ consists
of two parts:

\begin{itemize}
\item direct propagation from $D_{-}$ into $D_{+}$
\[
-e(\psi-\theta)\sqrt{1-w^{2}}e(-\beta\lambda)\hat{f}(\lambda)
\]
where $f$ is modulated by $-e(\psi-\theta)\sqrt{1-w^{2}}$;
\item resonance caused by the obstacles
\[
e(-\theta-(\beta-\alpha+1)\lambda)w^{2}H(\lambda)\hat{f}(\lambda)
\]
This differs from (\ref{eq:tmp-15}) by $w\: e(\phi-\theta)$. That
is, the scattered wave is transmitted out of the interacting region
$D_{0}$, into $D_{+}$, and is modulated by $w\: e(\phi-\theta)$.
\end{itemize}
\end{enumerate}

\begin{example}
\label{ex:wave}Continue with the previous example. Set $\theta=\phi=\psi=0$,
and $w=\frac{\sqrt{3}}{2}$, so $B=\left(\begin{array}{cc}
\frac{\sqrt{3}}{2} & -\frac{1}{2}\\
\frac{1}{2} & \frac{\sqrt{3}}{2}
\end{array}\right)$We construct three functions:
\begin{enumerate}
\item incoming wave
\[
f(x)=\begin{cases}
1 & x\in[-\frac{1}{2},0]\\
0 & \mbox{otherwise}
\end{cases};
\]

\item in the interacting region
\begin{alignat*}{1}
(a^{-1}\hat{f})^{\vee}(x) & =w\sum_{n=0}^{\infty}\left(1-w^{2}\right)^{\frac{n}{2}}f(x-1+n(\alpha-1))\\
 & =\frac{\sqrt{3}}{2}\sum_{n=0}^{\infty}\frac{1}{2^{n}}f(x-1+n(\alpha-1));
\end{alignat*}

\item outgoing wave
\begin{alignat*}{1}
(\tilde{S}f)(x) & =-\sqrt{1-w^{2}}f(x-\beta)\\
 & \quad+w^{2}\sum_{n=0}^{\infty}\left(1-w^{2}\right)^{\frac{n}{2}}f(x-\beta+(n+1)(\alpha-1))\\
 & =-\frac{1}{2}f(x-\beta)+\frac{3}{4}\sum_{n=0}^{\infty}\frac{1}{2^{n}}f(x-\beta+(n+1)(\alpha-1)).
\end{alignat*}

\end{enumerate}

\noindent Moreover,
\[
\lim_{t\rightarrow+\infty}\left\Vert U_{B}(t)f-U_{0}(t)\tilde{S}f\right\Vert =0.
\]
The propagation of $f$ through $I_{1}\cup I_{2}$ is shown in Figure
\ref{fig:w}.

\end{example}
\begin{figure}
\begin{tabular}{c}
\includegraphics{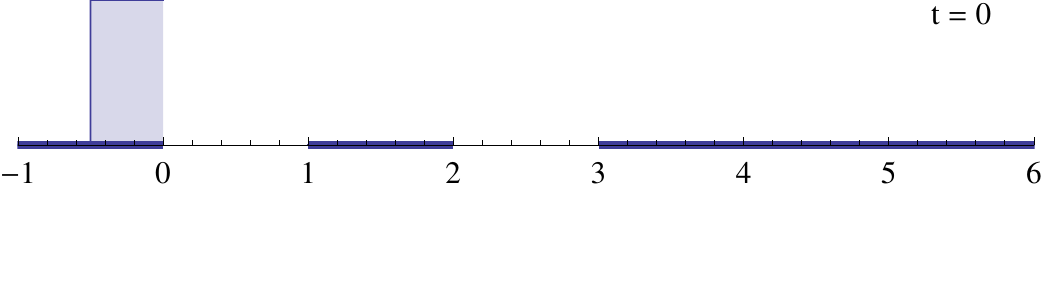}\tabularnewline
\includegraphics{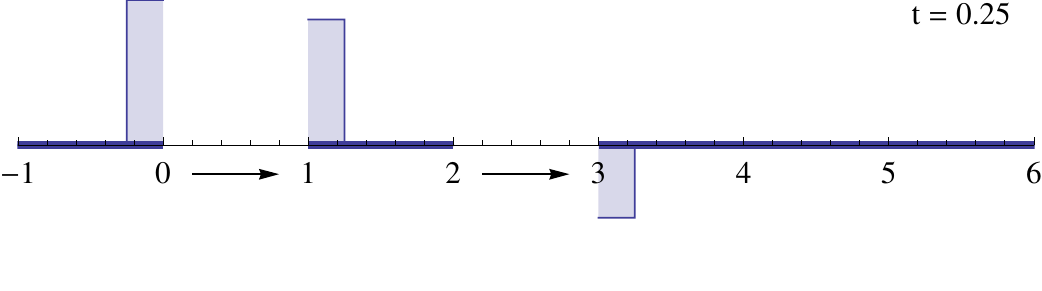}\tabularnewline
\includegraphics{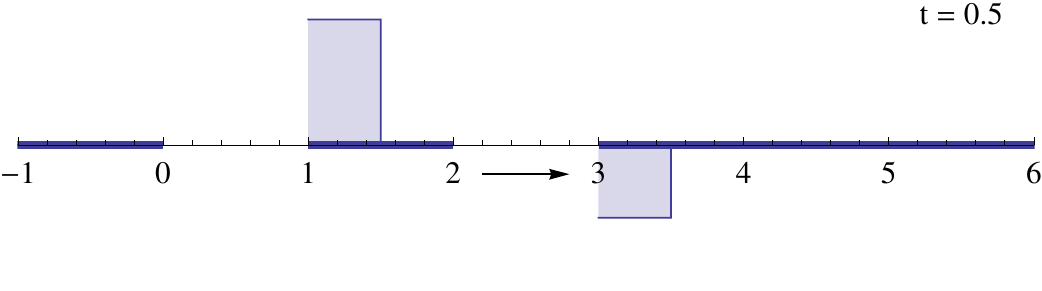}\tabularnewline
\includegraphics{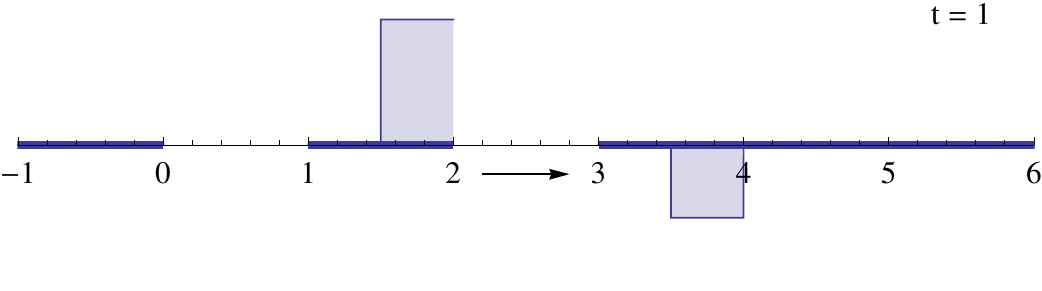}\tabularnewline
\includegraphics{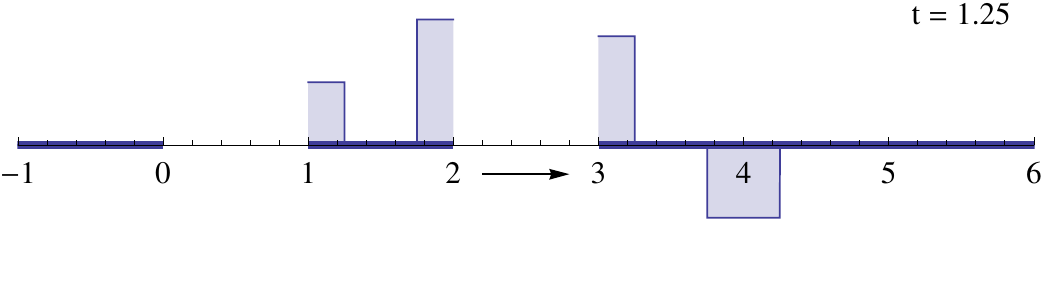}\tabularnewline
\includegraphics{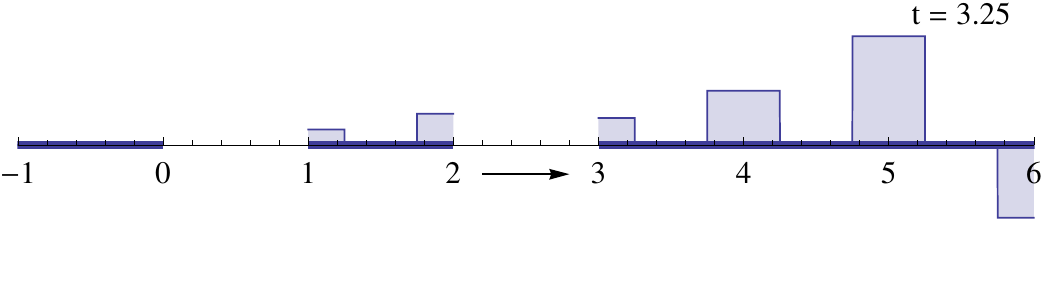}\tabularnewline
\end{tabular}

\caption{\label{fig:w}Wave-packet propagating through a double-barrier $I_{1}\cup I_{2}$
 where $I_{1}=[0,1]$, $I_{2}=[2,3]$, and $\Omega=\mathbb{R}\backslash(I_{1}\cup I_{2})$.
See Example \ref{ex:wave}.}

\end{figure}

\section{Spectral representation and scattering \label{sec:ss}}

In this section we calculate more details regarding spectral and scattering.
Since the scattering information is encoded in $L^{2}(I_{0})$, and
$I_{0}$ is a finite interval, the Fourier transform of functions
in $L^{2}(I_{0})$ are band-limited. As a result, by restricting one
of the variables in the Shannon kernel, we get an orthonormal basis
(ONB). We compute the scattering operator, and the Lax-Phillips semigroup
in this ONB.

\subsection{Obstacle scattering\label{sub:OS}}

\subsubsection{Two normalizations}

We continue our analysis of analysis in $L^{2}(\Omega)$ when $\Omega$
is the union of three disjoint open intervals, two infinite half-lines,
and a bounded interval $I_{0}$ in the middle. As we will be working
with Shannon\textquoteright{}s kernel, it will be convenient in some
computations to choose $I_{0}$ to have unit length.
\begin{enumerate}
\item $I_{-}=(-\infty,0)$, $I_{0}=(1,\alpha)$, and $I_{+}=(\beta,\infty)$;
\item $I_{-}=(-\infty,\tilde{\alpha})$, $I_{0}=(-\frac{1}{2},\frac{1}{2})$,
and $I_{+}=(\tilde{\beta},\infty)$.
\end{enumerate}
In both cases,
\[
\Omega:=I_{-}\cup I_{0}\cup I_{+};
\]
and let $P_{0}$ and $P_{\pm}$ be the projection operators given
by multiplication: 
\[
P_{0}:=multi\chi_{0},\ P_{\pm}:=multi\chi_{\pm}
\]
acting in the Hilbert space $L^{2}(\Omega)$. 

We need the Shannon kernel for both cases. 
\begin{lem}
Let 
\begin{eqnarray*}
\varphi(x) & = & \begin{cases}
1 & x\in[-\frac{T}{2},\frac{T}{2}]\\
0 & \mbox{otherwise}
\end{cases};
\end{eqnarray*}
then $\hat{\varphi}(\lambda)=\frac{\sin(\pi\lambda T)}{\pi\lambda}$. \end{lem}
\begin{proof}
This follows from a direct computation, see also \cite{DyMc72}.\end{proof}
\begin{rem}
For case (1), we choose $T=1$, and the Shannon kernel is
\begin{equation}
Shann(x):=\frac{\sin(\pi\lambda)}{\pi\lambda}=\mbox{Sinc}(\pi\lambda).\label{eq:sha1}
\end{equation}
For case (2), we choose $T=\alpha-1$ (length of the middle interval
$I_{0}$), and translation $\varphi$ to the right by $(\alpha+1)/2$
(i.e., the mid-point of $I_{0}$). 
\begin{alignat}{1}
Shann(x) & :=e^{i2\pi(\frac{\alpha+1}{2})}\frac{\sin(\pi(\alpha-1)\lambda)}{\pi\lambda}\nonumber \\
 & =e^{i\pi(\alpha+1)}\frac{\sin(\pi(\alpha-1)\lambda)}{\pi\lambda}\nonumber \\
 & =e^{i\pi(\alpha+1)}(\alpha-1)\frac{\sin(\pi(\alpha-1)\lambda)}{\pi(\alpha-1)\lambda}\nonumber \\
 & =e^{i\pi(\alpha+1)}(\alpha-1)\mbox{Sinc}(\pi(\alpha-1)\lambda).\label{eq:sha2}
\end{alignat}
Compare with the kernel in (\ref{eq:sha1}). Note the argument used
in the proofs applies to both kernels (\ref{eq:sha1}) and (\ref{eq:sha2}).
\end{rem}

\begin{lem}
The Shannon kernel on $I_{0}=(1,\alpha)$ is
\[
e\left(\frac{\alpha+1}{2}\lambda\right)\frac{\sin(\pi(\alpha-1)\lambda)}{\pi\lambda}.
\]
\end{lem}
\begin{proof}
We check that 
\begin{eqnarray*}
\int_{1}^{\alpha}e(\lambda x)dx & = & \frac{1}{i2\pi\lambda}\left(e(\alpha\lambda)-e(\lambda)\right)\\
 & = & \frac{1}{\pi\lambda}e\left(\frac{\alpha+1}{2}\lambda\right)\sin\left(\pi(\alpha-1)\lambda\right)\\
 & = & e\left(\frac{\alpha+1}{2}\lambda\right)\frac{\sin\left(\pi(\alpha-1)\lambda\right)}{\pi\lambda}.
\end{eqnarray*}

\end{proof}

\subsubsection{Summary }

For convenience, here is a quick summary of the comparison between
the two setups:
\begin{enumerate}
\item If $\Omega=(-\infty,0)\cup(1,\alpha)\cup(\beta,\infty)$; \\
Shannon kernel:
\[
K_{Shann}(x)=\frac{\sin(\pi\lambda T)}{\pi\lambda};
\]

\item Rescaled version - $\Omega=(-\infty,\tilde{\alpha})\cup(-\frac{1}{2},\frac{1}{2})\cup(\tilde{\beta},\infty)$;\\
Shannon kernel:
\[
K_{Shann}(x)=e^{i\pi(\alpha+1)}(\alpha-1)\frac{\sin(\pi(\alpha-1)\lambda)}{\pi(\alpha-1)\lambda}.
\]

\end{enumerate}
In both cases, the unitary group is
\[
U_{B}(t)=e^{-itP_{B}},t\in\mathbb{R}
\]
so that
\[
\left(V_{B}U_{B}(t)f\right)(\lambda)=e(-\lambda t)\left(V_{B}f\right)(\lambda).
\]
This amounts to a right-translation by $t$, i.e., 
\[
f\mapsto f(\cdot-t).
\]

\subsection{Computation of direct integral decomposition}

Fix $B(w,\theta,\phi,\psi)\in U(2)$, $0<w<1$, we have
\begin{itemize}
\item $P_{B}$ selfadjoint operator in $L^{2}(\Omega)$
\item $\{U_{B}(t)\}_{t\in\mathbb{R}}$ acting in $L^{2}(\Omega)$; here
$U_{B}(t):=e^{-itP_{B}}$.
\item A unitary operator $V_{B}:L^{2}(\Omega)\rightarrow L^{2}(\mathbb{R},\sigma_{B})$,
where $\sigma_{B}(d\lambda)=m^{-2}(\lambda)d\lambda$.
\end{itemize}
Let $f\in L^{2}(\Omega)$, $t,\lambda\in\mathbb{R}$, recall that
\[
\left(V_{B}U_{B}(t)f\right)(\lambda)=e_{\lambda}(-t)\left(V_{B}f\right)(\lambda)
\]
and 
\[
L^{2}(\Omega)\ni f=\int_{\mathbb{R}}^{\oplus}\left(V_{B}f\right)(\lambda)\psi_{\lambda}^{(B)}d\sigma_{B}(\lambda).
\]
i.e., a direct integral decomposition. 
\begin{rem}
The transform, generalized eigenfunctions, and the measure all depend
on $B$. This is indicated with the sup/sub-scripts.\end{rem}
\begin{lem}
For $f\in L^{2}(\Omega)$, we have 
\begin{alignat*}{1}
\int_{\Omega}\left|f(x)\right|^{2}dx & =\int_{\mathbb{R}}\left|\left(V_{B}f\right)(\lambda)\right|^{2}d\sigma_{B}(\lambda)\\
\left(V_{B}f\right)(\lambda) & =\int_{\Omega}\overline{\psi_{\lambda}(x)}f(x)dx\\
 & =\overline{a(\lambda)}\left(P_{-}f\right)^{\wedge}+\left(P_{0}f\right)^{\wedge}+\overline{a(\lambda)}\left(P_{+}f\right)^{\wedge}
\end{alignat*}
where
\[
\psi_{\lambda}=\psi_{\lambda}^{(B)}=a(\lambda)\chi_{-}+\chi_{0}+c(\lambda)\chi_{+}.
\]
Moreover, 
\begin{eqnarray}
P_{-}\psi_{\lambda} & = & a(\lambda)e_{\lambda}\mbox{ on }I_{-}\label{eq:4-1}\\
P_{0}\psi_{\lambda} & = & e_{\lambda}\mbox{ on }I_{0}\label{eq:5}\\
P_{+}\psi_{\lambda} & = & c(\lambda)e_{\lambda}\mbox{ on }I_{+}.\label{eq:6}
\end{eqnarray}
\end{lem}
\begin{proof}
A direct calculation.
\end{proof}
The spectral representation is summarized in the following theorem:
\begin{thm}
Let $B(w,\theta,\phi,\psi)\in U(2)$ be the boundary matrix in (\ref{eq:2-by-2 Unitary}),
$0<w<1$; and let $P_{B}$ be the corresponding selfadjoint extension.
The spectral representation theorem (in its fancy version) applied
to $P_{B}$ as a selfadjoint operator in $L^{2}(\Omega)$ has multiplicity-one,
and its direct integral measure is $d\sigma_{B}(\lambda):=m^{-2}(\lambda)d\lambda$
on the whole Hilbert space $L^{2}(\Omega)$.
\end{thm}

\subsection{Shannon kernel and scattering}

Suppose we are in case (1), i.e., the middle interval is $I_{0}=[-\frac{1}{2},\frac{1}{2}]$.
The Shannon kernel is
\begin{equation}
K(\lambda,\xi)=\frac{\sin(\pi(\lambda-\xi))}{\pi(\lambda-\xi)},\;\lambda,\xi\in\mathbb{R};\label{eq:tmp-22}
\end{equation}
See \cite{DyMc72} for its properties. 

Recall that the Shannon is the kernel of the projection operator onto
the the space of band-limited functions:
\[
BL:=\{\hat{f}(\cdot);\chi_{I_{0}}f=f\}\subset L^{2}(\mathbb{R},d\lambda).
\]
Note the identifications: 
\[
f\in L^{2}(I_{0})\Longleftrightarrow\{f\in L^{2}(\Omega);\chi_{I_{0}}f=f\}
\]
and 
\[
L^{2}(I_{0})\simeq P_{0}L^{2}(\Omega).
\]
So, 
\[
f\in L^{2}(I_{0})\Longleftrightarrow\hat{f}\in BL.
\]

\begin{lem}[Shannon Interpolation]
 If $f\in L^{2}(I_{0})=P_{0}L^{2}(\Omega)$, then
\[
\hat{f}(\lambda)=\sum_{n\in\mathbb{Z}}\hat{f}(n)\frac{\sin\pi(\lambda-n)}{\pi(\lambda-n)}
\]
and $\left\{ \frac{\sin\pi(\lambda-n)}{\pi(\lambda-n)}\right\} _{n\in\mathbb{Z}}$
is an ONB in $BL$ (band-limited functions, frequency band = $[-\frac{1}{2},\frac{1}{2}]$). \end{lem}
\begin{proof}
A calculation, see e.g., \cite{DyMc72}.
\end{proof}

\subsection{Computation of the scattering semigroup\label{sub:sg}}

In our model $\Omega$ has two unbounded components, and one bounded
$I_{0}$ in the middle. (By rescaling we may arrange that $I_{0}$
has unit length.) In the language of Lax-Phillips \cite{LP68}, $I_{0}$
then represents \textquotedbl{}obstacle\textquotedbl{} for the unitary
one-parameter group $U_{B}(t)$ transforming the global states. As
predicted by \cite{LP68}, we show below that the cut-down of $U_{B}(t)$
will then be a contraction semigroup (now in $L^{2}(I_{0})$). We
are further able to compute this semigroup and show how it depends
on the unitary matrix $B$ classifying our selfadjoint extension operators.
Moreover we show that the semigroup carries detailed scattering information;
and it is also of relevance to model theory; see \cite{JoMu80}.

\subsubsection{Key lemmas}

Recall some key steps that will be used below.
\begin{lem}
If $f\in L^{2}(I_{0})$, then
\begin{equation}
\left(V_{B}f\right)(\lambda)=\hat{f}(\lambda).\label{eq:tmp-21}
\end{equation}
\end{lem}
\begin{proof}
In fact,
\begin{alignat*}{1}
\left(V_{B}f\right)(\lambda) & =\left\langle \psi_{\lambda}^{(B)},f\right\rangle _{\Omega}=\left\langle \psi_{\lambda}^{(B)},P_{0}f\right\rangle _{\Omega}=\left\langle P_{0}\psi_{\lambda}^{(B)},f\right\rangle _{\Omega}\\
 & =\left\langle e_{\lambda},f\right\rangle _{\Omega}=\int_{I_{0}}\overline{e_{\lambda}(x)}f(x)dx=\hat{f}(\lambda).
\end{alignat*}
See (\ref{eq:5}).\end{proof}
\begin{lem}
\label{lem:middle}If $f\in L^{2}(I_{0})$ then
\[
f=\int_{\mathbb{R}}\hat{f}(\lambda)e_{\lambda}d\sigma_{B}(\lambda).
\]
In particular, for all $x\in I_{0}$, 
\[
f(x)=\left(m^{-2}\hat{f}\right)^{\vee}(x)=\sum_{n\in\mathbb{Z}}a_{n}f(x+n)
\]
where 
\[
m^{-2}(\lambda)=\sum_{n\in\mathbb{Z}}a_{n}e_{n}(\lambda)
\]
is the Fourier series.\end{lem}
\begin{proof}
By (\ref{eq:tmp-21}), we get
\[
f=\int_{\mathbb{R}}\hat{f}(\lambda)\psi_{\lambda}^{(B)}d\sigma_{B}(\lambda).
\]
Apply $P_{0}$ on both sides, we get
\begin{alignat*}{1}
f=P_{0}f & =\int_{\mathbb{R}}\hat{f}(\lambda)P_{0}\psi_{\lambda}^{(B)}d\sigma_{B}(\lambda)\\
 & =\int_{\mathbb{R}}\hat{f}(\lambda)e_{\lambda}d\sigma_{B}(\lambda)
\end{alignat*}
by (\ref{eq:5}).\end{proof}
\begin{cor}
$\hat{f}\mapsto\widehat{Z_{B}(t)f}$ is expressed in terms of $K_{Shann}$
as 
\begin{equation}
\widehat{Z_{B}(t)f}(\lambda)=\int_{\mathbb{R}}\frac{\sin\pi(\lambda-\xi)}{\pi(\lambda-\xi)}e_{\xi}(-t)\hat{f}(\xi)d\sigma_{B}(\lambda).\label{eq:tmp-23}
\end{equation}
Moreover, 
\[
\sum_{n\in\mathbb{Z}}\left|\widehat{Z_{B}(t)f}(n)\right|^{2}\leq\frac{4}{w^{2}}\left\Vert f\right\Vert _{\Omega}^{2}.
\]
\end{cor}
\begin{proof}
Here we use the middle interval $I_{0}=[-\frac{1}{2},\frac{1}{2}]$,
so the kernel is given in (\ref{eq:tmp-22}). By Lemma \ref{lem:middle},
we have 
\begin{eqnarray}
\left(Z_{B}(t)f\right)(x) & = & \chi_{0}\left(e(-\lambda t)m^{-2}(\lambda)\hat{f}(\lambda)\right)^{\vee}\nonumber \\
 & = & \left(\widehat{\chi_{0}}*\left(e_{\lambda}(-t)m^{-2}\hat{f}\right)\right)^{\vee}\nonumber \\
 & = & \int_{\mathbb{R}}\frac{\sin\pi(\lambda-\xi)}{\pi(\lambda-\xi)}e_{\xi}(-t)\hat{f}(\xi)d\sigma_{B}(\xi).\label{eq:tmp-24}
\end{eqnarray}
Using the interpolation formula, the RHS above is
\begin{alignat*}{1}
RHS & =\sum_{n\in\mathbb{Z}}\widehat{Z_{B}(t)f}(n)\frac{\sin\pi(\lambda-n)}{\pi(\lambda-n)}\\
 & =\sum_{n\in\mathbb{Z}}\int_{\mathbb{R}}\frac{\sin\pi(\lambda-\xi)}{\pi(\lambda-\xi)}e_{\xi}(-t)\hat{f}(\xi)d\sigma_{B}(\xi)\times\frac{\sin\pi(\lambda-n)}{\pi(\lambda-n)};
\end{alignat*}
and
\begin{alignat*}{1}
\sum_{n\in\mathbb{Z}}\left|\widehat{Z_{B}(t)f}(n)\right|^{2} & =\int_{\mathbb{R}}\left|\widehat{Z_{B}(t)f}(\lambda)\right|^{2}d\lambda\mbox{ (Parseval + Shannon ONB in BL)}\\
 & =\int_{\mathbb{R}}\left|m(\lambda)\right|^{2}\left|\widehat{Z_{B}(t)f}(\lambda)\right|^{2}d\sigma_{B}(\lambda)\\
 & \leq\frac{4}{w^{2}}\int\left|\widehat{Z_{B}(t)f}(\lambda)\right|^{2}d\sigma_{B}(\lambda)\\
 & =\frac{4}{w^{2}}\left\Vert Z_{B}(t)f\right\Vert _{\Omega}^{2}\leq\frac{4}{w^{2}}\left\Vert f\right\Vert _{\Omega}^{2}.
\end{alignat*}
Note the last two steps follows from Prop \ref{prop:m}.
\end{proof}

\subsubsection{Semigroups}

Below we derive explicit formulas for the Lax-Phillips semigroup $Z_{B}(t)$
making use of Shannon\textquoteright{}s kernel, as well as the Shannon
interpolation formula (see e.g., \cite{DyMc72}.) This material leads
up to Theorem \ref{thm:res}, giving a formula for the analytic resolvent
operator $R_{B}(\cdot)$, analytic in the complex right-half plane,
and computed from the infinitesimal generator of $Z_{B}(t)$.
\begin{thm}
\label{thm:ZB}$Z_{B}(t):=P_{0}U_{B}(t)P_{0}:L^{2}(I_{0})\rightarrow L^{2}(I_{0})$,
$t\geq0$, is a contraction semigroup, i.e.
\begin{enumerate}
\item For all $s,t\geq0$, 
\begin{equation}
Z_{B}(t)Z_{B}(s)=Z_{B}(t+s);\label{eq:ZB-1}
\end{equation}
 
\item $Z_{B}(0)=P_{0}$, acting as the identity operator in $L^{2}(I_{0})$.
\end{enumerate}
\end{thm}
\begin{proof}
For all $f\in L^{2}(I_{0})$, and $t>0$, 
\begin{alignat*}{1}
\left\Vert Z_{B}(t)f\right\Vert _{I_{0}} & =\left\Vert P_{0}U_{B}(t)P_{0}f\right\Vert _{\Omega}\\
 & \leq\left\Vert U_{B}(t)P_{0}f\right\Vert _{\Omega}=\left\Vert P_{0}f\right\Vert _{\Omega}=\left\Vert f\right\Vert _{I_{0}}
\end{alignat*}
since $\left\Vert P_{0}\right\Vert _{L^{2}(\Omega)\rightarrow L^{2}(\Omega)}\leq1$,
i.e., the projection $P_{0}$ is contractive. This proves that $\left\Vert Z_{B}(t)\right\Vert _{I_{0}}\leq1$.

Let $s,t\geq0$, then 
\begin{eqnarray}
Z_{B}(s)Z_{B}(t) & = & P_{0}U_{B}(s)P_{0}U_{B}(t)P_{0}\nonumber \\
 & = & P_{0}U_{B}(s)\left(P_{(\beta,\infty)}^{\perp}P_{(-\infty,0)}^{\perp}\right)U_{B}(t)P_{0}\nonumber \\
 & = & \left(P_{0}U_{B}(s)P_{(\beta,\infty)}^{\perp}\right)\left(P_{(-\infty,0)}^{\perp}U_{B}(t)P_{0}\right)\nonumber \\
 & = & \left(P_{(\beta,\infty)}^{\perp}U_{B}(-s)P_{0}\right)^{*}\left(P_{(-\infty,0)}^{\perp}U_{B}(t)P_{0}\right).\label{eq:1}
\end{eqnarray}
Since for $s\geq0$, we have $U_{B}(s)L^{2}(I_{+})\subset L^{2}(I_{+})$,
and it follows that 
\[
P_{(\beta,\infty)}^{\perp}U_{B}(-s)P_{0}=U_{B}(-s)P_{0}.
\]
Similarly, $t\geq0$ implies that 
\[
P_{(-\infty,0)}^{\perp}U_{B}(t)P_{0}=U_{B}(t)P_{0}.
\]
Therefore, (\ref{eq:1}) reads
\begin{eqnarray*}
Z_{B}(s)Z_{B}(t) & = & \left(U_{B}(-s)P_{0}\right)^{*}\left(U_{B}(t)P_{0}\right)\\
 & = & \left(P_{0}U_{B}(s)\right)\left(U_{B}(t)P_{0}\right)\\
 & = & P_{0}U_{B}(s+t)P_{0}.
\end{eqnarray*}
This shows that $Z_{B}(t)$ satisfies the semigroup law in (\ref{eq:ZB-1}).

Clearly, $Z_{B}(0)=P_{0}$; and this completes the proof of the theorem.
\end{proof}
The semigroup law (\ref{eq:ZB-1}) can be checked directly using the
Shannon kernels. 

Here we are still in case (2), where the middle interval is $I_{0}=[-\frac{1}{2},\frac{1}{2}]$.
But the same argument applies to $I_{0}=[1,\alpha]$ as well. Recall
the infinitesimal generator $G_{B}$ of $Z_{B}(t)$ is 
\[
G_{B}f:=\frac{1}{2\pi i}\lim_{t\rightarrow0_{+}}\frac{1}{t}\left(Z_{B}(t)f-f\right)
\]
with $dom(G_{B})=\{f\in L^{2}(I_{0}):\mbox{ the above limit exits}\}$.
That is,
\[
\mathscr{D}(G_{B})=\{f\in L^{2}(I_{0});\int_{\mathbb{R}}\left|\hat{f}(\lambda)\right|^{2}\lambda^{2}d\lambda<\infty\}.
\]
For motivations, see \cite{LP68}. 
\begin{proof}[Direct proof of Theorem \ref{thm:ZB}]
 Let $s,t\geq0$, $f\in L^{2}(I_{0})$, recall that
\[
\widehat{Z_{B}(t)f}(\lambda)=\int_{\mathbb{R}}\frac{\sin\pi(\xi-\lambda)}{\pi(\xi-\lambda)}e_{\xi}(-t)\hat{f}(\xi)d\sigma_{B}(\xi);
\]
so that 
\begin{alignat*}{1}
 & \left(Z_{B}(s)Z_{B}(t)f\right)^{\wedge}(\lambda)\\
= & \int_{\mathbb{R}}\frac{\sin\pi(\xi-\lambda)}{\pi(\xi-\lambda)}e_{\xi}(-s)\widehat{Z_{B}(t)f}(\xi)d\sigma_{B}(\xi)\\
= & \int_{\mathbb{R}}\frac{\sin\pi(\xi-\lambda)}{\pi(\xi-\lambda)}e_{\xi}(-s)\left(\int_{\mathbb{R}}\frac{\sin\pi(\eta-\xi)}{\pi(\eta-\xi)}e_{\xi}(-t)\hat{f}(\eta)d\sigma_{B}(\eta)\right)d\sigma_{B}(\xi)\\
= & \iint_{\mathbb{R}}\frac{\sin\pi(\xi-\lambda)}{\pi(\xi-\lambda)}\frac{\sin\pi(\eta-\xi)}{\pi(\eta-\xi)}e_{\xi}(-s)e_{\xi}(-t)\hat{f}(\eta)d\sigma_{B}(\eta)d\sigma_{B}(\xi)\\
= & \int_{\mathbb{R}}\left(\int_{\mathbb{R}}\frac{\sin\pi(\xi-\lambda)}{\pi(\xi-\lambda)}\frac{\sin\pi(\eta-\xi)}{\pi(\eta-\xi)}e_{\xi}(-t)d\sigma_{B}(\xi)\right)e_{\xi}(-s)\hat{f}(\eta)d\sigma_{B}(\eta)
\end{alignat*}

It suffices to show that
\begin{equation}
\int_{\mathbb{R}}\frac{\sin\pi(\xi-\lambda)}{\pi(\xi-\lambda)}\frac{\sin\pi(\eta-\xi)}{\pi(\eta-\xi)}e_{\xi}(-t)d\sigma_{B}(\xi)=\frac{\sin\pi(\eta-\lambda)}{\pi(\eta-\lambda)}e_{\eta}(-t).\label{eq:tmp-25}
\end{equation}
Note that 
\[
\frac{\sin\pi(\cdot-\lambda)}{\pi(\cdot-\lambda)},\;\frac{\sin\pi(\eta-\cdot)}{\pi(\eta-\cdot)}\in BL;
\]
hence the measure $d\sigma_{B}(\xi)$ on the LHS in (\ref{eq:tmp-25})
can be replaced by $d\xi$, i.e., the usual Lebesgue measure; and
the result follows from this.
\end{proof}

\subsubsection{Summary of results on $Z_{B}(t)$}

Recall the two setups:
\[
\Omega=(-\infty,0)\cup(1,\alpha)\cup(\beta,\infty)
\]
and 
\[
\tilde{\Omega}=(-\infty,\tilde{\alpha})\cup(-\frac{1}{2},\frac{1}{2})\cup(\tilde{\beta},\infty).
\]
In both cases
\[
\left(Z_{B}(t)f\right)^{\wedge}(\lambda)=\int_{\mathbb{R}}\hat{f}(\xi)K_{Shann}(\xi,\lambda)e_{\xi}(-t)d\sigma_{B}(\xi)
\]
for all $\lambda,t\in\mathbb{R}$, and all $f\in L^{2}(I_{0})$. Note
that
\[
f\in L^{2}(I_{0})\Longleftrightarrow\hat{f}\in BL=\widehat{L^{2}(I_{0})}.
\]

$\widehat{Z_{B}(t)}$ can be seen as defined by
\[
BL\ni\hat{f}\mapsto\widehat{Z_{B}(t)f}\in BL;
\]
then the transform $\widehat{Z_{B}(t)}$ is an integral operator
\[
\hat{f}\mapsto\int_{\mathbb{R}}K_{Shann}(\xi,\lambda)e_{\xi}(-t)\hat{f}(\xi)d\sigma_{B}(\xi).
\]
See eq. (\ref{eq:tmp-23}).
\begin{rem}
A few observations:\end{rem}
\begin{enumerate}
\item The Fourier basis $\{e_{n}\}_{n\in\mathbb{Z}}$ is an ONB in $L^{2}(I_{0})\subset L^{2}(\Omega)$,
but $\{e_{n}\}_{n\in\mathbb{Z}}$ do not belong to $\mathscr{D}(P_{B})$,
$0<w<1$. In deed, the generalized eigenfunction of $P_{B}$, as a
selfadjoint operator in $L^{2}(\Omega)$, are
\[
\psi_{\lambda}^{(B)}=a^{(B)}(\lambda)e_{\lambda}\chi_{I_{-}}+e_{\lambda}\chi_{I_{0}}+c_{\lambda}^{(B)}(\lambda)e_{\lambda}\chi_{I_{+}};
\]
and $a^{(B)}(\lambda)$, $c^{(B)}(\lambda)$ are NOT constants, see
Lemma \ref{lem:m} for an estimate, where
\[
\left|a^{(B)}(\lambda)\right|=\left|c^{(B)}(\lambda)\right|\in[-\frac{w}{2},\frac{w}{2}].
\]
Note $P_{0}\mathscr{D}(P_{B})\nsubseteq\mathscr{D}(P_{B})$ with $0<w<1$,
so
\[
P_{0}\psi_{\lambda}^{(B)}=e_{\lambda}\chi_{I_{0}}\in\mathscr{D}(P_{B})
\]
and so when $\lambda=n$, 
\[
e_{n}\chi_{I_{0}}\notin\mathscr{D}(P_{B}).
\]

\item $Z_{B}(t)$ acts in $L^{2}(I_{0})$, and it is zero on $L^{2}(I_{-})\oplus L^{2}(I_{+})$. \end{enumerate}
\begin{rem}
We have
\begin{equation}
\left(Z_{B}(t)\chi_{I_{0}}e_{n}\right)(x)=\chi_{0}(x)\int_{\mathbb{R}}\frac{\sin\pi(\lambda-n)}{\pi(\lambda-n)}e_{\lambda}(x-t)d\sigma_{B}(\lambda).\label{eq:tmp-26}
\end{equation}
\end{rem}
\begin{proof}
Recall that
\begin{eqnarray*}
\left(Z_{B}(t)f\right)(x) & = & \chi_{0}\left(e(-\lambda t)m^{-2}\hat{f}\right)^{\vee}\\
 & = & \chi_{0}(x)\int e(-\lambda t)e(\lambda x)\hat{f}(\lambda)d\sigma_{B}(\lambda)\\
 & = & \chi_{0}(x)\int e_{\lambda}(x-t)\hat{f}(\lambda)d\sigma_{B}(\lambda).
\end{eqnarray*}
Now, apply this to $f(x):=\chi_{I_{0}}e_{n}$. Note that $\hat{f}$
is the Shannon kernel.\end{proof}
\begin{cor}[application of (\ref{eq:tmp-26}) ]
 
\[
\left\Vert Z_{B}(t)e_{n}\chi_{0}\right\Vert ^{2}=\left|I_{0}\cap(I_{0}+t)\right|=\max(1-t,0),\; t\in\mathbb{R}_{+}.
\]
\end{cor}
\begin{rem}
For all $f\in L^{2}(I_{0})$, we define 
\[
\left(W_{0}f\right)(x):=\int\hat{f}(\lambda)e_{\lambda}(x)d\sigma_{B}(\lambda)=\left(m^{-2}\hat{f}\right)^{\vee}(x).
\]
From previous discussion, we see that
\[
\left(Z_{B}(t)f\right)(x)=P_{0}\left(W_{0}f\right)(x-t).
\]
Figure \ref{fig:ZB} below illustrates the semigroup law of $Z_{B}(t)$,
$t\geq0$.
\end{rem}
\begin{figure}[H]
\begin{tabular}{cc}
\includegraphics[scale=0.8]{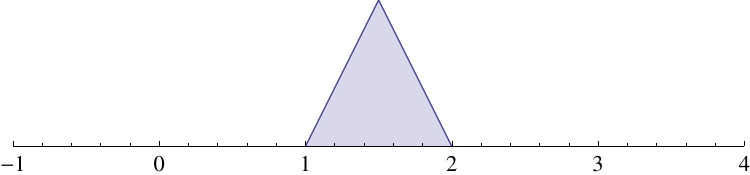} & \includegraphics[scale=0.8]{f}\tabularnewline
$f$ & $f$\tabularnewline
\includegraphics[scale=0.8]{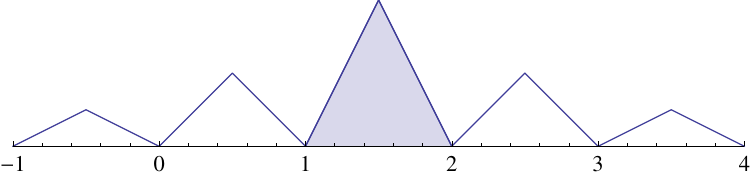} & \includegraphics[scale=0.8]{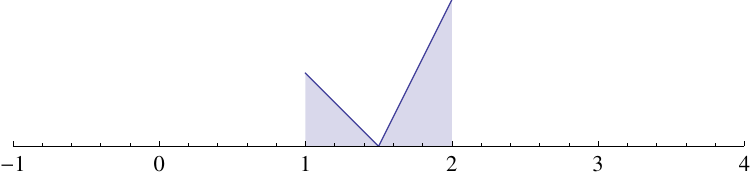}\tabularnewline
$W_{0}f$ & $Z_{B}(t)f=\chi_{0}(W_{0}f)(\cdot-t)$\tabularnewline
\includegraphics[scale=0.8]{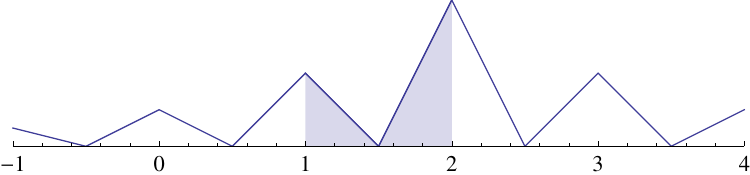} & \includegraphics[scale=0.8]{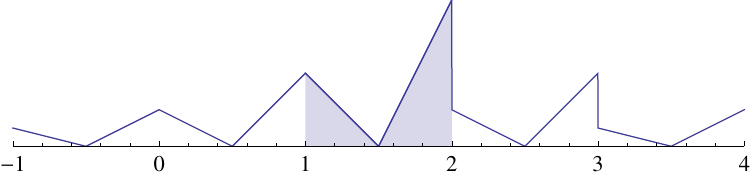}\tabularnewline
$(W_{0}f)(\cdot-t)$ & $W_{0}Z_{B}(t)f$\tabularnewline
\includegraphics[scale=0.8]{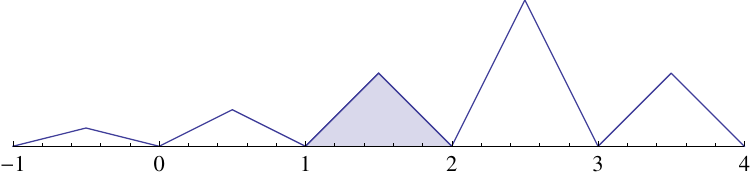} & \includegraphics[scale=0.8]{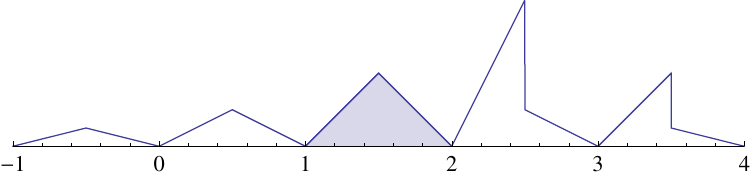}\tabularnewline
$(W_{0}f)(\cdot-s-t)$ & $(W_{0}Z_{B}(t)f)(\cdot-s)$\tabularnewline
\includegraphics[scale=0.8]{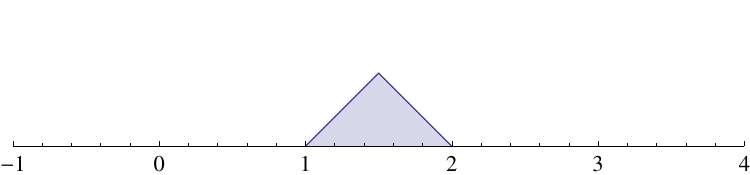} & \includegraphics[scale=0.8]{final}\tabularnewline
$Z_{B}(s+t)f$ & $Z_{B}(s)Z_{B}(t)f$\tabularnewline
\end{tabular}

\caption{\label{fig:ZB}$Z_{B}(s+t)f$}
\end{figure}

\begin{rem}
\label{rem:SB}Let $I_{0}=(1,\alpha)$. We proved in Theorem \ref{thm:ZB}
that, in the general case, the semigroup
\begin{equation}
Z_{B}(t)=P_{0}U_{B}(t)P_{0}:L^{2}(I_{0})\rightarrow L^{2}(I_{0}),\label{eq:ZB}
\end{equation}
defined for $t\in\mathbb{R}_{+}$, can be computed from the simpler
spatial semigroup $Z_{sp}(t):L^{2}(I_{0})\rightarrow L^{2}(I_{0})$
given by
\begin{equation}
\left(Z_{sp}(t)f\right)(x)=\chi_{I_{0}}(x)f(x-t),\; f\in L^{2}(I_{0}),t>0.\label{eq:ZB-2}
\end{equation}
Hence, the generator, and the resolvent operator for $Z_{B}(t)$ in
(\ref{eq:ZB-2}) may be computed from (\ref{eq:ZB-2}). One checks
that the domain of the infinitesimal generator $G_{sp}$ in (\ref{eq:ZB-2})
is 
\[
\mathscr{D}(G_{sp})=\{f\in L^{2}(I_{0});f'\in L^{2}(I_{0}),\mbox{ and }f(1)=0\}.
\]
Recall $x=1$ is the left endpoint in $I_{0}$. If $\lambda\in\mathbb{C}$,
and $\Re\lambda>0$ then the resolvent operator $R_{sp}(\lambda)=(\lambda I-G_{sp})^{-1}$
for (\ref{eq:ZB-2}) is the following Volterra integral operator (see
\cite{Ma10})
\begin{equation}
\left(R_{sp}(\lambda)f\right)(x)=\int_{1}^{x}e^{-\lambda(x-y)}f(y)dy,\label{eq:ZB-3}
\end{equation}
defined for all $f\in L^{2}(I_{0})$, and $x\in I_{0}$ ($=(1,\alpha)$.)
The Volterra property of (\ref{eq:ZB-3}) reflects causality for the
scattering we computed in section \ref{sub:OS}, see also \cite{LP68}. 
\end{rem}

\subsection{The resolvent family of $Z_{B}(t)$}

Let $\Omega=I_{-}\cup I_{0}\cup I_{+}=(-\infty,0)\cup(1,\alpha)\cup(\beta,\infty)$
be as above, i.e., $1<\alpha<\beta<\infty$ are fixed, and we set
$I_{0}=(1,\alpha)$. Fix $B\in U(2)$ such that $w>0$, and set
\begin{equation}
Z_{B}(t)=P_{0}U_{B}(t)P_{0},t\in\mathbb{R}_{+}.\label{eq:ZB-4}
\end{equation}
Here $P_{0}$ denotes the projection of $L^{2}(\Omega)$ onto the
subspace $L^{2}(I_{0})\subset L^{2}(\Omega)$, i.e., 
\begin{equation}
P_{0}f=\chi_{I_{0}}f,\label{eq:tmp-3}
\end{equation}
for all $f\in L^{2}(\Omega)$.

In this section, we shall compare the two $C_{0}$-semigroups $Z_{B}(t)$
and $Z_{sp}(t)$ from sect \ref{sub:sg} and Remark \ref{rem:SB}.
Recall, $Z_{sp}(t)$ is the spatial semigroup in $L^{2}(I_{0})$ given,
for $t\in\mathbb{R}_{+}$, by right-translation by $t$, followed
by truncation; i.e., $Z_{sp}(t)f=0$ if $t>\alpha-1=\mbox{length}(I_{0})$,
see (\ref{eq:ZB-2}). 

Both $\left(Z_{B}(t)\right)$ and $\left(Z_{sp}(t)\right)$ are $C_{0}$-semigroup
of contraction operators in $L^{2}(I_{0})$. 
\begin{lem}[\cite{LP68}]
Let $\mathscr{H}_{0}$ be a Hilbert space, and let $\{Z(t)\}_{t\in\mathbb{R}_{+}}$
be a contraction semigroup in $\mathscr{H}_{0}$. Then there is a
dense subspace $\mathscr{D}(G)$ in $\mathscr{H}_{0}$ such that,
for $f\in\mathscr{D}(G)$, the limit
\begin{equation}
\lim_{t\rightarrow0_{+}}\frac{1}{t}\left(Z(t)f-f\right)=Gf\label{eq:G}
\end{equation}
exists. The operator $G$ is called the infinitesimal generator. For
$\lambda\in\mathbb{C}$, $\Re\lambda>0$, the resolvent operator
\begin{equation}
R(\lambda)=(\lambda I-G)^{-1}:\mathscr{H}_{0}\rightarrow\mathscr{H}_{0}\label{eq:res}
\end{equation}
is an analytic family of bounded operators. We have 
\begin{equation}
\left\Vert R_{G}(\lambda)\right\Vert _{\mathscr{H}_{0}\rightarrow\mathscr{H}_{0}}\leq\frac{1}{\Re\lambda}\label{eq:res-1}
\end{equation}
for $\Re\lambda>0$; and moreover the following two limits hold in
the strong operator-topology:
\begin{equation}
R_{G}(\lambda)=\int_{0}^{\infty}e^{-t\lambda}Z(t)dt,\mbox{ and }\label{eq:res-2}
\end{equation}
\begin{equation}
\lim_{n\rightarrow\infty}\left(\frac{n}{t}R_{G}\left(\frac{n}{t}\right)\right)^{n}=Z(t),\: t\in\mathbb{R}_{+}.\label{eq:sg}
\end{equation}
\end{lem}
\begin{proof}
See \cite{LP68}. 
\end{proof}
From (\ref{eq:sg}), we see in particular that a given semigroup is
determined uniquely by its infinitesimal generator.
\begin{thm}
\label{thm:res}Now fix $B\in U(2)$ such that $w_{B}>0$, and denote
the infinitesimal generator of $Z_{B}(t)$ by $G_{B}$ i.e., for $f\in\mathscr{D}(G_{B})$,
we have 
\[
\lim_{t\rightarrow0_{+}}\frac{1}{t}\left(Z_{B}(t)f-f\right)=G_{B}f,
\]
see (\ref{eq:G}). For $\lambda\in\mathbb{C}$, $\Re\lambda>0$, set
\begin{equation}
R_{B}(\lambda):=\left(\lambda I-G_{B}\right)^{-1}.\label{eq:res-3}
\end{equation}
Finally we introduce the function $m_{B}(\cdot)$ from Lemma \ref{lem:m},
i.e.,
\begin{equation}
\mathbb{R}\ni\xi\mapsto m_{B}(\xi)\in\mathbb{R}_{+}.\label{eq:m}
\end{equation}
 Then for all $\lambda\in\mathbb{C}$, $\Re\lambda>0$, we have
\begin{equation}
R_{B}(\lambda)=R_{sp}\left(\lambda m_{B}(0)^{2}\right):L^{2}(I_{0})\rightarrow L^{2}(I_{0}),\label{eq:tmp-30}
\end{equation}
where $R_{sp}(\cdot)$ is the resolvent family from (\ref{eq:ZB-3})
in Remark \ref{rem:SB}.\end{thm}
\begin{proof}
We introduce the following notation, based on Corollary \ref{cor:scatter}.
Let $B\in U(2)$ be as above, i.e., $B=B(w,\phi,\psi,\theta)$, and
assume $w>0$. For $f\in L^{2}(I_{0})$, set 
\begin{alignat}{1}
\left(E_{B}f\right)(x) & :=\left(m_{B}^{-2}\hat{f}\right)^{\vee}(x)\label{eq:EB}\\
 & =\sum_{k\in\mathbb{Z}}a_{k}f(x+k(\alpha-1))\nonumber 
\end{alignat}
where $(a_{k})_{k\in\mathbb{Z}}$ is the set of Fourier coefficients
of $m_{B}^{-2}$, i.e., 
\begin{equation}
a_{k}=(1-w^{2})^{\frac{\left|k\right|}{2}}e(-k\psi),\: k\in\mathbb{Z},\label{eq:tmp-4}
\end{equation}
and 
\begin{equation}
m^{-2}(\xi)=\sum_{k\in\mathbb{Z}}a_{k}e_{k}((\alpha-1)\xi),\:\xi\in\mathbb{R}.\label{eq:tmp-28}
\end{equation}
For the semigroup $\left(Z_{B}(t)\right)_{t\in\mathbb{R}_{+}}$ we
proved in section \ref{sec:scatter} that
\[
\left(Z_{B}(t)f\right)(x)=\chi_{I_{0}}(x)\left(E_{B}f\right)(x-t),\ x\in I_{0},t>0.
\]
Using (\ref{eq:tmp-28}) we get
\begin{equation}
1=\sum_{k\in\mathbb{Z}}m_{B}(0)^{2}a_{k}.\label{eq:tmp-29}
\end{equation}
Using the argument for Remark \ref{rem:SB}, and (\ref{eq:tmp-29}),
the desired formula (\ref{eq:tmp-30}) follows.
\end{proof}

\section{\label{sec:points}Intervals versus Points}

There are good reasons to consider the cases when the scattering by
intervals degenerate to points. Obstacle scattering, both for the
acoustic wave equations and for quantum theory, behaves differently
in the degenerate cases. In quantum mechanics one studies what happens
at quantum scale; and wave-particle duality of matter is realized
experimentally, for example in quantum-tunneling: the phenomenon where
a particle/wave-function tunnels through a barrier (which could not
have been surmounted by a classical particle.) It is often explained
with use of the Heisenberg uncertainty principle. So quantum tunneling
is one of the defining features of quantum mechanics. Quantum differs
from classical mechanics in this way. Classical mechanics predicts
that particles that do not have enough energy to classically surmount
a barrier will not be able to reach the other side. By contrast, in
quantum mechanics, particles behave as waves and can, with positive
probability, tunnel through the barrier.

\subsection{Deleting one point}

Let $U_{t}:=e^{itP_{\theta}}$. Let $\Omega:=\mathbb{R}\setminus\{0\}.$
Let $\chi_{-}:=\chi_{(-\infty,0)}$ and $\chi_{+}:=\chi_{(0,\infty)}.$
The generalized eigenfunctions are
\[
\psi_{\xi}(x):=e(x\xi)\left(e(\theta)\chi_{-}(x)+\chi_{+}(x)\right),\xi\in\mathbb{R}.
\]
Define $V:L^{2}(\mathbb{R})\to L^{2}(\Omega),$ (of course $L^{2}(\Omega)=L^{2}(\mathbb{R})$
but the distinction is important below) by 
\begin{align*}
Vf(x) & :=\left(e(\theta)\chi_{-}(x)+\chi_{+}(x)\right)f(x)\\
 & =\int_{\mathbb{R}}\widehat{f}(\xi)e(x\xi)\left(e(\theta)\chi_{-}(x)+\chi_{+}(x)\right)d\xi\\
 & =\int_{\mathbb{R}}\widehat{f}(\xi)\psi_{\xi}(x)d\xi
\end{align*}
where we used $f(x)=\int_{\mathbb{R}}\widehat{f}(\xi)e(x\xi)d\xi.$
Since $\psi_{\xi}$ is a generalized eigenfunction 
\begin{align*}
U_{t}Vf(x) & =\int_{\mathbb{R}}\widehat{f}(\xi)e(t\xi)e(\xi x)\left(e(\theta)\chi_{-}(x)+\chi_{+}(x)\right)d\xi.
\end{align*}
Consequently, $V^{*}g(x)=\left(e(-\theta)\chi_{-}(x)+\chi_{+}(x)\right)g(x)$
and 
\[
\left(e(-\theta)\chi_{-}(x)+\chi_{+}(x)\right)\left(e(\theta)\chi_{-}(x)+\chi_{+}(x)\right)=1
\]
implies 
\begin{align*}
V^{*}U_{t}Vf(x) & =\int_{\mathbb{R}}\widehat{f}(\xi)e(t\xi)e(x\xi)d\xi=f(x+t)
\end{align*}
for all $f\in L^{2}(\mathbb{R})$ and all $x,t\in\mathbb{R}.$

\subsection{Deleting an interval}

Let $\Omega:=\mathbb{R}\setminus(0,\alpha).$ Let $\chi_{-}:=\chi_{(-\infty,0)}$
and $\chi_{+}:=\chi_{(\alpha,\infty)}.$ The generalized eigenfunctions
are
\[
\psi_{\xi}(x):=e(x\xi)\left(e(\theta)\chi_{-}(x)+e(-\xi\alpha)\chi_{+}(x)\right),\xi\in\mathbb{R}.
\]
Define $V:L^{2}(\mathbb{R})\to L^{2}(\Omega),$ (of course $L^{2}(\Omega)\varsubsetneqq L^{2}(\mathbb{R})$)
by
\begin{align*}
Vf(x) & :=e(\theta)f(x)\chi_{-}(x)+f(x-\alpha)\chi_{+}(x)\\
 & =\int_{\mathbb{R}}\widehat{f}(\xi)e(x\xi)\left(e(\theta)\chi_{-}(x)+e(-\xi\alpha)\chi_{+}(x)\right)d\xi\\
 & =\int_{\mathbb{R}}\widehat{f}(\xi)\psi_{\xi}(x)d\xi
\end{align*}
where we used $f(x)=\int_{\mathbb{R}}\widehat{f}(\xi)e(x\xi)d\xi$
so that $f(x-\alpha)=\int_{\mathbb{R}}\widehat{f}(\xi)e(-\xi\alpha)e(x\xi)d\xi.$
Since $\psi_{\xi}$ is a generalized eigenfunction 
\begin{align*}
U_{t}Vf(x) & =\int_{\mathbb{R}}\widehat{f}(\xi)e(t\xi)e(\xi x)\left(e(\theta)\chi_{-}(x)+e(\xi\alpha)\chi_{+}(x)\right)d\xi.
\end{align*}
Consequently, $V^{*}g(x)=e(-\theta)g(x)\chi_{-}(x)+g(x+\alpha)\chi_{(0,\infty)}(x)$
and $\chi_{+}(x+\alpha)=\chi_{(0,\infty)}(x)$ implies
\begin{align*}
V^{*}U_{t}Vf(x) & =V^{*}\int_{\mathbb{R}}\widehat{f}(\xi)e(t\xi)\left(e(\theta)e(\xi x)\chi_{-}(x)+e(-\xi\alpha)e(\xi x)\chi_{+}(x)\right)d\xi\\
 & =\int_{\mathbb{R}}\widehat{f}(\xi)e(t\xi)\left(e(\xi x)\chi_{-}(x)+e(-\xi\alpha)e(\xi(x+\alpha))\chi_{(0,\infty)}(x)\right)d\xi\\
 & =\int_{\mathbb{R}}\widehat{f}(\xi)e(t\xi)e(x\xi)d\xi=f(x+t)
\end{align*}
 for all $f\in L^{2}(\mathbb{R})$ and all $x,t\in\mathbb{R}.$

\subsection{Deleting two points}

Let $\Omega=(-\infty,0)\cup(0,\alpha)\cup(\alpha,\infty).$ Let $\chi_{-}:=\chi_{(-\infty,0)},$
$\chi_{0}:=\chi_{(0,\alpha)},$$\chi_{+}:=\chi_{(\alpha,\infty)}.$

Suppose $0<w<1$ and the remaining parameters are zero. The generalized
eigenfunctions are
\[
\psi_{\xi}(x):=e(x\xi)\left(a(\xi)\chi_{-}(x)+\chi_{0}(x)+c(\xi)\chi_{+}(x)\right),\xi\in\mathbb{R}.
\]
Define $V:L^{2}(\mathbb{R})\to L^{2}(\Omega),$ (of course $L^{2}(\Omega)=L^{2}(\mathbb{R})$)
by 
\[
Vf(x):=\int_{\mathbb{R}}\widehat{f}(\xi)\psi_{\xi}(x)d\xi.
\]
Since $\psi_{\xi}$ is a generalized eigenfunction 
\begin{align*}
U_{t}Vf(x) & =\int_{\mathbb{R}}\widehat{f}(\xi)e(t\xi)\psi_{\xi}(x)d\xi.
\end{align*}
Using $a(\xi)=\overline{c(\xi)}=\frac{1}{w}\left(1-\sqrt{1-w^{2}}e(\xi\alpha)\right),$
we get 
\begin{align*}
Vf(x) & =\int_{\mathbb{R}}\widehat{f}(\xi)e(x\xi)\left(a(\xi)\chi_{-}(x)+\chi_{0}(x)+c(\xi)\chi_{+}(x)\right)d\xi\\
 & =\left(\frac{1}{w}f(x)-\frac{\sqrt{1-w^{2}}}{w}f(x+\alpha)\right)\chi_{-}(x)\\
 & \quad+f(x)\chi_{0}(x)+\left(\frac{1}{w}f(x)-\frac{\sqrt{1-w^{2}}}{w}f(x-\alpha)\right)\chi_{+}(x).
\end{align*}
Hence, for $g\in L^{2}(\Omega)$ we have 
\begin{align*}
V^{*}g(x) & =\left(\frac{1}{w}g(x)-\frac{\sqrt{1-w^{2}}}{w}g(x-\alpha)\right)\chi_{-}(x)\\
 & \quad+g(x)\chi_{0}(x)+\left(\frac{1}{w}g(x)-\frac{\sqrt{1-w^{2}}}{w}g(x+\alpha)\right)\chi_{+}(x).
\end{align*}
Trying $g=\psi_{\xi}$ we get 
\begin{align*}
V^{*}\psi_{\xi}(x) & =\left(\frac{1}{w}\psi_{\xi}(x)-\frac{\sqrt{1-w^{2}}}{w}\psi_{\xi}(x-\alpha)\right)\chi_{-}(x)\\
 & \quad+\psi_{\xi}(x)\chi_{0}(x)+\left(\frac{1}{w}\psi_{\xi}(x)-\frac{\sqrt{1-w^{2}}}{w}\psi_{\xi}(x+\alpha)\right)\chi_{+}(x)\\
 & =e(x\xi)a(\xi)\left(\frac{1}{w}-\frac{\sqrt{1-w^{2}}}{w}e(-\alpha\xi)\right)\chi_{-}(x)\\
 & \quad+e(x\xi)\chi_{0}+e(x\xi)c(\xi)\left(\frac{1}{w}-\frac{\sqrt{1-w^{2}}}{w}e(\alpha\xi)\right)\chi_{+}(x)\\
 & =e(x\xi)\left(\left|a(\xi)\right|^{2}\chi_{-}+\chi_{0}+\left|a(\xi)\right|^{2}\chi_{+}\right)(x).
\end{align*}
Consequently, 
\begin{align*}
V^{*}U_{t}Vf(x) & =V^{*}\int_{\mathbb{R}}\widehat{f}(\xi)e(t\xi)\psi_{\xi}(x)d\xi\\
 & =\int_{\mathbb{R}}\widehat{f}(\xi)e(t\xi)e(x\xi)\left(\left|a(\xi)\right|^{2}\chi_{-}+\chi_{0}+\left|c(\xi)\right|^{2}\chi_{+}\right)(x)d\xi.
\end{align*}
Where 
\[
\left|a(\xi)\right|^{2}=\left|c(\xi)\right|^{2}=\frac{2-w^{2}}{w^{2}}-\frac{2\sqrt{1-w^{2}}}{w^{2}}\cos(\alpha\xi).
\]

\begin{prop}
\label{prop:m}If $\Omega$ is the complement of two points, then
the associated Fourier multiplier $\left|a\left(\xi\right)\right|$
is positive, bounded, and bounded away from zero. Specifically,
\[
\frac{w}{2}\leq\left|a\left(\xi\right)\right|\leq\frac{2}{w}.
\]
\end{prop}
\begin{proof}
Note that
\begin{align*}
\left|a(\xi)\right| & =\left|\frac{1}{w}-\frac{\sqrt{1-w^{2}}}{w}e(\alpha\xi)\right|\geq\frac{1}{w}-\frac{\sqrt{1-w^{2}}}{w}\\
 & =\frac{1}{w}\left(1-\sqrt{1-w^{2}}\right)\geq\frac{1}{w}\left(1-\left(1-\frac{1}{2}w^{2}\right)\right)=\frac{w}{2}.
\end{align*}
On the other hand, 
\[
\left|a(\xi)\right|=\left|\frac{1}{w}-\frac{\sqrt{1-w^{2}}}{w}e(\alpha\xi)\right|\leq\frac{1}{w}+\frac{\sqrt{1-w^{2}}}{w}\leq\frac{2}{w}.
\]

\end{proof}

\section{Vanishing Cross-terms\label{sec:vc}}

We include a direct computation to show the cross-terms in (\ref{eq:tmp-18})
all vanish. 

Let $\Omega=(-\infty,0)\cup(1,\alpha)\cup(\beta,\infty)$. For all
$f\in L^{2}(\Omega)$, write 
\[
f=f_{-}+f_{0}+f_{+}
\]
where $f_{-}=\chi_{-}f$, $f_{0}=\chi_{0}f$, and $f_{+}=\chi_{+}f$.
Recall that
\[
(V_{B}f)(\lambda)=\overline{a(\lambda)}\hat{f}_{-}(\lambda)+\hat{f}_{0}(\lambda)+\overline{c(\lambda)}\hat{f}_{+}(\lambda).
\]

\begin{lem}
For all $f=f_{-}+f_{0}+f_{+}\in L^{2}(\Omega)$, we have 
\begin{alignat*}{1}
\left\langle V_{B}f_{\pm},V_{B}f_{0}\right\rangle _{L^{2}(\sigma_{B})} & =0\\
\left\langle V_{B}f_{-},V_{B}f_{+}\right\rangle _{L^{2}(\sigma_{B})} & =0.
\end{alignat*}
\end{lem}
\begin{proof}
Note that
\begin{alignat*}{1}
\left\langle V_{B}f_{-},V_{B}f_{0}\right\rangle _{L^{2}(\sigma_{B})} & =\int\overline{\overline{a(\lambda)}\hat{f}_{-}(\lambda)}\hat{f}_{0}(\lambda)d\sigma_{B}(\lambda)\\
 & =\int\overline{\hat{f}_{-}(\lambda)}\hat{f}_{0}(\lambda)a(\lambda)m^{-2}(\lambda)d\lambda\\
 & =\int\overline{\hat{f}_{-}(\lambda)}\hat{f}_{0}(\lambda)\overline{a(\lambda)}^{-1}d\lambda.
\end{alignat*}
By eq. (\ref{eq:a-inv(lambda)}), $\overline{a^{-1}(\lambda)}$ has
Fourier series expansion
\[
\overline{a^{-1}(\lambda)}=\sum_{n=-\infty}^{0}a_{n}e_{n}((\alpha-1)\lambda);
\]
Notice that $a_{n}=0$, for all $n>0$. Therefore,
\begin{alignat}{1}
\int\overline{\hat{f}_{-}(\lambda)}\hat{f}_{0}(\lambda)\overline{a(\lambda)}^{-1}d\lambda & =\sum_{n=-\infty}^{0}a_{n}\int\overline{\hat{f}_{-}(\lambda)}\hat{f}_{0}(\lambda)e_{n}((\alpha-1)\lambda)d\lambda\nonumber \\
 & =\sum_{n=-\infty}^{0}a_{n}\left(g*f_{0}\right)(n(\alpha-1))\label{eq:tmp-1-1}
\end{alignat}
where $g(x):=\overline{f_{-}(-x)}$, and so $\hat{g}(\lambda)=\overline{\hat{f}_{-}(\lambda)}$.
But 

\[
\left(g*f_{0}\right)(n(\alpha-1))=0,\forall n=-1,-2,-3\ldots;
\]
since $g\in L^{2}(0,\infty)$. Thus the RHS of (\ref{eq:tmp-1-1})
vanishes. It follows that
\[
\left\langle V_{B}f_{-},V_{B}f_{0}\right\rangle _{L^{2}(\sigma_{B})}=0.
\]

Similarly, 
\[
\left\langle V_{B}f_{+},V_{B}f_{0}\right\rangle _{L^{2}(\sigma_{B})}=\int\overline{\hat{f}_{+}(\lambda)}\hat{f}_{0}(\lambda)\overline{c(\lambda)}^{-1}d\lambda;
\]
and $\overline{c(\lambda)}^{-1}$ has Fourier series expansion (see
(\ref{eq:c-inv(lambda)}))
\[
\overline{c^{-1}(\lambda)}=\sum_{n=0}^{\infty}c_{n}e_{n}((\alpha-1)\lambda).
\]
Therefore, 
\begin{alignat}{1}
\int\overline{\hat{f}_{+}(\lambda)}\hat{f}_{0}(\lambda)\overline{c(\lambda)}^{-1}d\lambda & =\sum_{n=0}^{\infty}c_{n}\int\overline{\hat{f}_{+}(\lambda)}\hat{f}_{0}(\lambda)e_{n}((\alpha-1)\lambda)d\lambda\nonumber \\
 & =\sum_{n=0}^{\infty}c_{n}\left(h*f_{0}\right)(n(\alpha-1))\label{eq:tmp-27}
\end{alignat}
where $h(x):=\overline{f_{+}(-x)}\in L^{2}(-\beta,0)$. It follows
that (\ref{eq:tmp-27}) vanishes, and
\[
\left\langle V_{B}f_{+},V_{B}f_{0}\right\rangle _{L^{2}(\sigma_{B})}=0.
\]

Finally,
\begin{eqnarray}
\left\langle V_{B}f_{-},V_{B}f_{+}\right\rangle _{L^{2}(\sigma_{B})} & = & \int\overline{\overline{a(\lambda)}\hat{f}_{-}(\lambda)}\hat{f}_{+}(\lambda)\overline{c(\lambda)}m^{-2}(\lambda)d\lambda\nonumber \\
 & = & \int\overline{\overline{a(\lambda)}\hat{f}_{-}(\lambda)}\hat{f}_{+}(\lambda)c^{-1}(\lambda)d\lambda\nonumber \\
 & = & \sum_{n=0}^{\infty}\overline{c_{n}}\int\overline{\overline{a(\lambda)}\hat{f}_{-}(\lambda)}\hat{f}_{+}(\lambda)e_{n}(-(\alpha-1)\lambda)d\lambda\nonumber \\
 & = & \sum_{n=0}^{\infty}\overline{c_{n}}\left(k*f_{+}\right)(-n(\alpha-1));\label{eq:tmp-2-1}
\end{eqnarray}
where $k(x):=\left(\overline{a}\hat{f}_{-}\right)^{\vee}(-x)$. Notice
that $\left(\overline{a}\hat{f}_{-}\right)\in L^{2}(-\infty,\alpha)$,
and so $k\in L^{2}(-\alpha,\infty)$. It follows that (\ref{eq:tmp-2-1})
vanishes, and we have
\[
\left\langle V_{B}f_{-},V_{B}f_{+}\right\rangle _{L^{2}(\sigma_{B})}=0.
\]

\end{proof}

\section*{Acknowledgments}

The co-authors, some or all, had helpful conversations with many colleagues,
and wish to thank especially Professors Daniel Alpay, Ilwoo Cho, Dorin
Dutkay, Alex Iosevich, Paul Muhly, and Yang Wang. And going back in
time, Bent Fuglede (PJ, SP), and Robert T. Powers, Ralph S. Phillips,
Derek Robinson (PJ).

\bibliographystyle{alpha}
\bibliography{Number2}

\end{document}